\documentclass[twoside,10pt]{article}
\usepackage[paperwidth=210mm,
                paperheight=297mm,
                top=3cm,%
                textheight=237mm,
                footnotesep=2\baselineskip,%
                left=2.3cm,%
                right=2.3cm,%
                marginparwidth=1.7cm,
                headsep=30pt,
                dvips=false]{geometry}

\usepackage{empheq}
\usepackage{fancyhdr}
\makeatletter \renewcommand*{\@seccntformat}[1]{\csname the#1\endcsname.\quad} \makeatother

\usepackage[T2A]{fontenc}
\DeclareUnicodeCharacter{00A0}{~}
\usepackage{ stmaryrd }
\usepackage{booktabs}
\usepackage[english]{babel}

\usepackage{amsmath,amscd,amsfonts,amsthm,bbm,mathrsfs,enumerate,mathtools, enumitem}
\usepackage{amssymb}
\usepackage[numbers,sort&compress]{natbib}

\usepackage[dvipsnames]{xcolor}
\usepackage[colorlinks=true, linkcolor=blue, citecolor=blue]{hyperref}
\usepackage{graphicx}
\usepackage{tikz,pgfplots}
\usepackage{svg}
\usepackage{caption}
\usepackage{subcaption}
\usepackage{float}
\usepackage{csquotes}

\usepackage[ruled, linesnumbered]{algorithm2e}
\usepackage{algorithmic}
\usepackage{dsfont}
\usepackage{xspace}

\DeclareMathOperator{\rank}{rank}
\DeclareMathOperator{\truncdist}{\overline{\dist}}
\newcommand{\sqbracket}[2]{\llbracket #1, #2 \rrbracket}

\usepackage{shortcuts}

\newtheorem{theorem}{Theorem}
\newtheorem{lemma}{Lemma}[section]
\newtheorem{proposition}[lemma]{Proposition}
\newtheorem{definition}[lemma]{Definition}
\newtheorem{example}[lemma]{Example}
\newtheorem{remark}[lemma]{Remark}
\newtheorem{assumption}{Assumption}
\newtheorem{corollary}[lemma]{Corollary}

\providecommand{\keywords}[1]{\textbf{\textbf{Keywords. }} #1}

\usepackage[most]{tcolorbox}
\newtcolorbox{nbox}[1][]{
  enhanced,
  fonttitle=\scshape,
  #1
}

\DeclareMathOperator{\Jac}{Jac}

\renewcommand{\c}{\mathsf{c}}
\newcommand{\tc}{\tilde{\c}}
\newcommand{\ttB}{\mathtt{B}}
\usepackage[normalem]{ulem}
\newcommand{\replace}[2]{
\textcolor{red}{#2}}

\begin{document}
\title{On convergence rates of subgradient descent on semialgebraic functions}
\author{Evgenii Chzhen\\Laboratoire de mathématiques d’Orsay\\Université Paris-Saclay, CNRS\\91405, Orsay, France\\\texttt{first.last@cnrs.fr}
\\\and
Sholom Schechtman\\SAMOVAR, Télécom SudParis\\Institut Polytechnique de Paris\\91120, Palaiseau, France\\\texttt{first.last@telecom-sudparis.eu}}
% \date{\phantom{} \\ \ \\ \rule{\linewidth}{.5pt}{}}
\date{}
\pagestyle{fancy}
\renewcommand{\footrulewidth}{0.4pt}
\maketitle
\thispagestyle{empty}
% \maketitle

\begin{abstract}

We analyze the constant step size subgradient method {on} nonsmooth, nonconvex functions. We identify geometric assumptions on the objective function under which \emph{i)} its domain admits a partition (stratification) into smooth manifolds (strata) on which the function is smooth; \emph{ii)} a global projection formula for Clarke subgradients holds; and \emph{iii)} quantitative curvature bounds hold on each stratum.
Under these conditions, we prove that the iterates of the subgradient method locally shadow a Riemannian gradient descent on nearby strata, which we use to measure stationarity.
We introduce a selection rule for the active stratum and develop a mechanism that assembles local descent inequalities across successive strata into explicit convergence rates.
These rates are expressed in terms of the number of dimensions present in the stratification, improve as the number of strata decreases, and recover, up to constants, the classical rates in the smooth case.
We show that the stated assumptions follow from the existence of Lipschitz stratifications of semialgebraic sets, and are therefore automatically satisfied for semialgebraic functions and, more generally, for functions definable in polynomially bounded o-minimal structures---yielding the first known convergence rates in these settings. As intermediate results of independent interest, we establish tubular neighborhood estimates for Lipschitz stratifications and a global projection formula for Clarke subgradients. Finally, we show that our framework extends to decreasing step size and recovers, via an alternative argument, the {recently announced} result of Lai and Song on sequential convergence of the subgradient method with step sizes $1/k$.

\end{abstract}
\begin{figure}[h!]
    \centering
    \includegraphics[width=\linewidth]{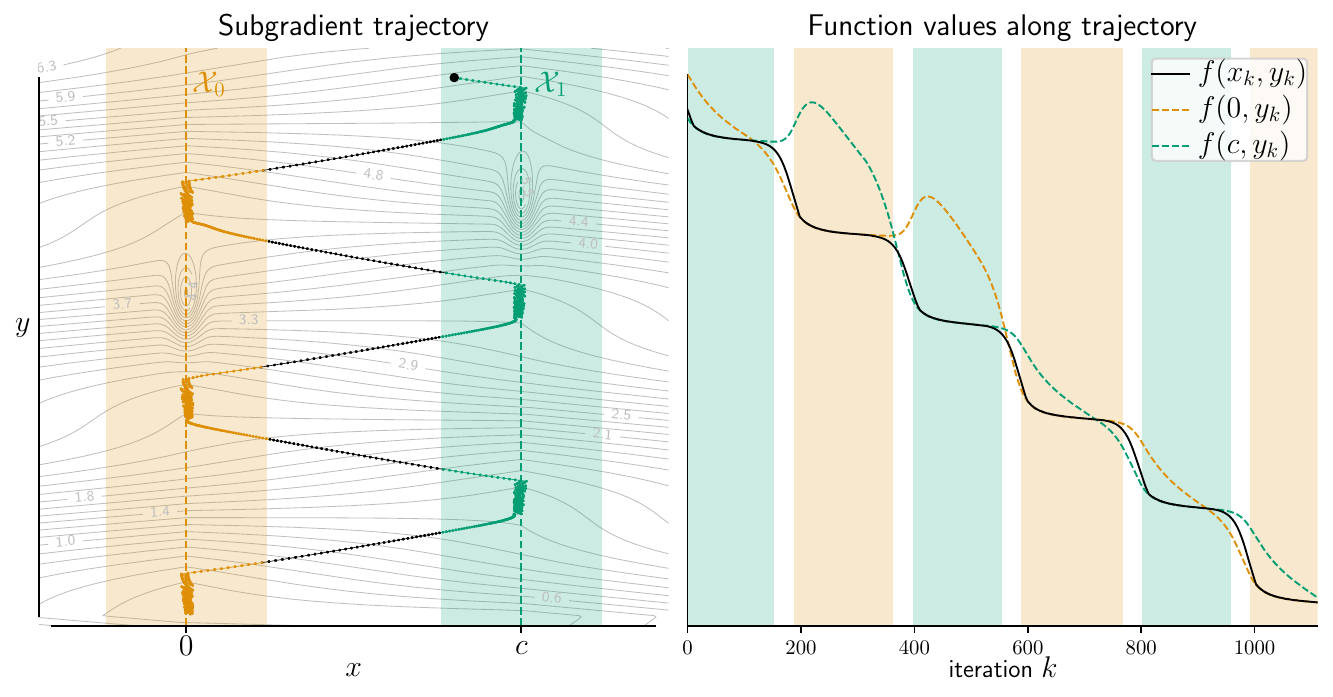}
\end{figure}

\newpage
\tableofcontents
\newpage

\section{Introduction}
The purpose of this work is to analyze the subgradient method \cite{shor1964structure} applied to semialgebraic functions. Given such function $f: \bbR^d \rightarrow \bbR$, possibly nonconvex and nonsmooth, the iterates $(x_k)_{k \geq 1}$ of the subgradient method are defined recursively by
\begin{equation}\label{eqdef:subg_des_intro}
    x_{k+1} = x_k - \gamma v_k \, ,
\end{equation}
where $\gamma > 0$ is the step-size and $v_k \in \partial f(x_k)$ for all $k$.

Despite the lack of convexity and smoothness, substantial empirical evidence
supports the effectiveness of (stochastic) subgradient methods. A prominent example arises in deep learning, where widely used architectures incorporate nonsmooth operations such as ReLU activations, max-pooling, and normalization layers. In these settings, subgradient-based methods are routinely observed to converge (see, e.g., \cite{zohrevand2022empirical}).

From a theoretical standpoint, however, this empirical success is not yet fully understood, as existing guarantees are largely asymptotic.
For sufficiently small step-sizes, the iterates are known to converge to a neighborhood of Clarke-critical points, without explicit bounds on either the size of this neighborhood or the rate of convergence
\cite{dav-dru-kak-lee-19,bianchi2022convergence,benaim2005stochastic,ermol1998stochastic}.

Obtaining convergence rates for subgradient-based algorithms first requires clarifying the notion of optimality with respect to which such rates are measured.
Closely related to this choice is the mechanism used to establish convergence:
typically, the analysis relies on identifying a Lyapunov function, with the amount of decrease serving as a meaningful certificate of optimality.
For instance, in the smooth case, the Lyapunov function is simply the objective,
and convergence rates are expressed in terms of the norm of its gradient.

In the nonsmooth setting, however, the objective need not decrease along the
iterations, and one must instead identify an appropriate surrogate. A notable
example is the work of Davis and Drusvyatskiy \cite{davis2019stochastic}, {which}
showed that subgradient descent applied to a weakly convex objective yields a
decrease in its Moreau envelope, with convergence rates expressed in terms of
the gradient of this regularization. A different line of work approaches the
minimization of nonsmooth and nonconvex objectives via the Goldstein
subdifferential
\cite{goldstein1977optimization,zhang2020complexity,kong2025lipschitz,
jordan2022complexity,davis2022gradient}. Roughly speaking, the key observation
is that the objective decreases when moving in the direction of a minimum-norm
Goldstein subgradient, which is then used as a convergence certificate.
Finally, in convex settings, optimality can also be measured in terms of the
distance to the solution set.

However, two important features are not captured by the aforementioned complexity analyses.
First, nonsmoothness in optimization problems often hides an underlying
\emph{smooth} structure. For instance, the function
$f(x,y) = (|x|-|y|)^2$, while nonsmooth, is smooth when restricted to the sets
$\cX = \bbR \times \{0\}$ and $\cY = \{0\} \times \bbR$, with
$f_{|\cX} = x^2$ and $f_{|\cY} = y^2$. This phenomenon is by no means specific to this example.
Indeed, it is well known that the domain of any semialgebraic function $f$ admits a partition, or more precisely a \emph{stratification}, into smooth manifolds (or \emph{strata}) $(\cX_i)$ such that the restriction $f_{|\cX}$ is smooth on each stratum $\cX$.
Moreover, for any $x \in \cX$, the Clarke subdifferential $\partial f(x)$ satisfies the projection formula: its projection onto the tangent space of $\cX$ coincides with $\nabla_{\cX}f(x)$, the Riemannian gradient of the smooth restriction $f_{|\cX}$
\cite{bolte2007clarke,bolte2021conservative}.

Second, recent analyses {suggests} that, under additional conditions on the stratification, subgradient methods implicitly exploit this structure. For example, the works \cite{bianchi2024stochastic,davis2025active} show that a Kuo--Verdier condition implies that, around any point of a stratum $\cX$, the projections of the iterates onto $\cX$ closely track a Riemannian gradient descent applied to the smooth function $f_{|\cX}$.

These observations suggest a natural route toward a complexity analysis. If the projected iterates approximately follow a Riemannian gradient descent on $f_{|\cX}$, then this function provides a natural Lyapunov candidate, with decrease certified by $\norm{\nabla_{\cX} f(x)}$.
However, a direct approach encounters two main difficulties. First, under existing conditions \cite{davis2025active,bianchi2024stochastic} (with the notable exception of \cite{lai2025diameter}), such shadowing holds only locally, in a neighbourhood of a fixed point on the stratum. Thus, even if the iterates remain close to $\cX$, it is not immediate that $f_{|\cX}$ decreases along the projected trajectory. Second, even if such a decrease can be established, the iterates may eventually approach a different stratum $\cX'$, at which point the relevant Lyapunov function becomes $f_{|\cX'}$. Figure~\ref{fig:switching_traj} illustrates this phenomenon. The subgradient descent trajectory clearly reflects the stratification, which here consists of two vertical lines. The plot on the right of the figure further shows that the active stratum $\cX$ is generally not unique along the trajectory, yet at each step there exists some $\cX$ along which $f_{|\cX}$ decreases.

In this work, we develop a framework that resolves these two difficulties and yields the first finite-time complexity guarantees for the subgradient method in this setting. Our main contributions are threefold.

\begin{itemize}
    \item First, we identify a set of assumptions on the stratification that
    \emph{i)} yields a \emph{global} projection formula for the Clarke
    subdifferential of $f$, and \emph{ii)} provides uniform bounds on the
    curvature of the strata $\cX_i$. The latter allows us to estimate the
    width of tubular neighborhoods around the strata on which the projection
    is well-defined.

    \item Second, under these assumptions, we show that the domain of $f$ can be
    partitioned into sets $\cC(\cX)$, corresponding to suitable neighborhoods
    of the strata, such that on each $\cC(\cX)$ the projected iterates shadow
    a Riemannian gradient descent on $f_{|\cX}$.

    \item
    Third, we develop a mechanism that tracks the successive strata visited by the trajectory and patches the resulting local descent inequalities.
    The resulting convergence rate is then obtained via a control of three effects:
     i) smooth (Riemannian) descent along strata; ii) controlled strata switching term; iii) a geometric error quantifying vicinity to a stratum and its boundary. Careful control of the three terms yields the main non-asymptotic result in Theorem~\ref{thm:main_announced}---first convergence rates for the subgradient method. These rates can be expressed in terms of the norm of Riemannian gradients on nearby strata, or, as we discuss, through the approximate stationarity measure of a conservative set-valued field associated to the stratification.
\end{itemize}
\begin{figure}
      \centering
    \includegraphics[width=\linewidth]{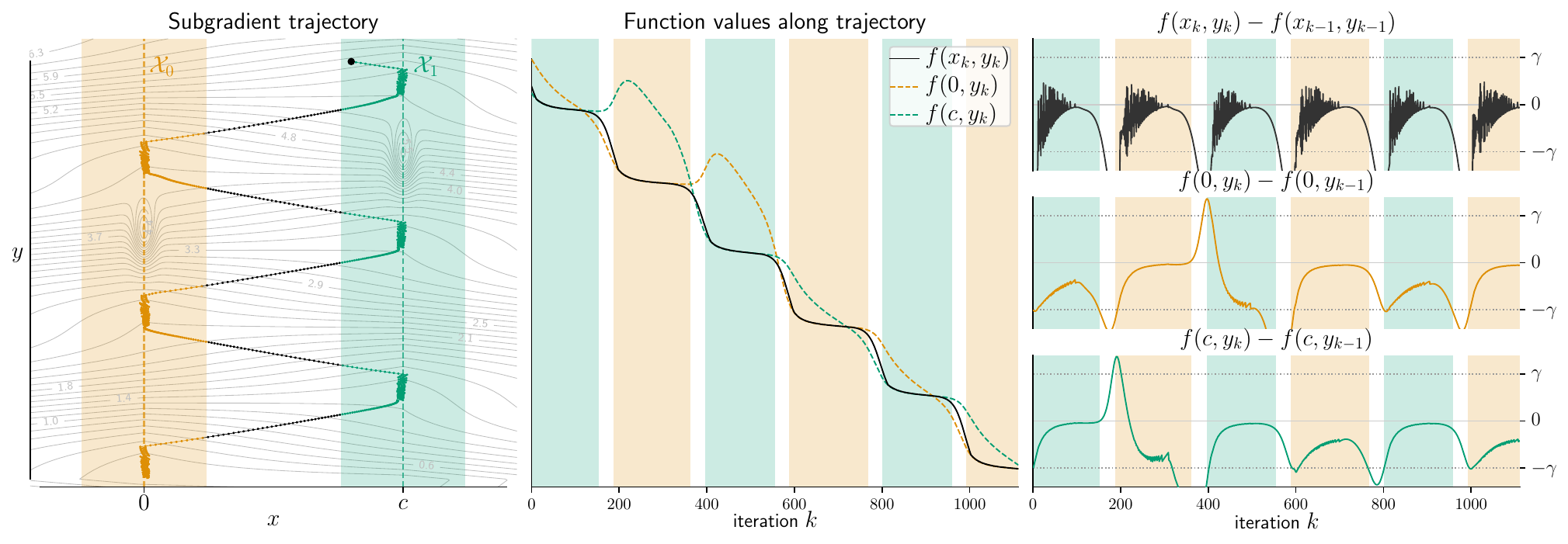}
    \caption{Gradient descent trajectory. Two linear strata: $\cX_0 = \{(x, y) \,:\, x = 0\}$, $\cX_1 = \{(x, y) \,:\, x = c\}$. Exact expression of the function is available in Appendix~\ref{app:switching_traj}. }
    \label{fig:switching_traj}
\end{figure}

Naturally, one may ask whether such an assumption is satisfied by classes of functions relevant to optimization. In Theorem~\ref{thm:good_assumption}, we show that any semialgebraic function, and more generally any function definable in a polynomially bounded o-minimal structure, when restricted to a compact semialgebraic set, satisfies this assumption. This result follows directly from the existence of Lipschitz stratifications of compact semialgebraic sets.

Here, we have to mention the recent work of Lai and Song \cite{lai2025diameter},
which partially inspired the present paper.
Their work is concerned with the
sequential convergence of the subgradient method and is, to our knowledge,
the first to highlight the significance of Lipschitz stratifications in this
context. Their analysis also relies on Lipschitz stratifications, and on the idea of switching between active strata, but uses
{\L}ojasiewicz-type inequalities to obtain the quantitative estimates needed
for the projection formula, curvature bounds, and control of tubular
neighborhoods. By contrast, our approach establishes these estimates directly from the stratification itself, independently of the existence of {\L}ojasiewicz exponents. This also applies to our strata selection mechanism, leading to a finite-time complexity analysis. Nevertheless, in Section~\ref{sec:KL_conv}, we show that the machinery
developed in the previous sections also allows us to recover the following
main result of Lai and Song by an alternative proof:

\begin{itemize}
    \item[] If the iterates $(x_k)$ of a subgradient method with decreasing
    step-sizes $\gamma_k = 1/k$ remain in a compact set, then $(x_k)$
    converges to a unique point.
\end{itemize}

More broadly, both works suggest that Lipschitz stratifications provide a fruitful geometric framework for the analysis of nonsmooth optimization methods.

In the remainder of this section, we fix the notation used throughout the paper. In Section~\ref{sec:assumptions_main}, we present our main results, Theorems~\ref{thm:good_assumption} and~\ref{thm:main_announced}. Section~\ref{sec:proof_main} is devoted to the proof of Theorem~\ref{thm:main_announced}, while Section~\ref{sec:extensions} presents several extensions of our framework. Finally, Section~\ref{sec:rem_proofs} contains the remaining proofs, including that of Theorem~\ref{thm:good_assumption}.

\subsection{Notation}
\label{sec:notation}
Appendix~\ref{app:prelim} provides a more detailed exposition on the required background. Here, we only fix notation and assume the reader is familiar with the relevant concepts.
For a $C^1$ function $G: \bbR^l \rightarrow \bbR^{k}$ and $x \in \bbR^l$ we denote $\Jac G(x) \in \bbR^{k \times l}$ its Jacobian and $\dif G(x): \bbR^l \rightarrow \bbR^k$ its differential at $x\in \bbR^l$, that is, $\dif G(x)[h] = \Jac G(x)h$. We say that $D: \bbR^l \rightrightarrows \bbR^k$ is a set-valued map if for all $x \in \bbR^l$, $D(x) \subset \bbR^k$. For a set $A \subset \bbR^d$ we denote by $\partial A = \overline{A} \backslash A$ the topological frontier of $A$. For non-empty $A \subset \bbR^d$ and $x \in \bbR^d$ define
\begin{align}\label{eqdef:distances}
    \dist(x, A) \eqdef \inf_{y \in A} \|x - y\| \quad \textrm{ and } \quad \truncdist(x, A) \eqdef \min(\dist(x, A),1)\, ,
\end{align}
and, whenever well-defined, $\pi_{A}(x) = \argmin_{y \in A} \norm{x-y}$.
For a finite set $A$, we denote $|A|$ its cardinal. For two integers $k_1, k_2$, we denote by $\sqbracket{k_1}{k_2}$ and $\rrbracket k_1,k_2\llbracket$ the consecutive integers between $k_1$, $k_2$ included and excluded respective. Similarly, for a sequence $(x_k)_{k \geq 1}$ and $A \subset \bbN$, we denote by $x_{A}$, the sub-sequence sequence $(x_k)_{k \in A}$.
The following conventions are used throughout. We fix $p \geq 2$. For $k > m$: $\sum_{i = k}^m = 0$, $\cup_{i = k}^m = \emptyset$, $\dist(x, \emptyset) = 1$.

This work uses a fair amount of notation and introduces many constants, for reader's convenience, in Tables~\ref{tab:not_strat}-\ref{tab:not_strata_select} we gathered notation related to stratifications, constants, and the strata selection mechanism.
\subsubsection{Semialgebraic sets and functions} A set $A \subset \bbR^d$ is semialgebraic if it can be written as a finite union and intersection of sets of the form $\{x: Q(x) \leq 0 \}$, where $Q: \bbR^d \rightarrow \bbR$ is some polynomial. A function $F: \bbR^d \rightarrow \bbR^m$ is said to be semialgebraic if its graph is semialgebraic (as a subset of $\bbR^{d + m}$). We emphasize the fact that all our results hold more generally upon replacing semialgebraic by \emph{definable in a polynomially bounded o-minimals structure} (see Appendix~\ref{app:o-minimal}).
\subsubsection{Submanifolds}
\label{sec:submanif_notation}
The following notation will be used throughout the paper. For $\cX \subset \bbR^d$ a $C^p$ submanifold and $x \in \cX$, the tangent (respectively normal) plane is denoted as $\cT_{x} \cX$ (respectively $\cN_x \cX$). For a $C^p$ function $f: \cX \rightarrow \bbR$, we denote $\nabla_{\cX}f$ the Riemannian gradient induced from the Euclidean metric. We denote by $P_{x}: \bbR^d \rightarrow \cT_{x} \cX$ the orthogonal projection onto $\cT_{x} \cX$. We denote by $\dist_{\cX}$ the inner metric of $\cX$, defined as:
\begin{equation*}
  \dist_{\cX}(x_0,x_1) = \inf\left\{ \int_{0}^1 \norm{\dot{\c}_t} \dif t : \textrm{ where $\c:[0,1] \rightarrow \cX$ is $C^p$, $\c_0 =x_0$ and $\c_1 =x_1$}\right\}\, .
\end{equation*}
Note that $ \norm{x_0 - x_1} \leq \dist_{\cX}(x_0, x_1)$. If, moreover, there is $C>0$ such that $\sup_{x_0, x_1 \in \cX}\tfrac{\dist_{\cX}(x_0, x_1)}{\norm{x_0 - x_1}} \leq C$, then we say that the inner metric of $\cX$ is equivalent to its outer metric.

\subsubsection{Stratifications}
\label{sec:strat_notation}

A stratification of a set is a finite partition of the set into submanifolds satisfying some regularity properties.
Any semialgebraic set can be stratified (see Appendix~\ref{app:o-minimal}).

\begin{definition}[Stratification]\label{def:stratif_main_new}
  Let $\cK \subset \bbR^d$, we say that $\bbX = (\cX_i)$, a finite partition of $\cK$, is a $C^p$ stratification of $\cK$, if the following conditions hold.
  \begin{enumerate}
    \item  Every $\cX \in \bbX$ is a path-connected $C^p$ manifold.
    \item Boundary condition: for $\cX, \cX' \in \bbX$
    \begin{align*}
        \cX \cap \overline{\cX'} \neq \emptyset \quad\implies\quad \cX \subset \partial \cX'\,.
    \end{align*}
  \end{enumerate}
\end{definition}
Following~\cite{parusinski1994lipschitz}, let $\sX_0 \subset \sX_1 \subset \ldots \subset \sX_d$ be filtration of $\bbX$ induced by skeletons $(\sX_j)_{j = 0}^d$, which are defined for $j \in \sqbracket{0}{d}$ as
\begin{align}
  \label{eqdef:closed_skeleton}
  \sX_j \eqdef \bigcup_{\substack{\cX \in \bbX \\ \dim(\cX) \leq j}} \cX\,,
\end{align}
 the convention that $\sX_{-1} = \emptyset$.

\subsubsection{Clarke subgradients and projection formula} Let $f: \bbR^{d} \rightarrow \bbR$ be a locally Lipschitz function.
 The Clarke subgradient \cite{cla-led-ste-wol-livre98} of $f$ at $x$ is defined as
  \begin{equation*}
    \partial f(x) := \conv \{ v \in \bbR^d: \textrm{ there is $x_n \rightarrow x$, with $f$ differentiable at $x_n$ and $\nabla f(x_n) \rightarrow v$}\}\, .
  \end{equation*}

As established in the seminal work of Bolte, Daniilidis, Lewis, and Shiota \cite{bolte2007clarke}, the Clarke subgradient of a semialgebraic function admits a transparent geometric structure.
\begin{proposition}[{Projection formula \cite{bolte2007clarke}}]\label{prop:bolte_projform_new}
    Let $f: \bbR^d \rightarrow \bbR$ be a semialgebraic, locally Lipschitz function and $p \geq 2$. There is $\bbX = (\cX_i)$ a stratification of $\bbR^d$ such that \emph{i)} any $\cX \in\bbX$ is semialgebraic, \emph{ii)} for any $\cX \in\bbX$, $f_{|\cX}$ is $C^p$ and, \emph{iii)} for any $x \in \cX$ and any $v \in \partial f(x)$,
    \begin{equation}\label{eq:projformula_bolte_new}
        P_{x} v = \nabla_{\cX} f(x)\, ,
    \end{equation}
    where $P_{x}$ denotes the orthogonal projection onto $\cT_{x} \cX$.
\end{proposition}
{Set-valued maps $D: \bbR^d \rightrightarrows \bbR^d$ elements of which satisfy~\eqref{eq:projformula_bolte_new} are called conservative set-valued fields \cite{bolte2021conservative} (see Appendix~\ref{sec:cons_fields}).}
We emphasize that all of our results hold more generally,
when $v_k \in \partial f(x_k)$, is replaced by $v_k \in D(x_k)$, for $D$ some bounded conservative set-valued field of $f$.

\section{Assumptions and summary of main result}
\label{sec:assumptions_main}
In this section and throughout the rest of the paper, we fix a function $f: \bbR^d \rightarrow \bbR$, a set
$\cK \subset \bbR^d$, and a $C^p$ stratification $\bbX$ (see
Definition~\ref{def:stratif_main_new}) of $\cK$. We assume that $\bbX$
satisfies certain regularity properties, both on its own and in relation to
the function $f$.

\begin{assumption}\label{ass:f_lip}
    There exists $G>0$ such that $f$ is $G$-Lipschitz on $\cK$.
\end{assumption}
The following assumption ensures that the function is piecewise smooth relative
to the stratification and strengthens the projection formula of Bolte et al.
(see Proposition~\ref{prop:bolte_projform_new}) in two directions. First, the
projection formula now holds not only for points lying on a stratum $\cX$, but
\emph{globally}, with an error controlled by the distance to the stratum and
its boundary. Second, a curvature bound is imposed on every stratum by
controlling the variation of its tangent planes.

\begin{assumption}
   \label{ass:main}
There exists a $C^p$ stratification $\bbX$ of $\cK$ such that:
   \begin{enumerate}[label=(\roman*)]
        \item\label{ass:main3} For every $\cX \in \bbX$, the restriction $f_{|\cX}$ is $C^p$.
  \item\label{ass:metric} For every $\cX \in \bbX$ its inner metric is equivalent to the outer metric.
        \item\label{ass:main2} There is $L_1 > 0$, so that for any $\cX \in \bbX$ and $x, x' \in \cX \cap \cK$
            \begin{align}\label{eqdef:Lip_tangents}
                \|P_x - P_{x'}\| \leq L_1\cdot \frac{\|x - x'\|}{\truncdist(x, \sX_{\dim(\cX)-1})}\,.
            \end{align}
        \item\label{ass:main4}
        There is $L_2 > 0$ such that for any $\cX \in \bbX$, any $x \in \cK$,
        any $y \in \cX \cap \cK$, and any $v \in \partial f(x)$
        \begin{align}\label{eqdef:proj_form_lip}
            \|P_yv - \nabla_{\cX} f(y)\| \leq L_2 \cdot\frac{\|x - y\|}{\truncdist(y, \sX_{\dim(\cX)-1})}\,.
        \end{align}
   \end{enumerate}
\end{assumption}
Let us comment on the assumptions.
Assumption~\ref{ass:main}-\ref{ass:main3} corresponds to the fact that $f$ is piecewise smooth. This is a standard assumption that ensures convergence of the (stochastic) subgradient method (see \cite{dav-dru-kak-lee-19,bolte2021conservative,benaim2005stochastic}).
Noting that the function $(E_1, E_2) \mapsto \norm{P_{E_1} - P_{E_2}}$ is a distance function on $j$-dimensional vector spaces of $\bbR^d$, Assumption~\ref{ass:main}--\ref{ass:main2} ensures a Lipschitz control on the variations of the tangent planes for a given stratum $\cX$.
As we show below in Lemma~\ref{lem:projection} this guarantees a bound on the curvature of $\cX$ and, in combination with Assumption~~\ref{ass:main}-\ref{ass:metric}, ensures that the projection is well-defined on a conical neighborhood of $\cX$.

Similarly, Assumption~\ref{ass:main}-\ref{ass:main4} provides a Lipschitz control on the behavior of Clarke subgradients projected along any stratum. It is a strengthening of the projection formula of Proposition~\ref{prop:bolte_projform_new} (indeed taking $x=y$ in~\eqref{eqdef:proj_form_lip} we obtain~\eqref{eq:projformula_bolte_new}) that describes the behavior of the subgradients not only \emph{on} a stratum but nearby. Comparing to the previously known Kuo-Verdier assumptions \cite{davis2025active,bianchi2024stochastic} such control holds not only on a neighborhood of a given point $y$ but \emph{globally}, although with a denominator that blows up when boundary is approached. Additionally, note that for $y, y' \in \cX$, and any $v \in\partial f(y)$ and $v \in \partial f(y')$, we obtain
\begin{equation*}
    \norm{\nabla_{\cX} f(y') - \nabla_{\cX} f(y)} \leq \norm{(P_{y'} -P_y)v'} + \norm{P_yv' - \nabla_{\cX}f(y) } \leq (L_1 + L_2)\frac{\|y-y'\|}{\truncdist(y, \sX_{j-1})}\, ,
\end{equation*}
Thus, providing a Lipschitz control on the Riemannian gradients of $f_{|\cX}$ along any stratum $\cX$.

In summary, given a stratification $\bbX$ such that $f$ is smooth when restricted to each stratum, Assumption~\ref{ass:main} provides control over three key aspects:
\begin{enumerate}
    \item the Lipschitz continuity of the Riemannian gradients of $f_{|\cX}$;
    \item the variation of the tangent spaces of the strata $\cX$;
    \item a \emph{global} projection formula relating subgradients of $f$ to the geometry of the stratification.
\end{enumerate}

\subsection{Examples}
In this section we provide some examples of functions, which satisfy the required assumptions.
First, let us give two basic examples, which highlight the necessity of division by the distance to the boundary.
\begin{example}
    For $\beta >0$, consider the function $f: \bbR \rightarrow \bbR$, defined as $f(x) = |x|^{1+\beta}$. It satisfies Assumption~\ref{ass:main} with $\cX_0 =  \{ 0\}$, $\cX_1  =\bbR_{>0}$ and $\cX_2 = \bbR_{<0}$. We note that while $f$ has locally Lipschitz gradient on $\cX_1$, for $x,x'>0$ close to zero we obtain
    \begin{equation*}
        |f'(x) - f'(x')| = (1+\beta)|x^{\beta} - x'^{\beta}| \leq (1+\beta)\frac{|x - x'|}{x}
    \end{equation*}
\end{example}
\begin{example}
    For any $d \geq 2$, the Hessian of $x \mapsto \norm{x}$ is unbounded. Consequently, there is no stratification of $\bbR^d$ for which the gradient of $f$ is Lipschitz on a full-dimensional stratum. The convexity of $x \mapsto \norm{x}$ is not essential here and can be removed by suitable modifications away from the origin.
\end{example}

Finally, the main motivation for the introduced assumptions is the fact that if $f$ and $\cK$ are semialgebraic and $\cK$ is compact, then there is automatically a stratification of $\cK$ such that our assumptions hold. This is a consequence of existence of Lipschitz stratifications for semialgebraic sets.

\begin{theorem}\label{thm:good_assumption}
    Let $f: \bbR^d \rightarrow \bbR$ be locally Lipschitz continuous. Assume that both $f$ and $\cK \subset \bbR^d$ compact. Then, there is a stratification $\bbX$ of $\cK$ such that Assumption~\ref{ass:main} holds and every $\cX \in \bbX$ is semialgebraic.
\end{theorem}
The proof, given in Section~\ref{subsec:lips_strat}, actually establishes a more general statement: $f$ and $\cK$ can be taken definable in an arbitrary polynomially bounded structure, $\partial f$ can be replaced by a definable conservative set-valued field of $f$. Since, as discussed in Appendix~\ref{app:deep_L} most of the deep learning architectures are (restricted to a compact) polynomially bounded \cite{telgarsky2016benefits,bareilles2025deep}, the result is appliable to these cases.
\subsection{Main results}

Before going further let us introduce the notation for a collection of strata of given dimension $j \leq d$:
\begin{align}
    \bbX_j :=\{ \cX \in \bbX\,:\, \dim(\cX) = j\}\,.
\end{align}

    Our non-asymptotic results will depend on a structural parameter of the stratification, {which we name the \emph{rank}}, that naturally appears in the analysis. It corresponds to the number of non-empty sets in $(\bbX_j)_{j=0}^d$, which we call rank of $\bbX$ and introduce below.
    \begin{definition}
        \label{def:rank_stratum}
        Let $\bbX$ be a $C^p$ stratification of $\cK$. Define the set of active dimensions
        \begin{align*}
            \cJ := \{ j \in \{0,\ldots,d\} \,:\, \bbX_j \neq \emptyset \}\,.
        \end{align*}
        Let $(j_r)_{r=0}^R$ be the increasing enumeration of $\cJ$, where $R = |\cJ|-1$.
        For $\cX \in \cK$, define its rank as the unique $r$ such that $\cX \in \bbX_{j_r}$, \emph{i.e.,} $\rank(\cX) = r \, \Longleftrightarrow \, \cX \in \bbX_{j_r}$.
    \end{definition}
    The obtained rates depend on the number of active strata levels $R$, which interpolates between the smooth case\footnote{Although our main interest lies in the case $R \geq 1$, the results can be adapted in a straightforward manner to the case $R = 0$, recovering—up to constants—the standard convergence rates for smooth functions. We do not pursue this refinement and focus on the regime $R \geq 1$.} $(R=0)$ and the fully stratified case $(R=d)$.
    For a first reading, it is helpful to consider the case $R = d$ and $j_r = r$, corresponding to a stratification that contains strata of all dimensions.

Finally, to state our main result, for a given stratum $\cX$, let us {set}
\begin{equation*}
    g_{\cX} \eqdef f\circ {\pi}_{\cX}\,,
\end{equation*}
whenever it is well-defined.
This function is always well-defined on an appropriate neighborhood of $\cX$ (which will be characterized in Lemma~\ref{lem:projection} below) and in the following theorems $g_{\cX}(x_k)$ is evaluated \emph{only} if $x_k$ lies in such neighborhood. We now have all the tools to present our complexity guarantees.
\begin{theorem}
    \label{thm:main_announced}
    Let $f$ and $\bbX$ satisfy Assumptions~\ref{ass:f_lip} and~\ref{ass:main} and $R \geq 1$ be the rank of $\bbX$. Then, for any $\alpha < \beta < 1/(R+1)$, there is $\gamma_0>0$, $C>0$, such that for all $K \geq 1$, if $\gamma < \gamma_0$ and $x_k \in \cK$ for all $k \leq K$, then there is a sequence of strata $(\Psi_k)_{k=1}^K$, with $\Psi_k \in \bbX$, for which
        \begin{equation*}
        \dist(x_k, \Psi_k) \leq \gamma^{\alpha + \rank(\Psi_k) \beta}\,
    \end{equation*}
    and
         \begin{align*}
        \frac{1}{K}\sum_{k = 1}^K\|\nabla g_{\Psi_k}(x_k)\|^2 \leq C\bigg(\underbrace{\frac{g_{\Psi_1}(x_1) - g_{\Psi_K}(x_{K+1})}{\gamma K}}_{\text{descent}} + \underbrace{\gamma^{\beta - \alpha}}_{\text{strata-switching}} + \underbrace{\gamma^{2\alpha}}_{\text{descent's error}}\bigg)\,.
    \end{align*}
\end{theorem}

Explicit rate of convegence is obtained by setting $3\alpha = \beta$, $\beta < 1/(R+1)$, and $\gamma = K^{-\frac{1}{2\alpha + 1}}$. To avoid cumbersome derivations and introduction of additional parameters, quantifying the gap between $\beta$ and $1 / (R+1)$, we only present the next corollary for $\beta = 1/(R+2)$. More generally, one can use any $\beta = 1 / (R + 1 + \varepsilon)$ for $\varepsilon > 0$.
\begin{corollary}\label{cor:complexity_guarantees}
    Under conditions and notation of Theorem~\ref{thm:main_announced}, there exists $K_0 > 1$  such that for any $K \geq K_0$ and $\gamma = K^{-1 + \frac{2}{3R+8}}$, $(x_k)_{k = 1}^K$ are produced by~\eqref{eqdef:subg_des_intro} satisfy
    \begin{align*}
        \frac{1}{K}\sum_{k = 1}^K {\|\nabla g_{\Psi_k}(x_k)\|^2} \leq C K^{-\frac{2}{3R+8}} \quad\text{and}\quad \dist(x_k, \Psi_k) \leq \gamma^{\frac{3\rank(\Psi_k) + 1}{3R+6}}\,,
    \end{align*}
    where $C$ depends on $K$ only via $g_{\Psi_1}(x_1) - g_{\Psi_K}(x_{K+1})$.
\end{corollary}

Note that the guarantee improves for higher-dimensional strata: for instance, when $\rank(\Psi_k) = R-1$, one has $\dist(x_k, \Psi_k) \leq \gamma^{1 - \frac{8}{3R+6}}$, reflecting the fact that iterates remain closer to strata of larger dimension.

While the primary contribution of Theorem~\ref{thm:main_announced} is theoretical---demonstrating that finite-time convergence rates can be rigorously established within the developed framework---it is nonetheless instructive to discuss its practical scope. Even in low-dimensional settings, such results were previously unavailable, and the rates obtained there remain informative. In full generality, however, the worst-case bound exhibits an exponential dependence on the ambient dimension, which limits its utility in high-dimensional regimes.

This worst-case behavior should nonetheless be interpreted with care, as the convergence rate is governed by the rank of the underlying stratification rather than the ambient dimension per se. Furthermore, as is apparent from the proof, the analysis holds uniformly over all sufficiently small step sizes $\gamma$. Since the subgradient descent algorithm does not depend on any particular choice of stratification, the bound may be evaluated with respect to any stratification satisfying Assumption~\ref{ass:main}---in particular, one that minimizes the rank among all admissible choices.

It is also apparent from the proof that the rank $R$ need not be taken over the entire stratification $\bbX$. For instance, strata that remain at a uniformly positive distance from the realized trajectory may be pruned from consideration. This observation suggests that Theorem~\ref{thm:main_announced} could potentially be sharpened so that the rate depends only on the rank of the sub-stratification witnessed along the trajectory. We do not pursue this refinement here.

\paragraph{Further developments.} {We} use the machinery developed for the proof of Theorem~\ref{thm:main_announced} to provide two additional results in Section~\ref{sec:varying}. More precisely, we make two developments:
\begin{enumerate}
    \item Extension of Theorem~\ref{thm:main_announced} to an arbitrary decreasing sequence of step-sizes $(\gamma_k)_{k \geq 1}$;
    \item For a sufficiently fast decreasing sequence of $(\gamma_k)_{k\geq 1}$, we recover the result announced in~\cite{lai2025diameter} as a relatively direct consequence of our main results;
\end{enumerate}

\subsection{Further discussion and related works}\label{sec:discuss}

We provide in this section some additional discussion on Theorems~\ref{thm:good_assumption} and \ref{thm:main_announced}.

\paragraph{On the projection formula.}
The importance of stratifications for the analysis of subgradient methods was first highlighted in
\cite{bolte2006nonsmooth,bolte2007clarke}. In these works, the authors established a projection
formula for Clarke subgradients of semialgebraic (and more generally definable) functions as a
consequence of the existence of a stratification of the graph satisfying a Whitney-(a) condition.
This projection formula was later used in \cite{drusvyatskiy2015curves}, to show that $f$ is a Lyapunov
function for the differential inclusion $\dot{\sx}_t \in -\partial f(\sx_t)$. Combined with
stochastic approximation techniques
\cite{benaim2005stochastic,borkar2008stochastic,duchi2018stochastic}, this led to asymptotic
convergence guarantees for the subgradient method toward the $\partial f$-critical set \cite{dav-dru-kak-lee-19}.

Subsequently, it was observed that replacing the Whitney-(a) condition with stronger regularity
conditions yields refined versions of the projection formula, allowing one to control subgradients
near neighboring strata. In particular,
\cite{bianchi2024stochastic,davis2025active} showed that the Kuo--Verdier condition
implies a Lipschitz-type property for $\partial f$: Equation~\eqref{eqdef:proj_form_lip} holds
locally around any $y \in \cX$ (in particular, on this neighborhood
$\dist(y,\sX_{\dim(\cX)-1})$ can be treated as a constant). This strengthened projection formula
was used in \cite{bianchi2024stochastic,davis2025active} to analyze the stochastic subgradient
method near generic nonsmooth critical points. The Kuo--Verdier condition was also used in
\cite{josz2024sufficient} to study instability of the subgradient method. We also note that
\cite{davis2025active} observed that the Whitney-(b) condition implies a semismoothness property
of $f$ along the strata.

To our knowledge, the recent work of Lai and Song \cite{lai2025diameter} provides the first use of Lipschitz stratifications in the context of subgradient descent.
Using Lipschitz
stratifications, the authors established a weaker form of
Equation~\eqref{eqdef:proj_form_lip} for semialgebraic functions, in which the term
$1/\dist(y,\sX_{\dim(\cX)-1})$ is replaced by $1/\dist(y,\partial\cX)^{\eta}$, where
$\eta>0$ is a {\L}ojasiewicz exponent (see \cite[Proposition 3.2]{lai2025diameter}). Assuming that\footnote{We conjecture that for semialgebraic functions,
Equation~\eqref{eqdef:proj_form_lip} can be established in a form where
$\dist(y,\sX_{\dim(\cX)-1})$ is replaced by $\dist(y,\partial\cX)$, but leave
this question for future work.} locally
$\dist(y,\sX_{\dim(\cX)-1})$ is of the same order as $\dist(y,\partial\cX)$, a key distinction of
Theorem~\ref{thm:good_assumption} is that we obtain $\eta=1$. In particular, when
$\eta > 1$, our condition implies theirs.

Importantly, unlike \cite{lai2025diameter}, our analysis does not rely on
{\L}ojasiewicz inequalities to establish the projection formula. Instead,
the estimate follows directly from the existence of a Lipschitz stratification
of the graph of $f$, showing that the projection formula
is a geometric consequence of Lipschitz stratifications and does not rely on
semialgebraicity. We also note here that the curvature bound~\eqref{eqdef:Lip_tangents}, giving us the key Lemma~\ref{lem:projection}, was not exploited in \cite{lai2025diameter}.

\paragraph{Complexity rates as approximate stationarity relative to a conservative set-valued field.}

As alluded to in the introduction, the key idea underlying
Theorem~\ref{thm:main_announced} is to note that when the iterates $(x_k)$ remain
close to a stratum $\cX$, we observe a decrease of $g_{\cX}$ governed, up to
error terms, by $\norm{\nabla g_{\cX}(x_k)}$. Naturally, one may ask whether
this decrease certificate relates to known optimality measures. We provide
here a viewpoint based on conservative set-valued fields, introduced by Bolte
and Pauwels in \cite{bolte2021conservative}.

A conservative set-valued field of $f$ is a closed-graph set-valued map
$D: \bbR^d \rightrightarrows \bbR^d$ satisfying, for any absolutely continuous curve
$\c:[0,1] \rightarrow \bbR^d$,
\begin{equation*}
   \frac{\dif}{\dif t} (f\circ \c)_t
    = \scalarp{v}{\dot{\c}_t} \qquad \textrm{for a.e. $t \in [0,1]$ and all $v \in D(\c_t)$,} \, .
\end{equation*}
It was shown in \cite{bolte2007clarke} that the Clarke subgradient is a
conservative set-valued field, as well as the output of backpropagation when
applied to neural network loss functions. The key argument relies on the fact
that these operators satisfy a projection formula: there exists a
stratification $\bbX$ of the domain of $f$ such that
Equation~\eqref{eq:projformula_bolte_new} holds for any $v \in D(x)$.
Subsequently, \cite{lewis2021structure} showed that this property is also
necessary: any nonempty-valued, closed-graph operator satisfying a projection
formula is a conservative set-valued field.

Thus, given the stratification of Assumption~\ref{ass:main}, we naturally
associate to it the set-valued field\footnote{Strictly speaking, this operator does not necessarily have a closed
graph; however, it is contained in its closure
$\bar D(x;\bbX) = \cap_{\delta>0} D^\delta(x;\bbX)$, where $D^\delta$ denotes
the $\delta$-enlargement of $D$ (see Equation~\eqref{eqdef:delta_enlargD}).
The operator $\bar D$ is a conservative set-valued field.}
\begin{equation}\label{eqdef:cons_field_lips}
    D(x; \bbX)
    = \{ v \in \bbR^d : v = \nabla_{\cX} f(x) \text{ if } x \in \cX \in \bbX \}\,,
\end{equation}
that is, the minimal operator associated with $\bbX$ satisfying the projection
formula.

As we show in Lemma~\ref{lem:projection} below, for some constant $C>0$, for $x$ close to a stratum $\cX$,
\begin{equation*}
    \norm{\nabla_{\cX}f(y)} \leq C\norm{\Jac \pi_{\cX}(x) \nabla_{\cX}f(y)} = C\norm{\nabla g_{\cX}(x)}\, , \quad \textrm{ where $y = \pi_{\cX}(x)$}\, ,
\end{equation*}
Therefore, the complexity guarantees of
Corollary~\ref{cor:complexity_guarantees} can be expressed in the form
\begin{equation*}
      \min_{k \in \sqbracket{1}{K}}
      \norm{D^{\delta_k}(x_k ; \bbX)}_{0}
      \leq C K^{-\frac{2}{3R+8}} \quad\text{and}\quad \delta_k \leq \gamma^{\frac{3\rank(\Psi_k) + 1}{3R+6}}\, ,
\end{equation*}
where $\norm{\cdot}_{0}$ denotes the minimal norm of the elements of a set,
and $D^{\delta}$ is the $\delta$-enlargement of $D$:
\begin{equation}\label{eqdef:delta_enlargD}
D^{\delta}(x; \bbX)
= \{ v \in \bbR^d : v \in D(x'; \bbX),\, \norm{x-x'} \leq \delta \}\,.
\end{equation}
Thus, our rates can be interpreted as guaranteeing approximate stationarity
with respect to the conservative field induced by the stratification.

We note that conservative set-valued fields have already appeared as
optimality certificates in several contexts. In particular, due to
nonsmoothness, the output of backpropagation typically does not belong to the
Clarke subgradient, but rather to a conservative set-valued field associated
with the computational graph of the neural network. In this setting, asymptotic
convergence is therefore guaranteed toward the critical points of this field \cite{bolte2023subgradient}.
{Conservative set-valued fields also arise independently of backpropagation; for example, \cite{pauwels2023conservative} established the existence of one governing the limit behavior of the Ridge method, while \cite{pmlr-v291-schechtman25a} showed that SGD on homogeneous neural networks converges to the critical points of a conservative field associated with the margin.}

\paragraph{Comparison to other certificates of stationarity.}
Optimality measured through a conservative field differs in several respects
from notions based on the Clarke subdifferential, which we now discuss.
First, for any $x \in \bbR^d$, one has
$\partial f(x) \subset \conv D(x)$~\cite[see][Corollary 1]{bolte2021conservative}, so from this perspective the notion of
$D$-criticality is weaker. However, due to the projection formula, for any
$x \in \cX \in \bbX$, the sets $D(x)$ and $\partial f(x)$ differ only in
directions normal to the stratum, i.e., along $\cN_x \cX$. Thus, conservative
set-valued fields capture the first-order geometry of the objective along the
strata. In particular, for almost every $x$, one
has $D(x)=\partial f(x)=\{\nabla f(x)\}$.

We now turn to approximate stationarity. Given a first-order operator
$D:\bbR^d \rightrightarrows \bbR^d$, which may be a subgradient mapping or a
conservative field, two natural notions arise. The first is
\emph{$(\varepsilon,\delta)$-stationarity}, which formalizes the fact that the
current iterate lies within distance $\delta$ of a point $x$ for which there
exists $v \in D(x)$ with $\|v\| \le \varepsilon$. The second is
\emph{$(\varepsilon,\delta)$-Goldstein stationarity}, which first convexifies
the set of vectors in a $\delta$-neighborhood of $x$ and then requires the
existence of a vector of norm at most $\varepsilon$ in this convex hull.

The latter notion is strictly weaker. For instance, in
\cite[Remark~6]{shamir2020can}, for any $\delta>0$, a function is constructed
for which the origin is $(\varepsilon,\delta)$-Goldstein stationary while it
is not $(\varepsilon,\delta)$-stationary for any sufficiently small
$\varepsilon$. We note that the guarantees obtained in
\cite{goldstein1977optimization,zhang2020complexity,kong2025lipschitz,
jordan2022complexity,davis2022gradient} are expressed in terms of Goldstein
stationarity (based on the Clarke subdifferential), while
\cite{shamir2020can} shows that dimension-free rates for near-stationarity are
impossible to obtain in general.

In contrast, our approach does not rely on such convexification: the operator
$D^\delta$ that we consider is not Goldstein-like. Nevertheless, the resulting
convergence rates depend on the ambient dimension. It would therefore be
interesting to understand whether, under Assumption~\ref{ass:main}, this
dependence is intrinsic to the subgradient method, or whether
dimension-independent rates could still be achieved, possibly by introducing
a Goldstein-type convexification of $D$.

\paragraph{On the iterates remaining in $\cK$.}
{While in some cases one can establish Assumption~\ref{ass:main} for $\cK = \bbR^d$, rendering the assumption $x_k \in \cK$ void, existence of a stratification satisfying Assumption~\ref{ass:main} is guaranteed for semialgebraic sets by Theorem~\ref{thm:good_assumption}, provided that $f$ and $\cK$ are semialgebraic and $\cK$ is compact. Therefore, to apply Theorem~\ref{thm:main_announced} to the subgradient method on semialgebraic functions, one needs a guarantee that the iterates remain in a bounded set.} Such a stability assumption is already required for existing asymptotic results (see \cite{dav-dru-kak-lee-19,bolte2023subgradient}). However, in those works, there is no need to know this set (and thus the corresponding stratification) in advance. In contrast, our analysis requires specifying a compact set $\cK$ beforehand in order to construct the associated stratification.

One ad-hoc way to address this issue is to analyze the projected subgradient method onto a semialgebraic compact $\cK$. We believe that a similar analysis can be carried out for this algorithm, with rates analogous to those of Theorem~\ref{thm:main_announced}. To avoid overloading the exposition, we leave such extensions for future work.

Furthermore, in some cases boundedness can be ensured by structural properties of $f$. For instance, if the objective is semialgebraic and coercive, then for sufficiently small $\gamma$ the iterates remain in a compact set \cite{josz2024global,bolte2025inexact}. A prominent example arises in machine learning, where a nonnegative loss function is regularized by a coercive term, thereby ensuring stability of the iterates.

\section{Proof of Theorem~\ref{thm:main_announced}}
\label{sec:proof_main}

\begin{table}[t!]
\begin{center}
\begin{tabular}{{@{}lll@{}}}
    \toprule
    Notation& Meaning &Defined \\
 \midrule
 $\bbX$ & stratification & Section~\ref{sec:strat_notation}\\
  $\sX_j $ & union of strata of dimension less than $j$ & Equation~\eqref{eqdef:closed_skeleton} \\
  $\rank(\bbX)$ & number of $j$ such that $\bbX_{j}$ is nonempty & Definition~\ref{def:rank_stratum} \\
  $\rank(\cX)$ & $r$ such that $\cX \in \bbX_{r}$ & Definition~\ref{def:rank_stratum}\\
$\cC^{*}(\cX)$ & well-behaved tubular neighborhood of $\cX$ & Equation~\eqref{eqdef:verybig_cone} \\
$\bigC(\cX)$ & outer tubular neighborhood of $\cX$ & Equation~\eqref{eqdef:large_cone} \\
$\smallC(\cX)$ & inner tubular neighborhood of $\cX$ & Equation~\eqref{eqdef:thin_cone}\\
$g_{\cX}$ & equal to $f \circ \pi_{\cX}$, well-defined on $\cC^{*}(\cX)$ & Equation~\eqref{eqdef:f_projX} \\
$\pi_{\cX}$ & is the projection onto $\cX$, well-defined on $\cC^{*}(\cX)$ & Section~\ref{sec:notation} \\
$P_x$ & projection onto $\cT_x \cX$, for $x \in \cX \in \bbX$ & Section~\ref{sec:submanif_notation}\\
 $\bbX_{\square \,j}$ & set of strata of dimension $\square$ $j$ & Equation~\eqref{eqdef:strat_quantifiers} \\
 $\bigC_{\square \,j}$& union of large tubular neighborhoods of dimension $\square \,j$& Equation~\eqref{eqdef:cone_quantifiers}\\
  $\smallC_{\square \,j}$& union of thin tubular neighborhoods of dimension $\square \,j$& Equation~\eqref{eqdef:cone_quantifiers}\\
 \bottomrule
\end{tabular}
\end{center}
\caption{Notation relative to a stratification $\bbX$. Here, $\square  \in \{\leq, <, =\}$}\label{tab:not_strat}
\end{table}

\begin{table}[t!]
\begin{center}
\begin{tabular}{@{}lll@{}}
    \toprule
    Notation& Relations &Defined \\
 \midrule
$G$ & $\norm{v} \leq G$, for all $v \in\partial f(x)$, $x \in \cK$ & Assumption~\ref{ass:f_lip} \\
  $L_0, \mathtt{A}_3, \ulambda, \blambda >0$ & such that Lemma~\ref{lem:projection} holds for all strata $\cX$ & Lemma~\ref{lem:projection}\\
 $\mathtt{A}_1$, $\mathtt{A}_2$, $\mathtt{A}_3$, &  $\mathtt{A}_1 = \ulambda / (16\blambda^2)$, $\mathtt{A}_2 = L_2^2(4\blambda^2/\ulambda + \ulambda/2)$, $4\mathtt{A}_3\leq 1$ & Lemma~\ref{lem:descent_deterministic} \\
  $\alpha,\beta \in (0,1)$, $R>0$ & $\alpha < \beta$, $(R+1)\beta < 1$ and $R= \rank(\bbX)$ & Theorem~\ref{thm:main_announced} \\
    $\gamma < \gamma_0 \leq 1$ & $4\gamma_0^{\beta - \alpha} \leq 1$, $2 G \leq \gamma_{0}^{(R+1)\beta -1}$, $\max\left\{4 \blambda G  , 8  L_2 \tfrac{\blambda^2}{\ulambda}, 2L_1G\bar{\lambda} \right\}\leq \gamma_{0}^{R\beta-1}$ & Theorem~\ref{thm:main_announced}\\
 \bottomrule
\end{tabular}
\end{center}
\caption{Notation relative to various constants.}\label{tab:not_const}
\end{table}

Before delving into technicalities let us describe the main ideas behind the proof of Theorem~\ref{thm:main_announced}.

The main argument is organized into four parts. The first part establishes geometric and analytic consequences of Assumption~\ref{ass:main}. The second part introduces and analyzes inner and outer neighborhoods of each stratum $\cX \in \bbX$. The third part shows how local descent relations with additional vicinity conditions combine along the iterates. The fourth part constructs a suitable rule for selecting the strata that ensures that the additional terms introduced in the second step remain controlled.

\begin{enumerate}
    \item  In the first step, we derive consequences of Assumption~\ref{ass:main}. Mainly, we show that this assumption guarantees two structural properties of the stratification and of the projected objective functions.
    \begin{enumerate}
        \item The projection onto a stratum $\cX \in \bbX$ is well-defined in a neighbourhood of $\cX$ and the spectrum of $\Jac \pi_{\cX}$ is lower and upper bounded in the directions tangent to $\cX$.
        Moreover, the size of this neighbourhood decreases linearly when approaching the boundary of $\cX$. Intuitively, this means that the projection is well-defined as long as the point of interest is closer to the strata itself than to $\sX_{\dim(\cX) - 1}$.
        These two properties of projection are established in Lemma~\ref{lem:projection}.
        \item Suppose that $x^+ = x - \gamma v$ with $v \in \partial f(x)$ and the entire segment $[x, x^+]$ lies in a region where the projection onto $\cX \in \bbX$ is well-defined. Then, as long as $\dist(x, \sX_{\dim(\cX) - 1}) \gtrsim \gamma$ the function $g_{\cX} \eqdef f \circ \pi_{\cX}$ satisfies an approximate descent inequality of the form
        \begin{align*}
            g_{\cX}(x^+) - g_{\cX}(x) \lesssim -\gamma \|\nabla g_{\cX}(x)\|^2 + \gamma \left(\frac{\dist(x, \cX)}{\truncdist(x, \sX_{\dim(\cX) - 1})}\right)^2\,.
        \end{align*}
        The first term corresponds to the usual descent obtained in smooth optimization, while the second term measures the error caused by projecting onto the stratum. This error becomes small when the iterate is much closer to $\cX$ compared to lower-dimensional strata. This result is established in Lemma~\ref{lem:descent_deterministic}.
    \end{enumerate}
    \item To each stratum $\cX \in \bbX$, we associate two nested conic neighbourhoods--outer and iner one--defined via $\alpha < \beta$ appearing in Theorem~\ref{thm:main_announced}. We study their geometric properties in Section~\ref{sec:geom_cones} yielding the condition $\beta \leq 1/(R+1)$. Intuitively, entering the inner neighbourhood of $\cX$ serves as a potential trigger to start using the descent inequality for $g_{\cX}$, while exiting the outer neighbourhood indicates that this descent is no longer reliable and that one should have switched to another stratum in-between the entry and the exit.
    \item We show how the descent relations obtained above accumulate along the sequence of iterates. To do this, we introduce a strata selection function $(\Psi_k)_{k = 1}^K$, where $\Psi_k \in \bbX$ specifies which stratum is used to evaluate the descent at iteration $k \in \sqbracket{1}{K}$.
    For a large class of such sequences (see Definition~\ref{def:valid_switch}), we prove in Lemma~\ref{lem:valid_descent} that
    \begin{align*}
    \sum_{ k = 1}^K \gamma \|\nabla g_{\Psi_k}(x_k)\|^2 \lesssim {\Delta_f} + \sum_{k = 2}^K \left(g_{\Psi_k}(x_k) - g_{\Psi_{k-1}}(x_k)\right) + \sum_{k = 1}^K \gamma\left(\frac{\dist(x, \Psi_k)}{\truncdist(x, \sX_{\dim(\Psi_k) - 1})}\right)^2\,,
    \end{align*}
    where $g_{\cX} \eqdef g \circ \pi_{\cX}$ and $\Delta_f \eqdef g_{\Psi_1}(x_1) - g_{\Psi_K}(x_{K+1})$.
    The three terms on the right-hand side have a natural interpretation. The first term corresponds to the cumulative descent of the projected objectives. The second term appears whenever the selected stratum changes between two consecutive iterations; it can therefore be interpreted as a switching cost between strata. The third term aggregates the projection errors introduced in the approximate descent inequalities of the first step.
    \item In the last step, we show that it is possible to construct a strata selection function that keeps the switching costs under control. This class of strata selection functions is introduced in Definition~\ref{def:good_switch}, and we show that any such sequence yields the desired control in Theorem~\ref{lem:rate_for_good}. Intuitively, Definition~\ref{def:good_switch} enforces the use of the lowest-dimensional stratum $\cX$ for as long as possible, and allows a switch to another stratum only if the iterate will have crossed the buffer region between the small and the large neighbourhoods associated with $\cX$.

    This part of the proof is constructive and purely combinatorial: we introduce Algorithm~\ref{algo:main} that selects $\Psi_k$ based on the position of the iterates relative to the strata. The algorithm ensures that the selected stratum remains stable whenever possible and only switches when necessary. In particular, Theorem~\ref{thm:algo_produce_good} establishes that Algorithm~\ref{algo:main} does indeed produce a selection function with desired properties.
\end{enumerate}
\paragraph{Organization of Section~\ref{sec:proof_main}.} The remainder of this section is organized as follows. In Section~\ref{sec:consequences}, we establish the main consequences of Assumption~\ref{ass:main}. Section~\ref{sec:nest_neighb} introduces outer and inner neighborhoods of strata, whose geometry is studied in Section~\ref{sec:geom_cones}. Section~\ref{sec:main} presents strata selection functions and establishes key inequalities for any such selection rule. In Section~\ref{sec:good_selection}, we impose additional constraints on the selection function and show how the main result of this work follows from the existence of such a function. Finally, Section~\ref{sec:goo_existence} proves the existence of the aforementioned selection function.
\subsection{Consequences of assumptions}
\label{sec:consequences}

As alluded previously, the following lemma shows that Assumption~\ref{ass:main} allows one to construct around any stratum, a tubular neighborhood, on which the projection behaves in a controlled way. This is a key step for writing the dynamic of the projected iterates. To not interrupt the exposition we provide its proof in Section~\ref{sec:rem_proofs}.

\begin{lemma}[Projection regularity]
    \label{lem:projection}
    Let $\bbX$ be a $C^p$ stratification of $\cK \subset \bbR^d$, satisfying Assumption~\ref{ass:main}\ref{ass:metric}--\ref{ass:main2}. There exist constants $L_0, \mathtt{A}_3, \blambda, \ulambda > 0$, such that for any $\cX \in \bbX$ such that $\dim(\cX) = j$
    the projection $\pi_{\cX}$ is well-defined on
    \begin{align}\label{eqdef:verybig_cone}
        \cC^*(\cX) = \left\{x \in \bbR^d \,:\, \dist(x, \cX) \leq \mathtt{A}_3 \cdot \truncdist(x, \sX_{j-1})\right\}\,.
    \end{align}
    Furthermore, for all $x \in \cC^*(\cX)$, denoting $y = \pi_{\cX}(x)$, we have
    \begin{align*}
        \|\Jac \pi_{\cX}(x) - P_{y}\| \leq L_0 \frac{\|x - y\|}{\truncdist(x, \sX_{j-1})}\,.
    \end{align*}
    And denoting $\lambda_{\max}$ the maximal eigenvalue of $\Jac \pi_{\cX}(x)$ and $\lambda_{\min}$ the minimal eigenvalue of $\Jac \pi_{\cX}$ \emph{restricted} to $\cT_{y} \cX$,
    \begin{equation}\label{eqdef:min_max_jac}
     \ulambda \leq \lambda_{\min} \leq \lambda_{\max}\leq \blambda\,.
   \end{equation}
\end{lemma}

The next lemma establishes an approximate descent with an extra error related to the projection. To keep exposition more compact, we (re)-introduce the following notation that will be used throughout the paper: for $\cX \in \bbX$ and $x \in \cC^*(\cX)$, define
\begin{align}\label{eqdef:f_projX}
    \boxed{g_{\cX}(x) \eqdef (f \circ \pi_{\cX})(x)\,.}
\end{align}
More concretely, the result below is established for $g_{\cX}$.
\begin{lemma}[Stratified descent lemma]
    \label{lem:descent_deterministic}
Under the assumptions of Lemma~\ref{lem:projection}, let $x \in \cK$, $\gamma > 0$ and
\begin{align*}
    x^+ = x - \gamma v\,,
\end{align*}
where $v \in \partial f(x)$. Assume that for some $\cX \in \bbX$, $[x, x^+] \in \cC^*(\cX)$, with $\dim(\cX) = j$ and $4\mathtt{A}_3 \leq 1$
, and that
\begin{equation*}
    \gamma \max\left(4 \blambda G , 8  L_2 \tfrac{\blambda^2}{\ulambda}, 2L_1G\bar{\lambda} \right) \leq \dist(x, \sX_{j-1})
\end{equation*}
\begin{align*}
    g_{\cX}(x^+) - g_{\cX}(x) \leq -\gamma \mathtt{A}_1 \|\nabla g_{\cX}(x)\|^2 + \gamma \mathtt{A}_2\left(\frac{\dist(x, \cX)}{\truncdist(x, \sX_{j-1})}\right)^2
\end{align*}
where $\mathtt{A}_1 = \ulambda / (16\blambda^2)$ and $\mathtt{A}_2 = L_2^2(\tfrac{4\blambda^2}{\ulambda} + \tfrac{\ulambda}{2})$
\end{lemma}

\begin{proof}
Set $\c_t = x + t(x^+ - x)$ and $\tc_t = \pi_{\cX}(\c_t)$, recall that $P_{\tc_t}$ is the orthogonal projection onto $\cT_{\tc_t} \cX$.
    The proof is split into three parts.

    \paragraph{Part I: distance between projected points.}
    For the first part, observing that $\dot{\tc}_t = \Jac \pi_{\cX}(\c_t)(x^+ - x)$ and using Lemma~\ref{lem:projection}, we can write
    \begin{align}
        \label{eq:descent0}
        \|\tc_1 - \tc_0\| \leq \sup_{t \in [0, 1]}\|\tc_t- \tc_0\| = \sup_{t \in [0, 1]}\bigg\|\int_{0}^t\Jac \pi_{\cX}(\c_s)(x^+ - x) \dif s\bigg\| \leq \blambda \gamma\int_{0}^1\| P_{\tc_s}v\| \dif s\,.
    \end{align}

    \paragraph{Part II: first version of descent.}
    For the second part, using Assumption~\ref{ass:main}-\ref{ass:main4} and the fact that $\nabla g(\c_t) = (\Jac \pi_{\cX}(\c_t))^{\top} \nabla_{\cX}f(\tc_t)$, we obtain
    \begin{equation}\label{eq:desc_interm}
        \begin{split}
        g(x^+) - g(x)
        &=
        -\gamma\int_{0}^1\scalarp{\nabla g(\c_t)}{v} \dif t\\
        &= - \gamma \int_{0}^1 \scalarp{\nabla_{\cX} f(\tc_t)}{\Jac \pi_{\cX}(\c_t)v} \dif t\\
        &= - \gamma \int_{0}^1 \scalarp{P_{\tc_t}v}{\Jac \pi_{\cX}(\c_t)v} \dif t- \gamma \int_{0}^1 \scalarp{\nabla_{\cX} f(\tc_t) - P_{\tc_t}v}{\Jac \pi_{\cX}(\c_t)v} \dif t \\
        &\stackrel{(a)}{\leq}  - \gamma \int_{0}^1 \scalarp{P_{\tc_t}v}{\Jac \pi_{\cX}(\c_t)v} \dif t + \gamma L_2 \int_{0}^1 \frac{\|\tc_t - \c_0\|}{\truncdist(\tc_t, \sX_{j-1})}  \norm{\Jac \pi_{\cX}(\c_t)v} \dif t\\
        &\stackrel{(b)}{\leq} -\gamma \ulambda \int_{0}^1 \|P_{\tc_t}v\|^2 \dif t +  \gamma L_2\blambda \left(\int_{0}^1 \frac{\|\tc_t - \c_0\| \cdot \|P_{\tc_t}v\|}{\truncdist(\tc_t, \sX_{j-1})}\right) \dif t\,,
        \end{split}
    \end{equation}
    where $(a)$ follows from Assumption~\ref{ass:main}\ref{ass:main4} and $(b)$ follows from the bound on the Jacobian, established in Lemma~\ref{lem:projection}.
    The goal now is to upper bound $\tfrac{\|\tc_t - \c_0\|}{\truncdist(\tc_t, \sX_{j-1})}$. We consider the numerator and denominator separately.

  For the numerator, Equation~\eqref{eq:descent0} and triangle's inequality further yield
    \begin{align*}
        \|\tc_t- \c_0\|
        &\leq \dist(x, \cX) + \|\tc_t - \tc_0\|
        \leq \dist(x, \cX) + \blambda \gamma \int_{0}^1 \|P_{\tc_t}v\| \dif t\,.
    \end{align*}
   For the denominator, since $\norm{v} \leq G$, we obtain that
    \begin{equation}
    \label{eq:descent1}
    \begin{aligned}
        \dist(\tc_t, \sX_{j-1})
        &\geq
        \dist(\c_0, \sX_{j-1}) - \|\c_0- \tc_t\|\\
        &\geq
        \dist(\c_0, \sX_{j-1}) - \dist(x, \cX) - \blambda G \gamma \\
        &\geq
        \frac{1}{2} \dist(\c_0, \sX_{j-1}) = \frac{1}{2} \dist(x, \sX_{j-1})\,,
    \end{aligned}
    \end{equation}
    where we have used the fact that both $\dist(x, \cX)$ and $G \blambda \gamma$ are less than $\dist(x, \sX_{j-1})/4$, by assumption of the lemma.

    Thus, combining the two above displays with Equation~\eqref{eq:desc_interm}, for $\Delta = g(x^+) - g(x)$, we can write
    \begin{equation}
        \label{eq:descent3}
    \begin{aligned}
        \Delta
        &\leq
        -\gamma \ulambda \int_{0}^1 \|P_{\tc_t}v\|^2 \dif t + 2\gamma L_2\blambda  \frac{\dist(x, \cX)}{\truncdist(x, \sX_{j-1})} \sqrt{\int_{0}^1 \|P_{\tc_t}v\|^2 \dif t}
        + 2 L_2 \gamma^2\blambda^2 \frac{\int_{0}^1 \|P_{\tc_t}v\|^2 \dif t}{\truncdist(x, \sX_{j-1})}\\
        &\leq
        - \frac{3\gamma\ulambda}{4} \int_{0}^1 \|P_{\tc_t}v\|^2 \dif t + \frac{4 \gamma L_2^2 \blambda^2}{\ulambda} \left(\frac{\dist(x, \cX)}{\truncdist(x, \sX_{j-1})}\right)^2
        + 2 L_2 \gamma^2\blambda^2 \frac{\int_{0}^1 \|P_{\tc_t}v\|^2 \dif t}{\truncdist(x, \sX_{j-1})}\\
        &\leq
        - \frac{\gamma\ulambda}{2} \int_{0}^1 \|P_{\tc_t}v\|^2 \dif t + \frac{4 \gamma L_2^2 \blambda^2}{\ulambda} \left(\frac{\dist(x, \cX)}{\truncdist(x, \sX_{j-1})}\right)^2\,,
    \end{aligned}
    \end{equation}
    where in the second inequality we used the relation $2ab\leq \varepsilon a^2 + b^2/\varepsilon$, for $a,b,\varepsilon >0$, and in last one the assumption that $\dist(x, \sX_{j-1}) \geq 8  L_2 \tfrac{\blambda^2}{\ulambda} \gamma$.

    \paragraph{Part III: deriving the final claim.} To derive the final claim, we further bound~\eqref{eq:descent3}. Let $y = \tc_0 = \pi_{\cX}(x)$. To bound the first term, we use triangle's inequality to obtain
    \begin{align}
        \label{eq:descent2}
        \sqrt{\int_{0}^1 \|P_{\tc_t}v\|^2 \dif t} \geq \|\nabla_{\cX} f(y)\| - \sqrt{\int_{0}^1 \|P_{\tc_t}v - \nabla_{\cX} f(y)\|^2 \dif t} \,.
    \end{align}
    Let us upper bound the second term on the right-hand-side of the above inequality. For any $t \in (0, 1)$, it holds thanks to Assumptions~\ref{ass:main}\ref{ass:main2} and~\ref{ass:main}\ref{ass:main4} that
    \begin{align*}
        \|P_{\tc_t}v - \nabla_{\cX} f(y)\| \leq \|P_{y}v - \nabla_{\cX} f(y)\| + \|(P_{y} - P_{\tc_t})v\| \leq L_2 \frac{\dist(x, \cX)}{\truncdist(y, \sX_{j-1})} + L_1G\frac{\|\tc_0 - \tc_t\|}{\truncdist(\tc_t, \sX_{j-1})}\,.
    \end{align*}
    By Equations~\eqref{eq:descent0} and \eqref{eq:descent1}, we can bound the second term on the right hand side above as
    \begin{align*}
        \frac{\|\tc_0 - \tc_t\|}{\truncdist(\tc_t, \sX_{j-1})} \leq 2\bar{\lambda}\gamma\frac{\int_{0}^1 \|P_{\tc_t}v\| \dif t}{\truncdist(x, \sX_{j-1})}\,.
    \end{align*}
    Similarly, by~\eqref{eq:descent1}, $\truncdist(x, \sX_{j-1}) \leq 2\truncdist(y, \sX_{j-1})$.
    Thus, the above two displays, in combination with Equation~\eqref{eq:descent2} and Jensen's inequality, yield
    \begin{align*}
        \sqrt{\int_{0}^1 \|P_{\tc_t}v\|^2 \dif t} \geq \|\nabla_{\cX} f(y)\| - {2 L_2} \frac{\dist(x, \cX)}{\truncdist(x, \sX_{j-1})} - 2L_1G\bar{\lambda}\gamma\frac{\sqrt{\int_{0}^1 \|P_{\tc_t}v\|^2 \dif t}}{\truncdist(x, \sX_{j-1})}\,.
    \end{align*}
    Since {$\tfrac{2L_1G\bar{\lambda}\gamma}{\truncdist(x, \sX_{j-1})} \leq 1$}, we deduce from the above
    \begin{align*}
        \sqrt{\int_{0}^1 \|P_{\tc_t}v\|^2 \dif t} \geq \frac{\|\nabla_{\cX} f(y)\|}{2} - {L_2} \frac{\dist(x, \cX)}{\truncdist(x, \sX_{j-1})}\,.
    \end{align*}

    A simple fact: $a \geq b - c \implies a^2 \geq \frac{b^2}{2} - c^2$, implies
    \begin{align*}
        {\int_{0}^1 \|P_{\tc_t}v\|^2 \dif t} \geq \frac{\|\nabla_{\cX} f(y)\|^2}{8} - {L_2^2} \left(\frac{\dist(x, \cX)}{\truncdist(x, \sX_{j-1})}\right)^2\,.
    \end{align*}
    Substitution of the above into Equation~\eqref{eq:descent3} in combination with Lemma~\ref{lem:projection} concludes the proof.
\end{proof}

\begin{remark}
    Note that the proof of Lemma~\ref{lem:descent_deterministic} also establishes
    \begin{align*}
    g_{\cX}(x^+) - g_{\cX}(x) \leq -\gamma \tilde{\mathtt{A}}_1 \|\nabla_{\cX}f(\pi_{\cX}(x))\|^2 + \gamma \mathtt{A}_2\left(\frac{\dist(x, \cX)}{\truncdist(x, \sX_{j-1})}\right)^2\,,
\end{align*}
where $\tilde{\mathtt{A}}_1 = \underline{\lambda} / 16$ and $\mathtt{A}_2 = L_2^2(\tfrac{4\blambda^2}{\ulambda} + \tfrac{\ulambda}{{2}})$
\end{remark}
    {We note that} if $\dim(\cX) = 0$, (that is $\cX = \{ z\}$ for some $z \in \bbR^d$) then the statement of Lemma~\ref{lem:descent_deterministic} is vacuous.
    Indeed, in our analysis, in such case $z$ is a critical point of the conservative field $D(\cdot; \bbX)$. If $0 \in \partial f(z)$, then our approach establishes approximate Clarke-stationarity: $\norm{x_k - z} \leq \gamma^{\alpha}$. Otherwise, it might be considered as a spurious point. More generally, the set of $D$-critical points can be different from the Clarke ones. While we do not pursue the question of avoidance of such spurious point, we note here that the set of the introduced zero-dimensional spurious points is finite. Furthermore, in Appendix~\ref{sec:zaplatki} we show how to adapt our analysis to show that, for a well-chosen $\gamma_0$, it is impossible to converge to such zero-dimensional spurious points.

\subsection{Nested neighbourhoods}\label{sec:nest_neighb}
\begin{figure}[t]
  \centering
  \includegraphics[width=\linewidth]{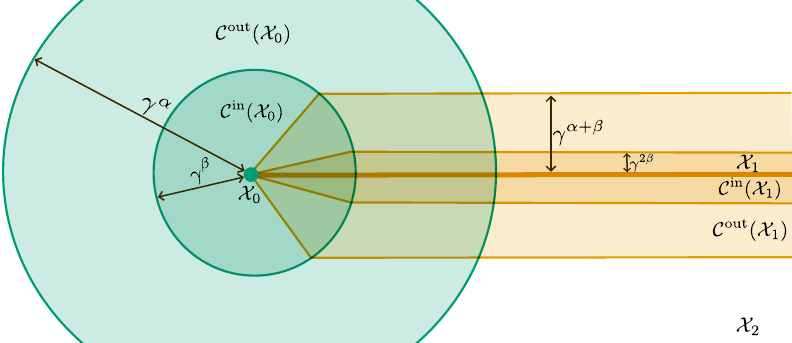}
  \caption{Stratification with associated inner and outer neighbourhoods. The stratification is $\bbX = \{\cX_0, \cX_1, \cX_2\}$ with $\cX_0 = \{(0,0)\}$, $\cX_1 = \{(x, 0) \,:\, x > 0\}$, $\cX_2 = \bbR^2 \setminus (\cX_0 \cup \cX_1)$. For this example we have $R = 2$, $\dim(\cX_j) = j$.}
  \label{fig:cones}
\end{figure}

Since this section introduces a significant amount of new notation, we collect,
for the reader's convenience, all notation related to the stratification in
Table~\ref{tab:not_strat}.

As established in the previous section, when the iterates are close to a stratum
$\cX$, subgradient descent yields an approximate decrease of $g_{\cX}$.
Consequently, if the entire trajectory remained close to a single stratum
$\cX$, convergence rates would follow directly from
Lemma~\ref{lem:descent_deterministic}. In general, however, no such guarantee
is available: the iterates may be close to one stratum at some iteration and
later move sufficiently far to approach another. This motivates the
introduction of ``buffer'' regions, which the iterates must traverse before
leaving the neighborhood of a given stratum $\cX$.

To this end, for each $\cX$, we introduce the \emph{outer} and \emph{inner} conical neighborhoods $\smallC(\cX) \subset \bigC(\cX) \subset \cC^{*}(\cX)$. The buffer region then corresponds to $\bigC(\cX) \backslash \smallC(\cX)$.

We fix parameters
$\alpha < \beta \in (0,1)$ and $\gamma \in (0,\gamma_0)$, which appear in
Theorem~\ref{thm:main_announced}.
For every $\cX \in \bbX$ with $\dim(\cX) < d$,
define \emph{outer} and \emph{inner} neighbourhoods respectively as

\begin{empheq}[box=\fbox]{align}
    \bigC(\cX) &=  \left\{ x \in \bbR^d: \dist(x, \cX) \leq \gamma^{\alpha} \cdot
    \min\left\{\gamma^{\rank(\cX)\beta},\, \dist(x, \sX_{\dim(\cX)-1})\right\}
    \right\}\, ,\label{eqdef:large_cone}\\
    \smallC(\cX) &=  \left\{ x \in \bbR^d: \dist(x, \cX) \leq \gamma^{\beta} \cdot \min\left\{\gamma^{\rank(\cX)\beta}\,,\dist(x, \sX_{\dim(\cX)-1}) \right\}\right\}\, .\label{eqdef:thin_cone}
\end{empheq}

Additionally, for every $\cX \in \bbX$ with $\dim(\cX) = d$, we set $\bigC(\cX) = \smallC(\cX) = \cX$.
We note that, as long as $\gamma^\alpha \leq \mathtt{A}_3$, for any $\cX \in \bbX$, we have $ \smallC(\cX) \subset \bigC(\cX) \subset \cC^*(\cX)$. In particular, from Lemma~\ref{lem:projection}, the projection onto $\cX$ is well-defined in outer (and inner) neighbourhood of $\cX$.  Figure~\ref{fig:cones} displays an example of stratification and corresponding neighbourhoods.

Finally, before presenting geometric properties of the introduced neighbourhoods, let us define additional notation that will be used throughout. For $\square \in \{\leq, <, =\}$, we define
\begin{equation}\label{eqdef:strat_quantifiers}
    \bbX_{\square \,j} = \{\cX \in \bbX: \dim(\cX) \square \,j \}\, ,
\end{equation}
and
\begin{align}\label{eqdef:cone_quantifiers}
    &\smallC_{\square \,j} \eqdef \bigcup_{\substack{\cX \in \bbX_{\square \,j}}} \smallC(\cX) \quad\text{ and }\quad \bigC_{\square \,j} \eqdef \bigcup_{\substack{\cX \in \bbX_{\square \,j}}} \bigC(\cX)\, ,
\end{align}
 with the agreement that $ \bbX_{< 0} = \smallC_{< 0} = \emptyset$.

\subsubsection{Geometric properties of nested neighbourhoods}\label{sec:geom_cones}

In this section, we collect several statements describing the geometry of the neighborhoods $\smallC$ and $\bigC$. Since their proofs are simple and mostly rely on the $1$-Lipschitz continuity of $x \mapsto \dist(x, \cX)$, they are deferred to Appendix~\ref{sec:proof_for_cones}. For intuition, the reader may keep Figure~\ref{fig:cones} in mind. Although it is a toy example, it captures the essential features underlying the results.

We first state a lemma that describes various metric relations between the inner/outer neighbourhoods.
\begin{lemma}
    \label{lem:geom_all}
    Let $\bbX$ satisfy Assumption~\ref{ass:main}, let $\gamma_0 \in (0, 1)$ be such that $4 \gamma_{0}^{\beta} \leq \gamma_{0}^{\alpha}$. Let $\gamma < \gamma_0$ and $\cX \in \bbX$, then
    \begin{enumerate}
        \item\label{:1} $x \in \smallC(\cX) \quad\implies\quad \dist(x, \cX) \leq \gamma^{(1 + \rank(\cX))\beta}$;
        \item\label{:2} $x \in \bigC(\cX) \quad\implies\quad \dist(x, \cX) \leq \gamma^{\alpha + \rank(\cX)\beta}$;
        \item\label{:3} $x \notin \smallC_{< \dim(\cX)} \quad\implies\quad \dist(x, \sX_{\dim(\cX) - 1}) \geq \gamma^{\rank(\cX)\beta}$;
        \item\label{:6} $x \in \smallC(\cX) \setminus \smallC_{ < \dim(\cX)} \text{ and } x' \notin \bigC(\cX) \quad\implies\quad \|x - x'\| \geq \gamma^{\alpha + \rank(\cX)\beta} / 4$;
    \end{enumerate}
    Moreover, if $(R+1)\beta < 1$ and $2G \leq \gamma_{0}^{(R+1) \beta - 1}$, then
    \begin{enumerate}
        \setcounter{enumi}{4}
        \item\label{:4} $x \in \smallC(\cX) \setminus \smallC_{<\dim(\cX)} \text{ and } \|x' - x\| \leq G \gamma \quad\implies\quad \dist(x', \cX) \leq 3\gamma^{\beta}\min\left\{\gamma^{\rank(\cX)\beta},\,\dist(x', \sX_{\dim(\cX)-1})\right\}$;
        \item\label{:5} $x \in \bigC(\cX) \setminus \smallC_{<\dim(\cX)} \text{ and } \|x' - x\| \leq G \gamma \quad\implies\quad \dist(x', \cX) \leq 3\gamma^{\alpha}\min\left\{\gamma^{\rank(\cX)\beta},\,\dist(x', \sX_{\dim(\cX)-1})\right\}$;
    \end{enumerate}
\end{lemma}
We comment on the importance of each item of Lemma~\ref{lem:geom_all}.

Items~\ref{:1} and~\ref{:2} of Lemma~\ref{lem:geom_all} are direct consequences of definitions in Equations~\eqref{eqdef:large_cone},~\eqref{eqdef:thin_cone}. They form the first constraint in the definition of the eventual strata selection sequence. If, at iteration $k$, we decide to use the stratum $\cX$ and the corresponding descent on $g_{\cX}$, we require that $x_k \in \bigC(\cX)$. This ensures that the iterate remains sufficiently close to the stratum $\cX$ (see Definition~\ref{def:valid_switch}).

Item~\ref{:3} of Lemma~\ref{lem:geom_all} states that if $x \not \in \smallC_{<j}$, then the distance to $\sX_{j-1}$ can be bounded from below. In particular, such an $x$ cannot approach any stratum of dimension strictly smaller than $j$.

Before discussing Item~\ref{:6} of Lemma~\ref{lem:geom_all}, let if state the key consequence of Item~\ref{:3} of Lemma~\ref{lem:geom_all} and discuss Items~\ref{:4} and~\ref{:5}. The proof of the following corollary is immediate.
\begin{corollary}
    \label{cor:distance_to_boundary_is_large}
    In the context of Lemma~\ref{lem:geom_all}, for all $\gamma \in (0, \gamma_0)$, and $\forall \cX \in \bbX$, $\forall x \notin \smallC_{< \dim(\cX)}$
    \begin{align*}
     \min\left\{\gamma^{\rank(\cX)\beta},\,\dist(x, \sX_{\dim(\cX)-1})\right\}
     \geq \frac{\gamma^{1-\beta}}{\gamma_0^{1 - (R+1)\beta }}\,.
    \end{align*}
\end{corollary}
This corollary has long-standing consequences and is the main bottleneck of the proof. Indeed, the condition of $\beta < 1 / (R+1)$ in the main Theorem~\ref{thm:main_announced} comes only from the above lemma. The yet not discussed items of Lemma~\ref{lem:geom_all} are all consequences of Corollary~\ref{cor:distance_to_boundary_is_large}.

Items~\ref{:4} and~\ref{:5} of Lemma~\ref{lem:geom_all} guarantee a certain notion of stability for two consecutive iterations. The key consequence is: if $x \in \smallC(\cX) \setminus \smallC_{< \dim(\cX)}$ and $x'$ is any other point that is at most $G\gamma$ away from $x$, then
\begin{itemize}
    \item  The segment $[x, x'] \subset \cC^*(\cX)$, that is, projection onto $\cX$ is well-defined on along the entire segment;
    \item Although it may happen that $x' \notin \smallC(\cX)$, the point $x'$ remains in a slightly inflated version of $\smallC(\cX)$
    \item The same statement holds for $x \in \bigC(\cX) \setminus \smallC_{< \dim(\cX)}$;
\end{itemize}
These properties match exactly the conditions required in Lemma~\ref{lem:descent_deterministic}, which constitutes the first step in the proof of Theorem~\ref{thm:main_announced}. Consequentely, this result highlights the second constraint in constructing strata selection functions. When, at some iteration $k$, we decide to use a stratum $\cX$ and the corresponding descent on $g_{\cX}$, we must ensure that $x_k \notin \smallC_{< \dim(\cX)}$. It will be reflected in Definition~\ref{def:valid_switch}. This condition guarantees that the iterate remains sufficiently separated from lower-dimensional strata, which ensures that the stability properties hold from one iteration to the next.

Finally, Item~\ref{:6} of Lemma~\ref{lem:geom_all} highlights the role of introducing both $\smallC$ and $\bigC$ for each stratum. It shows that if one starts from a point in $\smallC(\cX) \setminus \smallC_{ < \dim(\cX)}$ and moves to a point outside $\bigC(\cX)$, then the displacement must be relatively large compared to the step-size. This, in turn, guarantees that a non-negligible number of iterations needs to be spent in order to move from $\smallC(\cX) \setminus \smallC_{ < \dim(\cX)}$ to the complement of $\bigC(\cX)$.
This result leads to the third property required for a strata selection function, which will be formalized in Definition~\ref{def:good_switch}. Informally, the idea is the following. Suppose that at iteration $k$ the descent associated with a stratum $\cX$ is used, while at iteration $k+1$ the procedure switches to a stratum $\cX'$ with $\dim(\cX') \geq \dim(\cX)$. We require that $x_k \in \smallC(\cX) \setminus \smallC_{< \dim(\cX)}$ and that there exists a future time $\hat{k}$ such that $x_{\hat{k}} \notin \bigC(\cX)$. Moreover, between $k$ and $\hat{k}$ the stratum $\cX$ is not used.

Combined with Item~\ref{:6} of Lemma~\ref{lem:geom_all}, this condition ensures that once $\cX$ is selected, it must remain active for a sufficiently long period before a switch occurs. As a consequence, the number of switches from $\cX$ in the strata selection function can be properly controlled.

\subsection{Main result: strata selection functions}
\label{sec:main}

\begin{table}[t!]
\begin{center}
\begin{tabular}{@{}lll@{}}
    \toprule
    Notation& Meaning &Defined \\
 \midrule
 $K$ & number of iterates & \\
 $\bPsi:\bbR^{d \times K} \rightarrow \bbX^{K}$ & a strata selection function & Equation~\eqref{eqdef:strata_selection}\\
$\lswitch(\cX)$ & iterates $k$ so that $\Psi_{k} = \cX \neq \Psi_{k-1}$ and $\dim(\Psi_{k-1}) \geq \dim(\Psi_k)$ & Definition~\ref{def:individual_switches}\\
$\rswitch(\cX)$ & iterates $k$ so that $\Psi_{k} = \cX \neq \Psi_{k+1}$ and $\dim(\Psi_{k+1}) \geq \dim(\Psi_k)$ & Definition~\ref{def:individual_switches}\\
$k_s^{\mathrm L}(\cX)$, $k_s^{\mathrm R}(\cX)$ & re-indexing of $\lswitch(\cX)$ and $\rswitch(\cX)$ respectively & Definition~\ref{def:individual_switches}\\
$\mathtt{P}_{\mathrm{L}}$, $\mathtt{P}_{\mathrm{R}}$ & payments associated to respectively left and right switching & Equation~\eqref{eq:PL_PR_def}\\
$\hat{k}_{s}^{\mathrm L}(\cX)$ & $k$ between two left switches such that $x_k \not \in \bigC(\cX)$ & Definition~\ref{def:good_switch}\\
$\hat{k}_{s}^{\mathrm R}(\cX)$  & $k$ between two right switches such that $x_k \not \in \bigC(\cX)$ & Definition~\ref{def:good_switch}\\
$k^{{\mathrm{L}}}(\cX, \llbracket \ell, r\rrbracket)$& candidates for a left switch & Equation~\eqref{eq:left_time}\\
$k^{{\mathrm{R}}}(\cX, \llbracket \ell, r\rrbracket)$ & candidates for a right switch & Equation~\eqref{eq:right_time}\\
 \bottomrule
\end{tabular}
\end{center}
\caption{Notation relative to strata selection.}\label{tab:not_strata_select}
\end{table}

In this section, we fix $K \geq 1$ the number of gradient descent steps, $\alpha < \beta < (R+1)^{-1}$, $\gamma \leq \gamma_0$ such that conditions of Lemma~\ref{lem:geom_all} hold and, moreover,
   \begin{equation}\label{eq:final_gamma0_fix}
    \begin{split}
        \boxed{\frac{1}{\gamma_{0}^{1 - (R+1)\beta}}\geq \max\left\{4 \blambda G  , 8  L_2 \tfrac{\blambda^2}{\ulambda}, 2L_1G\bar{\lambda} \right\}\, , \quad \textrm{ and }\quad 3\gamma_0^{\alpha} \leq \mathtt{A}_3 \leq \frac{1}{4} \, .}
    \end{split}
   \end{equation}
   For reader's convenience we have collected all constants and relations between them in Table~\ref{tab:not_const}.

    A \emph{strata selection function} is a mapping
    \begin{equation}\label{eqdef:strata_selection}
        \bPsi = (\Psi_1, \ldots, \Psi_k) : \bbR^{d \times K} \rightarrow \bbX^K\, .
    \end{equation}
    We will occasionally omit the word strata and write selection function instead.
    \begin{definition}[Valid strata selection function]
    \label{def:valid_switch}
    For a given $K\in \bbN$, we say that a selection function $\bPsi$ is valid, if for any $k \in \llbracket 1, K\rrbracket$ it holds that $x_k \in \bigC(\Psi_k) \setminus \smallC_{< \dim(\Psi_k)}$.
\end{definition}

\makeatletter
\newcommand{\subalign}[1]{%
  \vcenter{%
    \Let@ \restore@math@cr \default@tag
    \baselineskip\fontdimen10 \scriptfont\tw@
    \advance\baselineskip\fontdimen12 \scriptfont\tw@
    \lineskip\thr@@\fontdimen8 \scriptfont\thr@@
    \lineskiplimit\lineskip
    \ialign{\hfil$\m@th\scriptstyle##$&$\m@th\scriptstyle{}##$\hfil\crcr
      #1\crcr
    }%
  }%
}
\makeatother

Let us also define the moments of switches from any $\cX$ to a higher (or equal) dimensional strata, these are precisely the times when an additional price, associated to strata changes, is incurred.
\begin{definition}[Individual switchings]
    \label{def:individual_switches}
    Let $\bPsi$ be a valid strata selection function. Let $\cX \in \bbX$, we define
    \begin{align*}
        \lswitch(\cX) &= \{k \in \llbracket 2, K \rrbracket \,:\, {\Psi_{k}} = \cX,\, \Psi_{k-1} \neq \Psi_k \text{ and } \dim(\Psi_k) \leq \dim(\Psi_{k-1})\}\,,\\
        \rswitch(\cX) &= \{k \in \llbracket 1, K-1 \rrbracket \,:\, {\Psi_{k}} = \cX,\, \Psi_{k+1} \neq \Psi_k \text{ and } \dim(\Psi_k) \leq \dim(\Psi_{k+1})\}\,,
    \end{align*}
    the sets of left and right upward switches from $\cX$ respectively.
    We will use the following short-hand notation
    \begin{align*}
        \lswitch(\cX) = \{k_1^{\mathrm L}(\cX) < \ldots < k_{|\lswitch(\cX)|}^{\mathrm L}(\cX)\} \quad\text{and}\quad \rswitch(\cX) = \{k_1^{\mathrm R}(\cX) < \ldots < k_{|\rswitch(\cX)|}^{\mathrm R}(\cX)\}\,.
    \end{align*}
    We will also set $k_0^{\mathrm L}(\cX) = 0$ and  $k_{|\rswitch(\cX)| + 1}^{\mathrm R}(\cX) = K+1$.
\end{definition}

\begin{remark}
    Let us remark some important observations from the above definition, applied to a valid strata-selection function $\bPsi$.
    \begin{enumerate}
        \item Assume that for some $k \in \sqbracket{1}{K}$, we have $\Psi_k = \cX$ and $\Psi_{k-1} = \cX'$ with $\cX \neq \cX'$, then
        \begin{enumerate}
            \item if $\dim(\Psi_{k-1}) > \dim(\Psi_k)$, then $k \in \lswitch(\cX)$ and $k-1 \notin \rswitch(\cX')$;
            \item if $\dim(\Psi_{k-1}) < \dim(\Psi_k)$, then $k \notin \lswitch(\cX)$ and $k-1 \in \rswitch(\cX')$;
            \item if $\dim(\Psi_{k-1}) = \dim(\Psi_k)$, then $k \in \lswitch(\cX)$ and $k-1 \in \rswitch(\cX')$;
        \end{enumerate}
        \item For any $\cX \in \bbX_d$, as long as $\gamma \in (0, \gamma_0)$, it holds that $ \lswitch(\cX), \rswitch(\cX) = \emptyset$.
    \end{enumerate}
\end{remark}

The rationale behind this definition is clear from the following lemma, which states that any valid strata selection function yields some convergence results.
\begin{lemma}
    \label{lem:valid_descent}
    Under the assumptions of Theorem~\ref{thm:main_announced},
     let $(x_1, \ldots, x_K)$ be a gradient descent trajectory for $\gamma \in (0, \gamma_0)$. Assume that $\bPsi$ is a valid selection function. Then,
    \begin{align*}
        \mathtt{A}_1\sum_{k = 1}^K \gamma\|\nabla g_{\Psi_k}(x_k)\|^2 \leq \left({g_{\Psi_1}(x_1) - g_{\Psi_K}(x_{K+1})}\right)+ \sum_{k = 2}^K(g_{\Psi_{k}}(x_{k}) - g_{\Psi_{k-1}}(x_{k})) + {\mathtt{A}_2}K\gamma^{2\alpha + 1}\,.
    \end{align*}
    Moreover,
    \begin{align*}
        \sum_{k = 2}^K(g_{\Psi_{k}}(x_{k}) - g_{\Psi_{k-1}}(x_{k})) \leq 4G\cdot(\mathtt{P}_{\mathrm{L}} + \mathtt{P}_{\mathrm{R}})\,,
    \end{align*}
    with
    \begin{align}\label{eq:PL_PR_def}
        \mathtt{P}_{\mathrm{L}} &\eqdef \sum_{\cX \in \bbX}\sum_{k=2}^{K-1} \dist(x_k, \cX)\ind{k \in \lswitch(\cX)}\quad\text{and}\quad
        \mathtt{P}_{\mathrm{R}} \eqdef \sum_{\cX \in \bbX}\sum_{k=2}^{K-1}\dist(x_k, \cX)\ind{k-1 \in \rswitch(\cX)}\,.
    \end{align}
\end{lemma}
\begin{proof}
    Note that due to our choice of constants, all the claims of Lemma~\ref{lem:geom_all} hold.

    First, note that the validity of the strata selection function implies that $x_k \in \bigC({\Psi_k}) \setminus \smallC_{< \dim(\Psi_k)}$. Since $\|x_{k+1} - x_k\| \leq G\gamma$, and $3 \gamma^{\alpha} \leq \mathtt{A}_3$, Lemma~\ref{lem:geom_all}(\ref{:4}) gives $[x_k, x_{k+1}] \in \cC^*(\Psi_k)$. Thus, $g_{\cX}$ is well-defined on $[x_k, x_{k+1}]$.

   Second, by Corollary~\ref{cor:distance_to_boundary_is_large} and our choice of $\gamma_0$ in~\eqref{eq:final_gamma0_fix}, we have
   \begin{equation*}
   \dist(x_k, \sX_{\dim(\Psi_k)-1}) \geq \frac{\gamma^{1-\beta}}{\gamma_{0}^{1- (R+1)\beta}} \geq \frac{\gamma}{\gamma_{0}^{1- (R+1)\beta}} \geq \gamma \max\left\{4 \blambda G  , 8  L_2 \tfrac{\blambda^2}{\ulambda}, 2L_1G\bar{\lambda} \right\}\, .
   \end{equation*}
   Thus, we are in position to use Lemma~\ref{lem:descent_deterministic}, which yields after summation and thanks to the fact that $x_k \in \bigC({\Psi_k})$
    \begin{align*}
       \mathtt{A}_1\sum_{k = 1}^K\gamma \|\nabla g_{\Psi_k}(x_k)\|^2\leq \sum_{k = 1}^K(g_{\Psi_k}(x_k) - g_{\Psi_{k}}(x_{k+1})) + \mathtt{A}_2 K \gamma^{2\alpha + 1}\,,
    \end{align*}
    The first part of the statement follows after re-indexing.

    To establish the moreover part, we apply Lemma~\ref{lem:distance_different_projection_same_point} below, which establishes that
        \begin{align*}
            g_{\Psi_{k}}(x_{k}) - g_{\Psi_{k-1}}(x_{k})
            &\leq G\|\pi_{\Psi_k}(x_k) - \pi_{\Psi_{k-1}}(x_k)\|\\
            &\leq 4G \big[ \dist(x_k, {\Psi_k})\ind{\dim(\Psi_k) \leq \dim(\Psi_{k-1})} + \dist(x_{k}, {\Psi_{k-1}})\ind{\dim(\Psi_k) \geq \dim(\Psi_{k-1})}\big]\,.
        \end{align*}
        The statement then follows from definition of $\lswitch(\cX)$ and $\rswitch(\cX)$ used to obtain $\mathtt{P}_{\mathrm{L}}$ and $\mathtt{P}_{\mathrm{R}}$ respectively.
\end{proof}

As alluded to earlier in the text, the bound reflects the interplay of three distinct phenomena.
\begin{enumerate}
    \item \textbf{First term} (descent / smooth optimization): Due to per-strata smoothness, this term is the standard one from smooth optimization. It captures the decrease of the objective along the iterates and corresponds to the usual descent argument within a fixed stratum;
    \item 	\textbf{Second term} (strata switching): This term arises from switches between different strata. It measures the discrepancy induced when the iterate $x_k$ is evaluated under two consecutive models $g_{\Psi_{k-1}} $ and $g_{\Psi_k}$, and thus quantifies the cost of moving across strata. Thanks to the second part of the above lemma, this term has a metric nature and is effectively controlled by the proximity to lower-dimensional strata at the moment of change.
    \item \textbf{Third term} (proximity to strata): This term quantifies the effect of being near a stratum rather than exactly on it. It reflects the approximation error due to operating in a neighborhood of a stratum.
\end{enumerate}
\begin{lemma}
    \label{lem:distance_different_projection_same_point}
    Under assumptions of Theorem~\ref{thm:main_announced}, let $(x_1, \ldots, x_K)$ be a gradient descent trajectory for $\gamma \in (0, \gamma_0)$. Assume that $\bPsi$ is a valid selection function. Then, for all $k = 2, \ldots, K$ if $\Psi_k \neq \Psi_{k-1}$
    \begin{align*}
        g_{\Psi_{k}}(x_{k}) - g_{\Psi_{k-1}}(x_{k})
        &\leq G\|\pi_{\Psi_k}(x_k) - \pi_{\Psi_{k-1}}(x_k)\|\\
        &\leq 4G \big[ \dist(x_k, {\Psi_k})\ind{\dim(\Psi_k) \leq \dim(\Psi_{k-1})} + \dist(x_{k}, {\Psi_{k-1}})\ind{\dim(\Psi_k) \geq \dim(\Psi_{k-1})}\big]
    \end{align*}
\end{lemma}
\begin{proof}
    First, note that since $x_{k-1} \in \bigC_{\Psi_{k-1}} \backslash \smallC_{< \dim(\Psi_{k-1})}$ and since $3 \gamma^{\alpha} < \mathtt{A}_3$, by Lemma~\ref{lem:geom_all}(\ref{:5}), $x_{k} \in \cC^{*}(\Psi_{k-1})$. Thus, $g_{\Psi_{k-1}}(x_k)$ is well-defined.

    Since $f$ is $G$-Lipschitz, using triangle's inequality, we obtain
    \begin{align*}
        g_{\Psi_{k}}(x_{k}) - g_{\Psi_{k-1}}(x_{k}) \leq G\|\pi_{\Psi_{k}}(x_{k}) - \pi_{\Psi_{k-1}}(x_k)\| \leq G\left(\dist(x_k, {\Psi_k}) + \dist(x_k, {\Psi_{k-1}})\right)\,.
    \end{align*}
    If $\dim(\Psi_k) = \dim(\Psi_{k-1})$ the statement follows.

   \paragraph{Case 1.} If $j = \dim(\Psi_k) >  \dim(\Psi_{k-1})$ (note that then $r = \rank(\Psi_k) > \rank(\Psi_{k-1})$), then, by definition of a valid selection function, $x_{k} \in \bigC({\Psi_{k}})$
    \begin{align*}
            \dist(x_k, {\Psi_k}) \leq \gamma^{\alpha}
                \min\left\{\gamma^{r\beta},\,\dist(x_k, \sX_{j-1})\right\}
                \,.
    \end{align*}
    Boundary condition on stratification implies \[\dist(x_k, \sX_{j-1}) \leq \dist(x_k, \sX_{\dim(\Psi_{k-1})}) \leq \dist(x_k, {\Psi_{k-1}})\,,\]
    where we used the fact that for any $j' < j$,  $\sX_{j'} \subset \sX_{j-1}$ and that $\cX \subset \sX_{\dim(\cX)}$.

    The combination of the above gives
    \begin{align*}
        \dist(x_k, {\Psi_k}) \leq \gamma^{\alpha} \min\left\{\gamma^{r\beta},\, \dist(x_{k}, {\Psi_{k-1}})\right\} \leq \dist(x_{k}, {\Psi_{k-1}})\,.
    \end{align*}
    \paragraph{Case 2.} If $\dim(\Psi_k) <  \dim(\Psi_{k-1}) = j$, then let $r = \rank(\Psi_{k-1})$. By definition of valid selection function, $x_{k-1} \in \bigC({\Psi_{k-1}}) \setminus \smallC_{<\dim(\Psi_{k-1})}$ and $\|x_k - x_{k-1}\| \leq \gamma G$, thus Lemma~\ref{lem:geom_all}(\ref{:5}) gives
    \begin{align*}
        \dist(x_k, {\Psi_{k-1}})
        &\leq
        3\gamma^\alpha\min\left\{\gamma^{r\beta},\,\dist(x_{k}, \sX_{j-1})\right\}\,.
    \end{align*}
    Boundary condition again gives $\dist(x_{k}, \sX_{j-1}) \leq \dist(x_{k}, \sX_{\dim(\Psi_k)}) \leq \dist(x_{k}, {\Psi_k})$. We conclude identically to the previous case.\end{proof}

\subsubsection{Good strata selection function}
\label{sec:good_selection}
Note that if one could choose a strata selection function that is constant everywhere, the desired rate of convergence would follow immediately from Lemma~\ref{lem:valid_descent}, since the second term vanishes. This is, however, not possible in general: the trajectory may wander away from any a priori fixed $\cX \in \bbX$. A natural approach is therefore to control and ultimately minimize the number of switches.

Now as we have introduced the main objects, let us explain the rationale behind additional properties that will be imposed in Definition~\ref{def:good_switch} below.
As shown in Lemma~\ref{lem:distance_different_projection_same_point}, the cost of a switch is governed by the distance to lower-dimensional strata. This motivates imposing additional structure on the strata selection function:
\begin{enumerate}
    \item \textbf{Controlled upward switches:} a transition from a lower- to a higher-dimensional stratum may only occur upon exiting the thin neighbourhood $\smallC$ of the lower-dimensional stratum.
    \item \textbf{Buffer traversal before reuse:} before reusing the same stratum, the trajectory must exit its large neighbourhood $\bigC$, that is, traverse the buffer zone. This ensures that Lemma~\ref{lem:geom_all}(\ref{:6}) applies, allowing us to bound the total number of switches from any $\cX$ to strata of higher (or equal) dimension.
\end{enumerate}

These requirements are formalized in the next definition.

\begin{definition}[Good strata selection function]\label{def:good_switch}
    We say that a valid selection function is good if for any $\cX \in \bbX$ we have
    \begin{enumerate}
        \item $\forall k \in \lswitch(\cX) \cup \rswitch(\cX)$ it holds that $x_{k} \in \smallC(\cX)$.
        \item $s \in \llbracket 1,  |\lswitch(\cX)|\rrbracket$ there exists $\hat{k}_s^{\mathrm  L }(\cX) \in\,\, \rrbracket k_{s-1}^{\mathrm L}(\cX), k_{s}^{\mathrm L}(\cX) \llbracket$ such that $x_{\hat{k}_s^{\mathrm  L }(\cX)} \notin \bigC(\cX)$.
        \item $s \in \llbracket 1,  |\rswitch(\cX)|\rrbracket$ there exists $\hat{k}_s^{\mathrm  R }(\cX) \in\,\, \rrbracket k_{s}^{\mathrm R}(\cX), k_{s+1}^{\mathrm R}(\cX) \llbracket$ such that $x_{\hat{k}_s^{\mathrm  R }(\cX)} \notin \bigC(\cX)$.
    \end{enumerate}
\end{definition}
The main result of this work follows from Theorem below, that controls the second term in Lemma~\ref{lem:valid_descent} which is associated to strata switching.
\begin{theorem}
    \label{lem:rate_for_good}
    Under assumptions of Theorem~\ref{thm:main_announced}, let $(x_1, \ldots, x_K)$ be a gradient descent trajectory for $\gamma \in (0, \gamma_0)$. Assume that $\bPsi$ is a good selection function. Then,
        \begin{align*}
        {\mathtt{A}_1}\sum_{k = 1}^K \gamma\|\nabla g_{\Psi_k}(x_k)\|^2 \leq {g_{\Psi_1}(x_1) - g_{\Psi_K}(x_{K+1})} + 64|\bbX|G^2K\gamma^{1 + \beta - \alpha} + {\mathtt{A}_2}K\gamma^{1 + 2\alpha}\,.
    \end{align*}
\end{theorem}
\begin{proof}
    By Lemma~\ref{lem:valid_descent} it is sufficient to bound the sum of
    \begin{align*}
        \mathtt{P}_{\mathrm{L}}  {= {\sum_{\cX \in \bbX}}\sum_{k=2}^{K-1} \dist(x_k, \cX)\ind{k \in \lswitch(\cX)}}\quad\text{and}\quad
        \mathtt{P}_{\mathrm R} {={\sum_{\cX \in \bbX}}\sum_{k=2}^{K-1}\dist(x_k, \cX)\ind{k-1 \in \rswitch(\cX)}}\,.
    \end{align*}
    We bound the two terms separately.
        \paragraph{Reduction to number of switches}
        By definition of a good strata selection function $\Psi_k \neq \Psi_{k-1}$, $\dim(\Psi_k) \leq \dim(\Psi_{k-1})$ implies $x_{k} \in \smallC(\Psi_k)$. By Lemma~\ref{lem:geom_all}(\ref{:1}), $\dist(x_k, \Psi_k) \leq \gamma^{(\rank(\Psi_k)+1)\beta}$.
        Hence, pealing the sum over $\bbX$ in $\mathtt{P}_{\mathrm L}$ by corresponding ranks, we get
        \begin{align}
            \label{eq:rate_0}
            \mathtt{P}_{\mathrm L} \leq \sum_{r = 0}^{R-1}\sum_{\cX \in \bbX_{j_r}} \gamma^{(r+1)\beta} |\lswitch(\cX)|\,.
        \end{align}
        Similarly, since $\|x_{k-1} - x_k\| \leq G\gamma$, by Lemma~\ref{lem:geom_all}(\ref{:4}) if $\Psi_k \neq \Psi_{k-1}$, $\dim(\Psi_k) \geq \dim(\Psi_{k-1})$, then $\dist(x_{k-1}, {\Psi_{k-1}}) \leq 3\gamma^{(\rank(\cX)+1)\beta}$, thus
        \begin{align}
            \label{eq:rate_1}
            \mathtt{P}_{\mathrm R} \leq 3\sum_{r = 0}^{R-1}\sum_{\cX \in \bbX_{j_r}} \gamma^{(r+1)\beta} |\rswitch(\cX)|\,.
        \end{align}
        \paragraph{Bounding number of switches for good strata selection function.}
        From the previous paragraph we see that it is sufficient to bound the number of left and right switches for each stratum $\cX \in \bbX$. We establish the bound on $|\lswitch(\cX)|$ and the bound on $|\rswitch(\cX)|$ follows the same lines.

        Fix some $\cX \in \bbX$ with $\rank(\cX) = r$. Let $N = |\lswitch(\cX)|$ and $\lswitch(\cX) = \{k_1^{\mathrm{L}}(\cX), \ldots, k_N^{\mathrm{L}}(\cX)\}$ with $k_0(\cX) = 1$. By Definition~\ref{def:good_switch} of good strata selection function, for every $s = 1, \ldots, N + 1$ there exist $\hat{k}_{s}^{\mathrm{L}}(\cX)$, satisfying
        \begin{align}
            \label{eq:order_times}
           k_{s-1}^{\mathrm{L}}(\cX) < \hat{k}_{s}^{\mathrm{L}}(\cX) < k_{s}^{\mathrm{L}}(\cX)\,,
        \end{align}
        $x_{\hat{k}_{s}^{\mathrm{L}}(\cX)} \notin \bigC(\cX)$ and $x_{k_{s}^{\mathrm{L}}(\cX)} \in \smallC(\cX)$, and we recall the agreement $k_0^{\mathrm{L}}(\cX) = 0$ and $k_{N+1}^{\mathrm{L}}(\cX) = K+1$. Thus, by Lemma~\ref{lem:geom_all}(\ref{:6}), it holds that
        \begin{align*}
            \|x_{\hat{k}_{s}^{\mathrm{L}}(\cX)} - x_{k_{s}^{\mathrm{L}}(\cX)}\| \geq \gamma^{\alpha + r\beta} / 4\,.
        \end{align*}
        Hence, we can write for $\gamma_k = \gamma$ that
        \begin{align}
            \label{eq:rate_2}
            \frac{N\gamma^{\alpha + r\beta}}{4} \leq \sum_{s = 1}^N \|x_{\hat{k}_{s}^{\mathrm{L}}(\cX)} - x_{k_{s}^{\mathrm{L}}(\cX)}\| \leq G\sum_{s = 1}^N \sum^{k_{s}^{\mathrm{L}}(\cX)-1}_{k = \hat{k}_{s}^{\mathrm{L}}(\cX)} \gamma_k \stackrel{(a)}{\leq} G\sum_{s = 1}^N \sum^{k_{s}^{\mathrm{L}}(\cX)-1}_{k = {k}_{s-1}^{\mathrm{L}}(\cX)} \gamma_k \leq G\sum_{k = 1}^{K-1}\gamma_k \leq G\gamma K\,,
        \end{align}
        where $(a)$ follows from Equation~\eqref{eq:order_times}, which is a key requirement in Definition~\ref{def:good_switch}.

        Same argument goes for $|\rswitch(\cX)|$. Thus,
        \begin{equation*}
        4 G ( \mathtt{P}_{\mathrm L} + \mathtt{P}_{\mathrm R}) \leq 4 G (4 G \gamma K + 12 G \gamma K) \sum_{r=0}^{R-1}|\bbX_{j_{r}}|\gamma^{\beta - \alpha}\, ,
        \end{equation*}
        which with Lemma~\ref{lem:valid_descent} completes the proof.
\end{proof}
      It only remains to show that good strata selection function exists. This is the purpose of the next section.

\newcommand\mycommfont[1]{\small\ttfamily\textcolor{Cerulean}{#1}}
\SetCommentSty{mycommfont}
\SetKwRepeat{Do}{do}{while}%
\begin{algorithm}[th]

  \DontPrintSemicolon
  \KwIn{Trajectory $x_{\llbracket 1, K\rrbracket}$}
  \KwOut{Strata assignment $\bPsi$}

  At the beginning $(\Psi_k)_{k = 1}^K$ is not assigned anywhere;

  \For(\tcp*[f]{Assignment starts from the smallest dimension; if $|\bbX_j| = \emptyset$, this dimension will be skipped}){$j = 0, \ldots, d-1$}{
        Let $\{\llbracket \ell_i^j, r_i^j \rrbracket \}_{i=1}^{m_j}$ be the decomposition of indices $k$ for which $\Psi_k$ is not yet defined into maximal consecutive integer intervals, ordered so that $r_i^j < \ell_{i+1}^j$.\label{algo:connected_components1}\;

        \For{$i = 1, \ldots, m$\label{algo:for_loop_entry}}{
                $\llbracket \ell, r\rrbracket \gets \texttt{CheckLeft}({\llbracket \ell_i^j, r_i^j \rrbracket}, j)$\tcp*[r]{Making sure left buffers are crossed}
                \While{$\llbracket \ell, r\rrbracket \neq \emptyset$\label{algo:while_loop_entry}}{
                        $\llbracket \ell, r\rrbracket \gets \texttt{BuildInside}({\llbracket \ell, r\rrbracket}, j)$\tcp*[r]{Build selection function}

                }\label{algo:while_loop_exit}
        }\label{algo:for_loop_exit}

  }
 Every index $k$ that is not yet assigned is set as $\Psi_k = \argmin_{\cX \in \bbX_{d}} \dist(x_k, \cX)$\label{algo:for_last_assignment}\;

  \caption{\texttt{Building strata selection function}}
  \label{algo:main}
\end{algorithm}

\SetKwInput{KwProm}{Promise}
\SetKwInput{KwAss}{Assignment}
\SetKw{KwAssign}{Assign}

\begin{algorithm}[th]
  \DontPrintSemicolon
  \KwIn{Non-empty consecutive interval ${\llbracket \ell, r \rrbracket}$ with $\ell \leq r$, target dimension $j \in \llbracket 0, d-1 \rrbracket$}
\KwAss{Indices $\sqbracket{\ell}{k} \subset \sqbracket{l}{r}$ were assigned to some $j$-dimensional stratum}
  \KwOut{Indices $\sqbracket{k+1}{r}$, that could be still potentially assigned to a $j$-dimensional stratum}

  \If(\tcp*[f]{If no boundary condition, continue}){$\max_{\cX \in \bbX_j}k^{{\mathrm{R}}}(\cX, \llbracket \ell, r\rrbracket) = -1$ from Equation~\eqref{eq:right_time}\label{algline:check_left_kR}} {\Return {\color{YellowOrange}$\llbracket \ell, r \rrbracket$}}
  \Else(\tcp*[f]{If there is boundary condition, take the latest inner cone exit}){
  Set $\cX^* \in \argmax_{\cX \in \bbX_j}k^{{\mathrm{R}}}(\cX, \llbracket \ell, r\rrbracket)$ \tcp*[f]{Break ties arbitrarily}
  }

   \If(\tcp*[f]{check right corner case: do we remain in the outer cone?}){$x_{\llbracket \ell, r\rrbracket} \subset \bigC(\cX^*)$\label{algline:rcorner_l1}}{
   \KwAssign $\Psi_{\llbracket \ell, r\rrbracket} = \cX^*$\;\label{algline:rcorner_l1_2}
   \Return ${\color{YellowOrange}\emptyset}$\label{algline:rcorner_l2}\;
   }
   \Else(\tcp*[f]{If buffer zone was traversed, continue with sub-trajectory}){
    Set $k^{{\mathrm{R}}} \gets \max_{\cX \in \bbX_j}k^{{\mathrm{R}}}(\cX, \llbracket \ell, r\rrbracket)$\label{algo:check_left_rightswitch} \tcp*[r]{Record right-switch}
    \KwAssign $\Psi_{\llbracket \ell, k^{{\mathrm{R}}}\rrbracket} = \cX^*$\;\label{algo:check_left_rightswitch_1}
    \Return ${\color{YellowOrange}\llbracket k^{{\mathrm{R}}} + 1, r \rrbracket}$\;\label{algo:check_left_rightswitch_2}
   }
  \caption{\texttt{Sub-routine}: \texttt{CheckLeft}}
  \label{algo:check_left}
\end{algorithm}

\newpage
\begin{algorithm}[th]
  \DontPrintSemicolon
  \KwIn{Non-empty consecutive interval ${\llbracket \ell, r \rrbracket}$ with $\ell \leq r$, target dimension $j \in \llbracket 0, d-1\rrbracket$}

  \KwAss{Indices $\sqbracket{k_1}{k_2} \subset \sqbracket{l}{r}$ were assigned to some $j$-dimensional stratum}
  \KwOut{Indices $\sqbracket{k_2+1}{r}$ that could be still potentially assigned to a $j$-dimensional stratum}
  \KwProm{$x_{\sqbracket{l}{k_1-1}} \cap \smallC_{=j} = \emptyset$}\label{algo:prom2}

    \If(\tcp*[f]{No entrance in an inner cone? stop!}){$\min_{\cX \in \bbX_j}k^{{\mathrm{L}}}(\cX, \llbracket \ell, r\rrbracket) = +\infty$ from Eq.~\eqref{eq:left_time}}{
        \Return $\color{YellowOrange}{\emptyset}$\;
    }
    \Else(\tcp*[f]{If there is an inner cone entry on $\llbracket \ell, r \rrbracket$, take the earliest}){
    Set $\cX^* \in \argmin_{\cX \in \bbX_j}k^{{\mathrm{L}}}(\cX, \llbracket \ell, r\rrbracket)$\tcp*[r]{Break ties arbitrarily}
    Set $k^{{\mathrm{L}}} \gets \min_{\cX \in \bbX_j}k^{{\mathrm{L}}}(\cX, \llbracket \ell, r\rrbracket)$\tcp*[r]{record left-switch}\label{algo:build_in_leftswitch}
    }

   \If(\tcp*[f]{check right corner case: do we remain in the outer cone?}){$x_{\llbracket k^{{\mathrm{L}}}, r\rrbracket} \in \bigC(\cX^*)$\label{algline:rcorner_l3}}{
   \KwAssign $\Psi_{\llbracket k^{{\mathrm{L}}}, r\rrbracket} = \cX^*$\;\label{algline:rcorner_l4_3}
   \Return $\color{YellowOrange}{\emptyset}$\label{algline:rcorner_l4}\;
   }
   \Else(\tcp*[f]{If buffer zone was traversed, continue with sub-trajectory}){
     Set $k^{{\mathrm{R}}} \gets k^{{\mathrm{R}}}(\cX^*, \llbracket k^{{\mathrm{L}}}, r\rrbracket)$ by Eq.~\eqref{eq:right_time}\label{algline:right_norm1} \tcp*[r]{record right-switch}
    \KwAssign $\Psi_{\llbracket k^{{\mathrm{L}}}, k^{{\mathrm{R}}}\rrbracket} = \cX^*$\;\label{algline:right_corner_assign}
    \Return ${\color{YellowOrange}\llbracket k^{{\mathrm{R}}} + 1, r \rrbracket}$\;\label{algline:right_norm2}
   }

  \caption{\texttt{Sub-routine}: \texttt{BuildInside}}
  \label{algo:build_inside}
\end{algorithm}

\subsection{Existence of a good strata selection function}
\label{sec:goo_existence}

In this section, we present Algorithm~\ref{algo:main}, which given a trajectory constructs a good strata selection function. We emphasize that our approach in this section is entirely combinatorial and deliberately avoids using any specific properties of the underlying dynamics or additional geometric structure. In fact, in this section, the reader may ignore the dynamical and geometric aspects altogether and focus only on the nestedness of the neighborhoods $\smallC(\cX) \subset \bigC(\cX)$.

We use the following notation for the construction of Algorithm~\ref{algo:main}. For a stratum $\cX \in\bbX$ and a consecutive interval $\llbracket \ell, r\rrbracket \subset \llbracket 1, K\rrbracket$,
\begin{align}
    &k^{{\mathrm{L}}}(\cX, \llbracket \ell, r\rrbracket) := \min\{k \in \llbracket \ell, r\rrbracket\,:\, x_k \in \smallC(\cX)\}\label{eq:left_time}\,,\\
    &k^{{\mathrm{R}}}(\cX, \llbracket \ell, r\rrbracket) := \max\{k \in \llbracket \ell, r\rrbracket\,:\,x_{\llbracket \ell, k\rrbracket}\subset \bigC(\cX),\, x_k \in \smallC(\cX)\}\label{eq:right_time}\,,
\end{align}
with agreements that $\min\{\emptyset\} = +\infty$ and $\max\{\emptyset\} = - 1$. As we'll see next, $k^{{\mathrm{L}}}(\cX, \llbracket \ell, r\rrbracket)$ and $k^{{\mathrm{R}}}(\cX, \llbracket \ell, r\rrbracket)$ correspond, respectively, to candidate left and right switching times associated with the stratum $\cX$. It is also convenient to set $\dim(\Psi_0) = \dim(\Psi_{K+1}) = \rank(\Psi_0) = \rank(\Psi_{K+1}) =  -1$.

Before stating the main results related to Algorithm~\ref{algo:main}, we summarize the key principles developed so far, which will guide its construction. Recall that our goal is to establish the existence of a strata selection function satisfying Definition~\ref{def:good_switch}, and hence, in particular, Definition~\ref{def:valid_switch}.
\begin{remark}[Rules of the game]
Validity, as specified in Definition~\ref{def:valid_switch}, imposes the following basic constraint:
\begin{itemize}
    \item If an iterate satisfies $x_k \in \smallC_{=j} \setminus \smallC_{<j}$, then k \emph{must} be assigned to a $j$-dimensional stratum. Conversely, if $x_k \notin \bigC(\cX)$, then $k$ \emph{cannot} be assigned to $\cX$.
\end{itemize}
The additional “goodness” conditions from Definition~\ref{def:good_switch} can be reformulated as follows. Let $\llbracket k_1, k_2 \rrbracket$ be the maximal consecutive interval so that $\Psi_{\llbracket k_1, k_2 \rrbracket} = \cX$ for some $\cX \in \bbX_j$. Then one of the following situations must occur:
\begin{enumerate}
    \item If $\max\{\dim(\Psi_{k_1 - 1}), \dim(\Psi_{k_2 + 1})\} < j$, then, $k_1 \notin \lswitch(\cX), k_2 \notin \rswitch(\cX)$, and we require $x_{k_1}, x_{k_2} \in \bigC(\cX)$;
    \item If $\dim(\Psi_{k_1 - 1}) \geq j > \dim(\Psi_{k_2 + 1})$, then, $k_1 \in \lswitch(\cX), k_2 \notin \rswitch(\cX)$, and we require
    \begin{align*}
         x_{k_1} \in \smallC(\cX), x_{k_2} \in \bigC(\cX)\text{ and }\exists \hat{k}_1 < k_1 \text{ s.t. } x_{\hat{k}_1} \notin \bigC(\cX) \text{ and } \llbracket \hat{k}_1, k_1 \llbracket\, \cap \lswitch(\cX) = \emptyset\,;
    \end{align*}
    \item If $\dim(\Psi_{k_2 + 1}) \geq j > \dim(\Psi_{k_1 - 1})$, then, $k_1 \notin \lswitch(\cX), k_2 \in \rswitch(\cX)$, and we require
    \begin{align*}
         x_{k_1} \in \bigC(\cX), x_{k_2} \in \smallC(\cX)\text{ and }\exists \hat{k}_2 > k_2 \text{ s.t. } x_{\hat{k}_2} \notin \bigC(\cX) \text{ and } \rrbracket k_2, \hat{k}_2 \rrbracket \cap \rswitch(\cX) = \emptyset\,;
    \end{align*}
    \item If $\min\{\dim(\Psi_{k_1 - 1}), \dim(\Psi_{k_2 + 1})\} \geq j$, then, $k_1 \in \lswitch(\cX), k_2 \in \rswitch(\cX)$, and we require
    \begin{align*}
         x_{k_1} \in \smallC(\cX) \text{ and }\exists \hat{k}_1 < k_1 \text{ s.t. } x_{\hat{k}_1} \notin \bigC(\cX) \text{ and } \llbracket \hat{k}_1, k_1 \llbracket\, \cap \lswitch(\cX) = \emptyset\,,\\
         x_{k_2} \in \smallC(\cX)\text{ and }\exists \hat{k}_2 > k_2 \text{ s.t. } x_{\hat{k}_2} \notin \bigC(\cX) \text{ and } \rrbracket k_2, \hat{k}_2 \rrbracket \cap \rswitch(\cX) = \emptyset\,;
    \end{align*}
\end{enumerate}
\end{remark}

The idea behind Algorithm~\ref{algo:main} is to construct the assignment recursively with respect to the dimension of the strata. Fix a dimension $j$, and consider a maximal consecutive interval $\sqbracket{l}{r} \subset \sqbracket{1}{K}$ such that:
\begin{itemize}
    \item no stratum has yet been assigned for $k \in \sqbracket{l}{r}$,
    \item $x_k \notin \smallC_{<j}$ for all $k \in \sqbracket{l}{r}$, and
    \item the boundary indices $l-1$ and $r+1$ have already been assigned to strata of strictly smaller dimension.
\end{itemize}

The construction at level $j$ is performed locally on such intervals $\sqbracket{l}{r}$. Ideally, one would like the following property to hold: for every $\cX \in \bbX_j$, each entrance of the trajectory into $\smallC(\cX)$ within $\sqbracket{l}{r}$ is preceded by an exit from $\bigC(\cX)$, and each exit from $\smallC(\cX)$ is followed by an exit from $\bigC(\cX)$, all within the same interval. Under this assumption, it suffices to run Algorithm~\ref{algo:build_inside}, which constitutes the main inner loop of Algorithm~\ref{algo:main}. In words, this procedure detects the first entrance of $x_k$ into $\smallC_{=j}$, selects the corresponding stratum, and continues assigning it for as long as the admissibility conditions are satisfied.

However, since the interval $\sqbracket{l}{r}$ is arbitrary, the above matching between entrances into $\smallC(\cX)$ and exits from $\bigC(\cX)$ may fail near the boundaries. To handle this, we introduce two corner-case mechanisms:
\begin{enumerate}
    \item \textbf{Left boundary:} since $l-1$ is assigned to a lower-dimensional stratum, we may start an assignment at $l$ even if $x_l \in \bigC_{=j}$, without requiring $x_l \in \smallC_{=j}$. This case is handled by Algorithm~\ref{algo:check_left}.
    \item \textbf{Right boundary:} similarly, since $r+1$ is assigned to a lower-dimensional stratum, we may terminate an assignment at $r$ even if $x_r \notin \smallC(\cX)$. This situation is handled by Lines~\ref{algline:rcorner_l1}--\ref{algline:rcorner_l2} in Algorithm~\ref{algo:check_left} and Lines~\ref{algline:rcorner_l3}--\ref{algline:rcorner_l4} in Algorithm~\ref{algo:build_inside}. At each step, these checks determine whether the current assignment can be extended up to the boundary of the interval.
\end{enumerate}

As we state in the next theorem, Algorithm~\ref{algo:main} produces a strata selection function that satisfies Definition~\ref{def:good_switch}.

\begin{theorem}\label{thm:algo_produce_good}
    Under assumptions of Theorem~\ref{thm:main_announced}, let $(x_1, \ldots, x_K)$ be a gradient descent trajectory for $\gamma \in (0, \gamma_0)$. A good selection function $\bPsi$ exists and can be obtained by Algorithm~\ref{algo:main}.
\end{theorem}

\begin{figure}[t]
  \centering

  \begin{subfigure}{0.48\textwidth}
    \centering
    \includegraphics[width=\linewidth]{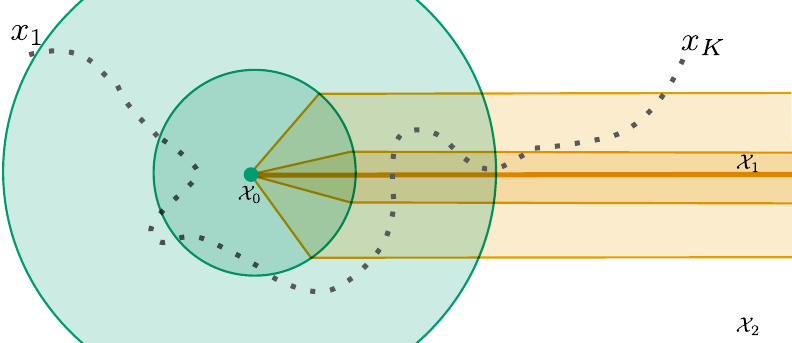}
    \caption{}
    \label{fig:ddd2}
  \end{subfigure}\hfill
  \begin{subfigure}{0.48\textwidth}
    \centering
    \includegraphics[width=\linewidth]{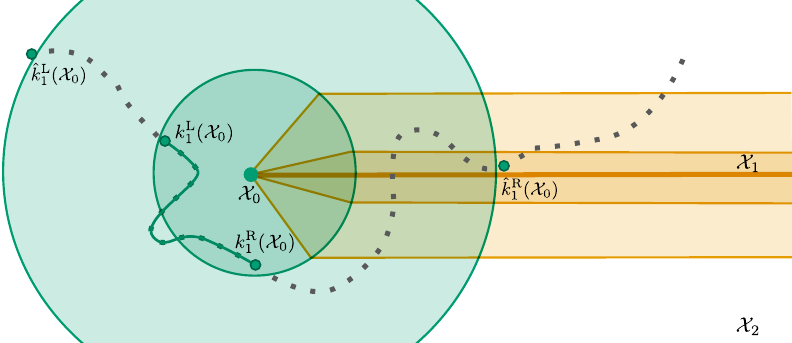}
    \caption{}
    \label{fig:ddd3}
  \end{subfigure}

  \vspace{0.5em}

  \begin{subfigure}{0.48\textwidth}
    \centering
    \includegraphics[width=\linewidth]{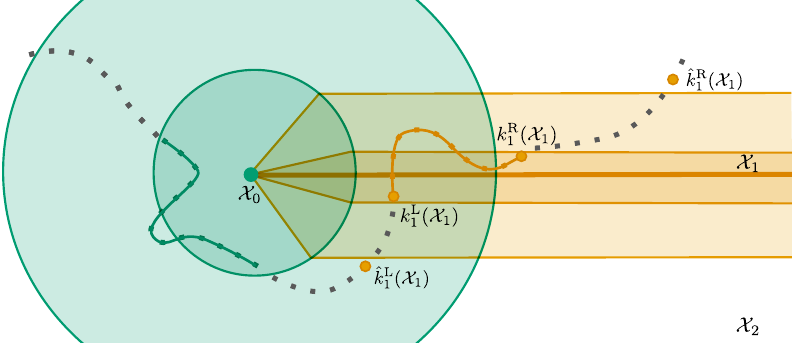}
    \caption{}
    \label{fig:ddd4}
  \end{subfigure}\hfill
  \begin{subfigure}{0.48\textwidth}
    \centering
    \includegraphics[width=\linewidth]{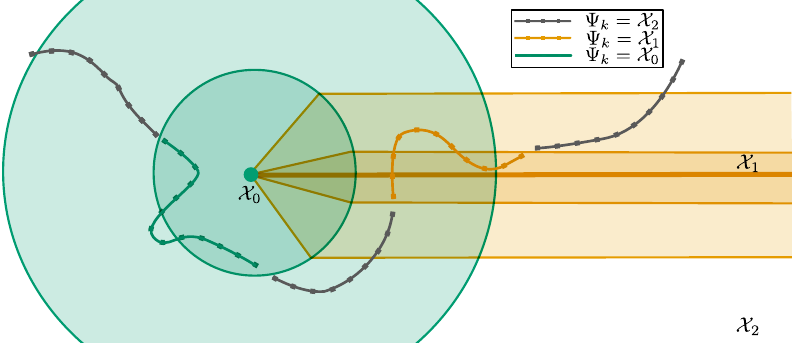}
    \caption{}
    \label{fig:ddd5}
  \end{subfigure}

  \caption{Step-by-step construction of the strata selection function. Figure~\ref{fig:ddd2} displays the full trajectory, traversed from left to right. Figure~\ref{fig:ddd3} shows the outcome of the first step of the main loop in Algorithm~\ref{algo:main}, together with the corresponding switching times and outer-neighbourhood exit times. Figure~\ref{fig:ddd4} displays the second step of the same loop. Finally, Figure~\ref{fig:ddd5} presents the resulting strata selection function; the gaps between consecutive sub-trajectories correspond to strata switchings.}
  \label{fig:four_plots}
\end{figure}

\paragraph{An example of execution of Algorithm~\ref{algo:main}}
Before delving into the proof of Theorem~\ref{thm:algo_produce_good}, postponed to Section~\ref{sec:proof_for_algo}, we first present a simple example illustrating the internal logic of Algorithm~\ref{algo:main}. In the main text, we focus on a setting without corner cases; an example covering such cases is deferred to Appendix~\ref{app:add_details}.

We consider the configuration shown in Figure~\ref{fig:four_plots}. The stratification is given by $\bbX = \{\cX_0, \cX_1, \cX_2\}$, where $\cX_0 = \{(0, 0)\}$, $\cX_1 = \{(x, 0) \,:\, x > 0\}$, and $\cX_2 = \setminus (\cX_0 \cup \cX_1)$. The specific form of the underlying function is irrelevant; only the realized trajectory matters. Figure~\ref{fig:ddd2} displays the three strata together with the corresponding $\smallC, \bigC$, as well as a trajectory naturally evolving from left to right. In particular, Figure~\ref{fig:ddd2} contains all the inputs required by Algorithm~\ref{algo:main}.

The execution of Algorithm~\ref{algo:main} is summarized by three subplots in Figures~\ref{fig:ddd3}–\ref{fig:ddd5}. Each subplot (read from left to right, top to bottom) shows the outcome of one pass of the main inner loop of Algorithm~\ref{algo:main}, corresponding to $j = 0, 1$, as well as the final strata assignment on Line~\ref{algo:for_last_assignment}, where the full-dimensional stratum is assigned. We describe each subplot below.
\begin{enumerate}
    \item Figure~\ref{fig:ddd3} corresponds to the iteration of the main loop with $j = 0$. The algorithm first calls Algorithm~\ref{algo:check_left}, which performs no action since every entrance into $\smallC(\cX_0)$ is preceded by an exit from $\bigC(\cX_0)$. It then calls Algorithm~\ref{algo:build_inside} with $j = 0$. During this step, the algorithm identifies two indices: $k^{\mathrm L}_1(\cX_0)$--the first index $k$ such that $x_k \in \smallC(\cX_0)$--and $k^{\mathrm R}_1(\cX_0)$--the last index $k$ such that $x_k \in \smallC(\cX_0)$, without leaving $\bigC(\cX_0)$ in between. The corresponding sub-trajectory is then assigned to $\cX_0$. The figure also highlights $\hat{k}^{\mathrm L}_1(\cX_0)$ and $\hat{k}^{\mathrm R}_1(\cX_0)$, which denote the last exit from $\bigC(\cX_0)$ before $k^{\mathrm L}_1(\cX_0)$ and the first exit after $k^{\mathrm R}_1(\cX_0)$, respectively, and play a key role in Definition~\ref{def:good_switch}.
    \item Figure~\ref{fig:ddd4} corresponds to the iteration with $j = 1$. The same procedure is applied independently on the intervals $\sqbracket{1}{\hat{k}^{\mathrm L}_1(\cX_0)-1}$ and $\sqbracket{\hat{k}^{\mathrm R}_1(\cX_0)+1}{K}$. On the first interval, no assignment is made, since $\smallC(\cX_1)$ is never entered. On the second interval, a portion of the trajectory is assigned to $\cX_1$, analogously to the previous step. The indices $\hat{k}^{\mathrm L}_1(\cX_1)$ and $\hat{k}^{\mathrm R}_1(\cX_1)$ are also displayed.
    \item Finally, Figure~\ref{fig:ddd5} corresponds to Line~\ref{algo:for_last_assignment} of Algorithm~\ref{algo:main}, where all remaining unassigned indices $k$ are assigned to $\cX_2$.
\end{enumerate}
Note that, when passing from $j$ to $j+1$, the algorithm automatically guarantees that for any index $k$ not yet assigned, we have $x_k \notin \smallC_{\leq j}$, since otherwise it would have already been assigned at a previous step. The existence of $\hat{k}^{\mathrm{L} \text{ or } \mathrm{R}}_s(\cX)$ is likewise ensured by the construction of the algorithm.

Although somewhat cumbersome, the proof of Theorem~\ref{thm:algo_produce_good} essentially consists in systematically listing the properties that follow directly from the definition of Algorithm~\ref{algo:main}.

\subsubsection{Proof of Theorem~\ref{thm:algo_produce_good}}
\label{sec:proof_for_algo}

To prove Theorem~\ref{thm:algo_produce_good} we first record some simple observations.

\begin{lemma}\phantom{=}\label{lm:algo_simple_observ}
    Algorithm~\ref{algo:main} terminates and produces a strata selection function. Moreover, if Algorithm~\ref{algo:check_left} or Algorithm~\ref{algo:build_inside}, with input an interval $\sqbracket{l}{r}$ and a dimension $j\leq d$, \emph{i)} assigns an interval $\sqbracket{k_1}{k_2}$ and \emph{ii)} outputs an interval $\sqbracket{k_2 +1}{r}$, then the following holds.
    \begin{enumerate}
        \item We have $x_{\sqbracket{l}{k_1-1}} \cap \smallC_{=j} = \emptyset$.
    \item We have $x_{\sqbracket{k_1}{k_2}} \subset \bigC(\Psi_{k_1})$.
        \item If $\sqbracket{k_1}{k_2}$ was assigned through Algorithm~\ref{algo:check_left}, then $k_1 = l$. Otherwise, $x_{k_1} \in \smallC(\Psi_{k_1})$.
        \item Either $k_2 = r$, or $x_{k_2} \in \smallC(\Psi_{k_2})$.
    \end{enumerate}
\end{lemma}
\begin{proof}
    To prove termination, it suffices to show that the while loop in Line~\ref{algo:while_loop_entry} cannot run indefinitely. Observe that each call to Algorithm~\ref{algo:build_inside} returns an interval of the form $\sqbracket{k}{r}$ with $k > l$. Hence, the left endpoint strictly increases, and the length of the remaining interval decreases. It follows that the while loop terminates after at most $r - l$ iterations. Finally, the last line of Algorithm~\ref{algo:build_inside} assigns all remaining unassigned indices to a $d$-dimensional stratum, thereby producing a complete strata selection function.

    To prove the remaining claims, we describe the possible assignments produced by Algorithm~\ref{algo:check_left} and Algorithm~\ref{algo:build_inside}.

    \paragraph{If assignment happened in Algorithm~\ref{algo:check_left}.}
    The assignment returned by Algorithm~\ref{algo:check_left} can take two forms.

    First, the assignment may occur at Line~\ref{algline:rcorner_l1_2}. In this case, we assign $\Psi_{\sqbracket{l}{r}} = \cX^*$, with $x_{\sqbracket{l}{r}} \subset \bigC(\cX^*)$ (see Line~\ref{algline:rcorner_l1}). All four required properties are then satisfied immediately.

    Otherwise, the assignment occurs at Line~\ref{algo:check_left_rightswitch_1}. In this case, we assign $\Psi_{\sqbracket{l}{k^{\mathrm{R}}}} = \cX^*$, where
    \[
    k^{\mathrm{R}} = \max_{\cX \in \bbX_j} k^{\mathrm{R}}(\cX, \sqbracket{l}{r}),
    \qquad
    \cX^* = \argmax_{\cX \in \bbX_j} k^{\mathrm{R}}(\cX, \sqbracket{l}{r}).
    \]
    By the definition of $k^{\mathrm{R}}(\cX, \sqbracket{l}{r})$ in \eqref{eq:right_time}, the required properties are satisfied in this case as well.

    \paragraph{If assignment happened in Algorithm~\ref{algo:build_inside}.}
    Similarly, the assignment produced by Algorithm~\ref{algo:build_inside} can take two forms.

    First, the assignment may occur at Line~\ref{algline:rcorner_l4_3}, in which case it is of the form $\Psi_{\sqbracket{k^{\mathrm{L}}}{r}} = \cX^*$, where
    \[
    k^{\mathrm{L}} = \min_{\cX \in \bbX_j} k^{\mathrm{L}}(\cX, \sqbracket{l}{r}),
    \qquad
    \cX^* = \argmin_{\cX \in \bbX_j} k^{\mathrm{L}}(\cX, \sqbracket{l}{r}).
    \]
    In this case, $x_{\sqbracket{k^{\mathrm{L}}}{r}} \subset \bigC(\cX^*)$ (see Line~\ref{algline:rcorner_l3}), and all required properties follow directly from the definition of $k^{\mathrm{L}}(\cX, \sqbracket{l}{r})$ in \eqref{eq:left_time}.

    Otherwise, the assignment occurs at Line~\ref{algline:right_corner_assign} and is of the form $\Psi_{\sqbracket{k^{\mathrm{L}}}{k^{\mathrm{R}}}} = \cX^*$, where $k^{\mathrm{L}}$ is as above and $k^{\mathrm{R}} = k^{\mathrm{R}}(\cX^*, \sqbracket{k^{\mathrm{L}}}{r})$. Compared to the previous case, the only difference concerns the fourth property: here, by definition of $k^{\mathrm{R}}$, we have $x_{k^{\mathrm{R}}} \in \smallC(\cX^*)$, as required.
\end{proof}

\begin{lemma}
    The selection function produced by Algorithm~\ref{algo:main} is valid.
\end{lemma}
\begin{proof}
    Thanks to the first point of Lemma~\ref{lm:algo_simple_observ}, a straightforward induction shows that the consecutive intervals $\sqbracket{l_i^j}{r_i^j}$ appearing on Line~\ref{algo:connected_components1} of Algorithm~\ref{algo:main} satisfy
    \[
    x_{\sqbracket{l_i^j}{r_i^j}} \cap \smallC_{<j} = \emptyset.
    \]
    Moreover, by the fourth point of Lemma~\ref{lm:algo_simple_observ}, every index $k \in \sqbracket{l_i^j}{r_i^j}$ that has already been assigned satisfies $x_k \in \bigC(\Psi_k)$. This concludes the proof.
\end{proof}

To show that the produced selection function is good we first state some observations on produced switches.

\begin{lemma}
    Let $\sqbracket{l}{r}$ be the interval from which started the inner loop of Algorithm~\ref{algo:main} at Line~\ref{algo:while_loop_entry} and $\sqbracket{k_1}{k_2}$ be an interval assigned in this inner loop through Algorithm~\ref{algo:check_left} or Algorithm~\ref{algo:build_inside}. Denote $\cX$ the stratum corresponding to this assignment. The following holds.

    \begin{enumerate}
        \item $k_1 \in \lswitch(\cX)$ if and only if $k_1 \neq l$.
        \item $k_2 \in \rswitch(\cX)$ if and only if $k_2 \neq r$.
        \item If $\sqbracket{k_1}{k_2}$ was assigned through Algorithm~\ref{algo:build_inside}, then there is $\hat{k} \subset \sqbracket{l}{k_1}$ such that $x_{\hat{k}} \not \in \bigC(\cX)$.
        \item If the assignment took place either at Lines~\ref{algo:check_left_rightswitch}--\ref{algo:check_left_rightswitch_2} of Algorithm~\ref{algo:check_left} or at Lines~\ref{algline:right_norm1}--\ref{algline:right_norm2} of Algorithm~\ref{algo:build_inside}. Then, there is $\hat{k} \in \sqbracket{k_2}{r}$ such that $x_{\hat{k}} \not \in \bigC(\cX)$.
    \end{enumerate}
\end{lemma}
\begin{proof}
    The first two points follow from the fact that $l-1$ and $r+1$ were assigned to a stratum of dimension lower than $\dim(\cX)$.

    For the third point, we note that since $x_{k_1} \in \smallC(\cX)$, if $x_{\sqbracket{l}{k_1}} \subset \bigC(\cX)$, then during the run of Algorithm~\ref{algo:check_left} with input $\sqbracket{l}{r}$, $k^{\mathrm{R}}$ of Line~\ref{algline:check_left_kR} would be necessarily larger than $k_1$. Thus, in this case, $\sqbracket{l}{k_1}$ would be assigned by Algorithm~\ref{algo:check_left} contradicting the assumption.

The fourth point follows simply for the fact that if $x_{\sqbracket{k_2}{r}} \subset \bigC(\cX)$, then either by Lines~\ref{algline:rcorner_l1}--\ref{algline:rcorner_l2} of Algorithm~\ref{algo:check_left} or by Lines~\ref{algline:rcorner_l3}--\ref{algline:rcorner_l4} of Algorithm~\ref{algo:build_inside}, the whole interval $\sqbracket{k_2}{r}$ would be assigned to $\cX$.

Finally, for the last point assume for the sake of contradiction that $\sqbracket{k_1}{k_2}$ and $\sqbracket{k_3}{k_4}$ were assigned through the same inner loop and $x_{\sqbracket{k_1}{k_4}} \subset \bigC(\cX)$. We note that in this case, $\sqbracket{k_1}{k_2}$ was assigned either on Lines~\ref{algo:check_left_rightswitch}--\ref{algo:check_left_rightswitch_1} of Algorithm~\ref{algo:check_left} or on Lines~\ref{algline:right_norm1}--\ref{algline:right_norm2} of Algorithm~\ref{algo:build_inside}. In both cases, $k^{\mathrm{R}}(\cX, \sqbracket{k_1}{r}) = k_2$, which implies that $x_{k_2 +1} \not \in \bigC$.
\end{proof}
\begin{lemma}
    The selection function is good.
\end{lemma}
\begin{proof}
    Proofs for left and right switches are carried out separately and are very similar.

    First, suppose that $\sqbracket{l_s}{r_s} = \sqbracket{l_{s+1}}{r_{s+1}}$. We claim that there exists $\hat{k} \in \sqbracket{k_s^{\mathrm{L}}}{k_{s+1}^{\mathrm{L}}}$ such that $x_{\hat{k}} \notin \bigC(\cX)$. Indeed, if this were not the case, then
    \[
    x_{\sqbracket{k_s^{\mathrm{L}}}{k_{s+1}^{\mathrm{L}}}} \subset \bigC(\cX),
    \]
    and therefore
    \[
    k^{\mathrm{R}}(\cX,\sqbracket{k_s^{\mathrm{L}}}{r_s}) \geq k_{s+1}^{\mathrm{L}}.
    \]
    But then, at Line~\ref{algline:right_norm1}, the algorithm would have assigned at least the interval $\sqbracket{k_s^{\mathrm{L}}}{k_{s+1}^{\mathrm{L}}}$ to $\cX$, contradicting the fact that $k_s^{\mathrm{L}}$ and $k_{s+1}^{\mathrm{L}}$ are two distinct consecutive left switches.

    Now suppose that $\sqbracket{l_s}{r_s} \neq \sqbracket{l_{s+1}}{r_{s+1}}$. Then $k_{s+1}^{\mathrm{L}}$ is assigned during the inner loop corresponding to a different interval, namely $\sqbracket{l_{s+1}}{r_{s+1}}$. By the previous lemma, there exists $\hat{k} \in \sqbracket{l_{s+1}}{k_{s+1}^{\mathrm{L}}}$ such that $x_{\hat{k}} \notin \bigC(\cX)$. Since necessarily $k_s^{\mathrm{L}} < r_s < l_{s+1}$, this yields
    \[
    \hat{k} \in \sqbracket{k_s^{\mathrm{L}}}{k_{s+1}^{\mathrm{L}}},
    \]
    and proves the claim for left switches.

        \paragraph{Right switches.}

        Similarly, let $k_s^{\mathrm{R}}$ and $k_{s+1}^{\mathrm{R}}$ be two consecutive right switches of the stratum $\cX$. Let $\sqbracket{l_s}{r_s}$ and $\sqbracket{l_{s+1}}{r_{s+1}}$ denote the intervals from which the inner loop of Algorithm~\ref{algo:main}, started at Line~\ref{algo:while_loop_entry}, assigned $\Psi_{k_s^{\mathrm{R}}} = \cX$ and $\Psi_{k_{s+1}^{\mathrm{R}}} = \cX$, respectively. As before, we distinguish two cases.

        First, suppose that $\sqbracket{l_s}{r_s}=\sqbracket{l_{s+1}}{r_{s+1}}$. We claim that there exists $\hat{k}\in \sqbracket{k_s^{\mathrm{R}}}{k_{s+1}^{\mathrm{R}}}$ such that $x_{\hat{k}} \notin \bigC(\cX)$. Indeed, if
        \[
        x_{\sqbracket{k_s^{\mathrm{R}}}{k_{s+1}^{\mathrm{R}}}} \subset \bigC(\cX),
        \]
        then the value of $k^{\mathrm{R}}$ computed either at Line~\ref{algo:check_left_rightswitch} of Algorithm~\ref{algo:check_left} or at Line~\ref{algline:right_norm1} of Algorithm~\ref{algo:build_inside} would necessarily be at least $k_{s+1}^{\mathrm{R}}$, contradicting the fact that $k_s^{\mathrm{R}}$ and $k_{s+1}^{\mathrm{R}}$ are two distinct consecutive right switches. Hence such a $\hat{k}$ must exist.

        Now suppose that $\sqbracket{l_s}{r_s}\neq \sqbracket{l_{s+1}}{r_{s+1}}$. Then, by the previous lemma, there exists $\hat{k}\in \sqbracket{k_s^{\mathrm{R}}}{r_s}$ such that $x_{\hat{k}} \notin \bigC(\cX)$. Since $\sqbracket{k_s^{\mathrm{R}}}{r_s}\subset \sqbracket{k_s^{\mathrm{R}}}{k_{s+1}^{\mathrm{R}}}$, this proves the claim for right switches as well.
\end{proof}

\section{Extensions: varying step-sizes and sequential convergence}\label{sec:extensions}
This section presents two immediate extensions of the developed framework.

First, our analysis has so far focused on the constant step-size setting. In contrast, standard convergence results for subgradient methods applied to semialgebraic functions are typically established in the decreasing step-size regime~\cite{dav-dru-kak-lee-19}. In Section~\ref{sec:varying}, we show how to extend our results to the varying step-size setting.

Second, in Section~\ref{sec:KL_conv}, we establish sequential convergence of the iterates for step-sizes of the form $\gamma_k = 1/k$, thereby recovering, via an alternative argument, the recent result of Lai and Song~\cite{lai2025diameter}.

\subsection{Varying step-sizes}\label{sec:varying}

In this section, under the same assumptions, we study the recursion
\begin{align}\label{eqdef:subg_desc_decr}
x_{k+1} = x_k - \gamma_k v_k,,
\end{align}
where $v_k \in \partial f(x_k)$ and $(\gamma_k)_{k \geq 1}$ is a sequence of step-sizes.
{For the proofs of this section we do not track precise values of constants.}

We work under the following assumption.
\begin{assumption}\label{ass:decr_steps}
    The sequence $(\gamma_k)$ is decreasing and for all $k \geq 1$, $\gamma_k \leq \gamma_0$.
\end{assumption}
For this extension, unlike the main proofs, we will not give explicit expression $\gamma_0$, which needs to be at most a multiplicative constant factor smaller than the one used in Equation~\eqref{eq:final_gamma0_fix}.

The proof relies on a standard doubling trick applied to the step-size. Fix $K\in \bbN$, $k(1)>1$ and while $k(I)<K$, define the $k(i+1)$ as
\begin{align}\label{eqdef:ki}
k(i+1) = {\min\left\{\min\{k \geq k(i) \,:\, \gamma_{k} \leq \gamma_{k(i)} / 2\},\, K\right\}}\,.
\end{align}
Since $(\gamma_k)$ is decreasing, it follows that for all $k \in \sqbracket{k(i)}{k(i+1)-1}$,
\[\frac{\gamma_{k(i)}}{2} < \gamma_k \le \gamma_{k(i)}\,,\]
that is, the step-sizes within each interval $\sqbracket{k(i)}{k(i+1)-1}$ differ by at most a factor of two.

The proof proceeds as follows. On each interval $\sqbracket{k(i)}{k(i+1) - 1}$, we construct a strata selection function $\bPsi$ using the constant step-size argument, but with a mild change in the definitions of $\smallC$, $\bigC$. For $\cX \in \bbX_j$,
\begin{equation}\label{eqdef:big_small_C_var}
    \begin{split}
\bigC_{i}(\cX) &\eqdef \left\{x \in \bbR^d \,:\, \dist(x, \cX) \leq \gamma_{k(i+1)}^\alpha \min\left\{\gamma_{k(i+1)}^{\rank(\cX)\beta},\, \dist(x, \sX_{\dim(\cX)-1})\right\}\right\}\,,\\
\smallC_{i}(\cX) &\eqdef \left\{x \in \bbR^d \,:\, \dist(x, \cX) \leq \gamma_{k(i+1)}^\beta \min\left\{\gamma_{k(i+1)}^{\rank(\cX)\beta},\, \dist(x, \sX_{\dim(\cX)-1})\right\}\right\}\,.
    \end{split}
\end{equation}
Thus, for all $k \in \sqbracket{k(i)}{k(i+1) - 1}$, the {$\smallC$ and $\bigC$} are fixed and do not depend on $k$. This allows us to apply the results of the previous section, which were derived for the constant step-size case.
We say that $\bPsi$ is defined on $\sqbracket{k(i)}{k(i+1) - 1}$ if these {conical neighborhoods} are used in its construction.

\begin{definition}[Valid and {good} strata selection functions, varying stepsizes]\label{def:val_good_var}
    We say that a strata selection function $\bPsi$ is valid (respectively good) if Definition~\ref{def:valid_switch} (respectively \ref{def:good_switch}) holds for each $i <I$ and $\bPsi: \bbR^{d \times \sqbracket{k(i)}{k(i+1) - 1}}$, where $\smallC$ and $\bigC$ are defined through Equation~\eqref{eqdef:big_small_C_var}.
\end{definition}
\begin{lemma}
    \label{lem:payments_varying}
    Under Assumptions~\ref{ass:f_lip},\ref{ass:main} and \ref{ass:decr_steps}, let $\bPsi$ be a good strata selection function. There exists a constant $C > 0$ such that for any $k(1) > 1$, it holds that
    \begin{align*}
        \sum_{k = k(1)}^{K} \left(g_{\Psi_k}(x_{k}) - g_{\Psi_{k-1}}(x_k)\right)
        &\leq  G\sum_{k = k(1)}^{K}\|\pi_{\Psi_{k}}(x_k) - \pi_{\Psi_{k-1}}(x_k)\|\\
        &\leq C \left(G^2 |\bbX| \sum_{k = k(1)}^K \gamma_k^{1 + \beta - \alpha} + G\gamma_{k(1)}^{\beta}\right)\,.
    \end{align*}
\end{lemma}
\begin{proof}
    Let $\Delta_{k-1} \eqdef g_{\Psi_k}(x_{k}) - g_{\Psi_{k-1}}(x_k)$. One can write
    \begin{align*}
        \sum_{k = k(1)}^{K} \Delta_{k-1} \leq \sum_{i = 1}^{I - 1} \left(\sum_{k = k(i)}^{k(i+1) - 2} \Delta_{k} + \Delta_{k(i+1) - 1}\right) = \sum_{i = 1}^{I-1}\sum_{k = k(i)}^{k(i+1) - 2} \Delta_{k} + \sum_{i = 1}^{I-1}\Delta_{k(i+1) - 1}\,.
    \end{align*}
    Fix some $i \in \sqbracket{1}{I}$, then similarly to Lemma~\ref{lem:distance_different_projection_same_point} and thanks to the assumption that $\Psi$ is good (in particular it is valid globally, which is the only requirement for this result), for any $k \in \sqbracket{k(i)}{k(i+1) - 1}$, we have
    \begin{equation}
        \label{eq:varying1}
    \begin{aligned}
        \Delta_{k} \leq G\|\pi_{\Psi_{k+1}}(x_{k+1}) - \pi_{\Psi_k}(x_{k+1})\| \leq 4G\bigg[&\dist(x_{k+1}, \Psi_{k+1})\ind{\Psi_{k+1} \neq \Psi_{k},\,\dim(\Psi_{k+1}) \leq \dim(\Psi_k)} \\
        &+ \dist(x_{k+1}, \Psi_{k})\ind{\Psi_{k+1} \neq \Psi_{k},\,\dim(\Psi_{k+1}) \geq \dim(\Psi_k)}\bigg]\,.
    \end{aligned}
    \end{equation}
   Moreover, as $\Psi$ is valid, we deduce from {Lemma~\ref{lem:geom_all}(\ref{:2})} that for some constant $C>0$, that can vary from one inequality to another,
    \begin{align*}
        \sum_{i = 1}^{I - 1}\Delta_{k(i+1) - 1} \leq C G \sum_{i = 1}^{I-1}\left(\gamma^{\beta}_{k(i+1)} + \gamma_{k(i+1)}^\beta\right) \leq CG\gamma_{k(1)}^\beta \sum_{i  = 1}^{\infty} 2^{-i\beta} \leq CG\gamma_{k(1)}^\beta\,.
    \end{align*}
    Thus, it remains to bound
    \begin{align*}
        \mathtt{T}_i \eqdef \bigg[&\sum_{k = k(i) + 1}^{k(i+1) - {1}}\dist(x_k, {\Psi_k})\ind{\Psi_k \neq \Psi_{k-1},\,\dim(\Psi_{k}) \leq \dim(\Psi_{k-1})}\\
        &+ \sum_{k = k(i) + 1}^{k(i+1) - {1}}\dist(x_k, {\Psi_{k-1}})\ind{\Psi_k \neq \Psi_{k-1},\,\dim(\Psi_k) \geq \dim(\Psi_{k-1})}\bigg]\,.
    \end{align*}
    By the assumption of goodness of $\Psi$ on each $\sqbracket{k(i)}{k(i+1) - 1}$, similar proof as in Theorem~\ref{lem:rate_for_good}, gives for some absolute constant $C > 0$ that
    \begin{align*}
        \mathtt{T}_i \leq CG|\bbX|\sum_{k = k(i)}^{k(i+1)-1}\gamma_{k}^{1 + \beta - \alpha}\,.
    \end{align*}
    To get the above result one follows exactly the same steps as in the proof of Theorem~\ref{lem:rate_for_good}. That is, establishing Equations \eqref{eq:rate_0}, \eqref{eq:rate_1} and \eqref{eq:rate_2}, additionally using the fact that for any $k \in \sqbracket{k(i)}{k(i+1) - 1}$ and any $\beta - \alpha > 0$ we have $\gamma_{k(i+1)}^{\beta - \alpha} \leq 2^{\beta - \alpha} \gamma_k^{\beta - \alpha}$.
\end{proof}

\begin{corollary}\label{cor:var_steps}
    Under the conditions of Theorem~\ref{thm:main_announced} and with step-sizes $(\gamma_k)_{k \in \sqbracket{1}{K}}$ satisfying
        Assumption~\ref{ass:decr_steps}
    There exists a strata selection function $\bPsi$ such that
    \begin{align*}
        \dist(x_k, \Psi_k) \leq C \gamma_k^{\alpha + \rank(\Psi_k)\beta}\,
    \end{align*}
    and
    \begin{align*}
        \mathtt{A}_1\sum_{k = 1}^{K} \gamma_k \norm{\nabla g_{\Psi_k}(x_k)}^2 \leq (g_{\Psi_1}(x_1) - g_{\Psi_K}(x_{K + 1})) + C\sum_{k = 1}^{K}\left(G^2 |\bbX|\gamma_{k}^{1 + \beta - \alpha} + \mathtt{A_2}\gamma_k^{1 + 2\alpha}\right) + CG\,,
    \end{align*}
    for some absolute contant $C > 0$.
\end{corollary}
\begin{proof}
    Existence of a good strata selection function follows from Theorem~\ref{thm:algo_produce_good}, applied to each sub-interval $\sqbracket{k(i)}{k(i+1) - 1}$. Then, the proof follows the one of Lemma~\ref{lem:valid_descent} using the argument on each sub-interval and using Lemma~\ref{lem:payments_varying}.
\end{proof}
The interpretation of the rate carries over directly. Specifically, the first term captures ``pure descent,'' the two terms in the second correspond respectively to the stratum selection mechanism and the proximity error to the stratum, and the final additive constant arises from our use of the doubling trick.

\subsection{Convergence of iterates}\label{sec:KL_conv}
We now turn to the question of sequential convergence. It is well known that,
even in the smooth setting, a descent lemma alone is insufficient to guarantee
convergence (see \cite{absil2005convergence} for a counterexample), and that a
Kurdyka--{\L}ojasiewicz (KL) inequality is typically required \cite{attouch2013convergence}. In the nonsmooth
setting, \cite{rios2022examples} further shows that even the KL inequality
alone may not suffice to ensure convergence.

Recently, \cite{lai2025diameter} established convergence of the subgradient
method with step-sizes decreasing as $1/k$. In this section, we show that this
result can be recovered using the framework developed in the present work.

\begin{theorem}\label{thm:KL_conv}
    Assume that there is $c_1, c_2>0$ such that $c_1 k^{-1}\leq \gamma_k \leq c_2 k^{-1}$, the function $f: \bbR^d \rightarrow \bbR$ is semialgebraic and that the iterates $(x_k)$, satisfying Equation~\eqref{eqdef:subg_desc_decr} remain in a compact set. Then, the iterates converge.
\end{theorem}

To prove Theorem~\ref{thm:KL_conv} we first notice some standard facts. First, we know the functional convergence of the subgradient method.

\begin{proposition}[{\cite[Corollary 5.9]{dav-dru-kak-lee-19}}]
    Under the assumptions of Theorem~\ref{thm:KL_conv}, $(f(x_k))$ converges.
\end{proposition}
Without loss of generality, we assume in the rest of this section that $f(x_k) \rightarrow 0$.

By assumptions there is $r>0$ such that $\norm{x_k} \leq r$. Fix $\cK$ be the closed ball of radius $r$. Since $\cK$ is semialgebraic, by Theorem~\ref{thm:good_assumption}, there is a definable stratification $\bbX$ of $\cK$, for which Assumption~\ref{ass:main} holds. For the rest of this section let $\bbX$ be this stratification.
In the classical smooth setting, convergence of the iterates follows from the convergence of $\sum_{k} \gamma_k \|v_k\|$. Indeed, convergence of this series implies that the trajectory has bounded length. In our setting, however, an additional term arises due to the strata selection mechanism. The next lemma identifies the relevant quantity whose convergence guarantees convergence of the iterates. In the next lemma we say that $\bPsi:\bbR^d \times \bbN \rightarrow \bbX$ is valid if it satisfies Definition~\ref{def:val_good_var} restricted to any subinterval $\sqbracket{1}{K}$.
\begin{lemma}\label{lm:KL_quanti}
    Let $\Psi_k$ be a valid selection function. Assume that
    \begin{equation*}
      \sum_{k=1}^{+\infty}\norm{\pi_{\Psi_{k+1}}(x_{k+1})- \pi_{\Psi_k}(x_{k+1})}+  \sum_{k=1}^{+\infty} \gamma_k \norm{\nabla g_{\Psi_{k}}(x_k)} < + \infty \, .
    \end{equation*}
    Then $(x_k)$ converges.
\end{lemma}
\begin{proof}
Similarly to the proof of Lemma~\ref{lem:descent_deterministic} (see Equation~\eqref{eq:descent0}), and using the fact that $x_k \in \bigC(\Psi_k)\backslash \smallC_{< \Psi_k}$, there is $C>0$, such that
    \begin{equation*}
        \norm{\pi_{\Psi_k}(x_{k+1}) - \pi_{\Psi_{k}}(x_k)} \leq C \gamma_k \norm{\nabla g_{\Psi_k}(x_k)} + C \gamma_k^{1+\alpha}\, .
    \end{equation*}
    Therefore,
    \begin{equation*}
        \begin{split}
      \sum_{k=1}^{+\infty} \norm{\pi_{\Psi_{k+1}}(x_{k+1}) - \pi_{\Psi_{k}}(x_k)}
      &\leq    \sum_{k=1}^{+\infty}\norm{\pi_{\Psi_{k+1}}(x_{k+1}) - \pi_{\Psi_{k}}(x_{k+1})} +    \sum_{k=1}^{+\infty}  \norm{\pi_{\Psi_k}(x_{k+1}) - \pi_{\Psi_{k}}(x_k)} < + \infty \, .
        \end{split}
    \end{equation*}
    This implies that $(\pi_{\Psi_k}(x_k))$ is a Cauchy sequence and thus converges. Since $x_k \in \bigC(\Psi_k)$, $\norm{x_k - \pi_{\Psi_k}(x_k)} \leq\gamma_k^{\alpha + \rank(\Psi_k) \beta}$, which implies convergence of $x_k$.
\end{proof}

For a good selection function $\Psi$, Lemma~\ref{lem:payments_varying}, immediately yields convergence of the first term. The second term is handled via a standard KL-inequality argument, with mild adjustments as in~\cite{lai2025diameter}. Notably, this is the only point where the KL inequality is invoked. This is a consistent with the smooth case, where first a descent lemma is established and only then a KL inequality is required to prove convergence \cite{attouch2013convergence}. In the remainder of this section, we develop all the elements needed to control the second term in Lemma~\ref{lem:payments_varying}.

First, since $f$ and any $\cX \in \bbX$ are definable we know by \cite{lojasiewicz1982trajectoires,kurdyka1998gradients} that $f_{|\cX}$ satisfies a {\L}ojasiewicz inequality.
\begin{proposition}[\cite{lojasiewicz1982trajectoires,kurdyka1998gradients}]
    There is $\delta, C>0$ and $\theta \in (0,1)$, such that if  $x \in \cX$ and $0 < |f(x)| < \delta$, then
    \begin{equation*}
        \varphi'(f(x)) \norm{\nabla_{\cX}f(x)} \geq C\, ,
    \end{equation*}
    with $\varphi :z \mapsto \sign(z)|z|^{\theta}$.
\end{proposition}
Similar inequality can be established for $f \circ \pi_{\cX}$.
\begin{corollary}\label{cor:KL_g}
    There is $C>0$ such that if $x \in \bigC(\cX)$ and $0 <|g_{\cX}(x)| < \delta$, then
  \begin{equation*}
    \varphi'(g_{\cX}(x)) \norm{\nabla g_{\cX}(x)} \geq C \, .
  \end{equation*}
\end{corollary}
\begin{proof}
    For $x \in \bigC(\cX)$, denote $y = \pi_{\cX}(x)$. We have
    $\nabla_{\cX} g(x) = (\Jac \pi_{\cX}(x))^{\top}\nabla_{\cX}f(y)$ and $g_{\cX}(x) = f(y)$. Therefore, by Lemma~\ref{lem:projection}
    \begin{equation*}
        \varphi'(g_{\cX}(x)) \norm{\nabla g_{\cX}(x)} = \varphi'(f(y)) \norm{\nabla g_{\cX}(x)} \geq \ulambda \varphi'(f(y))  \norm{\nabla_{\cX}f(y)} \geq C \ulambda \, .
    \end{equation*}
\end{proof}
We also restate the decrease given in Section~\ref{sec:varying}.
\begin{lemma}\label{lem:kl_ass}
    Let $\bPsi$ be a good selection function $\bPsi$. There is $\ttB_1 >0$ and $C>0$ such that for all $k$,
\begin{equation*}
    (g_{\Psi_{k+1}}(x_{k+1}) + a_{k+1}) - (g_{\Psi_{k}}(x_k) + a_k) \leq - C \gamma_k \norm{\nabla g_{\Psi_{k}}(x_k)}^{2}\, ,
\end{equation*}
where
\begin{equation*}
    a_{k} = \ttB_1\sum_{k'=k+1}^{+\infty}\gamma_{k'}^{1 +2 \alpha} + \sum_{k'=k+1}^{+\infty} (g_{\Psi_{k'+1}}(x_{k'+1}) - g_{\Psi_{k'}}(x_{k'+1}))\,,
\end{equation*}
and where
\begin{equation*}
   \sum_{k'=k+1}^{+\infty} (g_{\Psi_{k'+1}}(x_{k'+1}) - g_{\Psi_{k'}}(x_{k'+1})) \leq C\sum_{k'=k+1}^{+\infty}\gamma_{k'}^{1+ \beta - \alpha}\,.
\end{equation*}
\end{lemma}
\begin{proof}
    Since $\bPsi$ is a good selection function, proceeding as in Lemma~\ref{lem:descent_deterministic}, and using the fact that $x_k \in \bigC(\cX) \backslash \smallC_{<\dim(\cX)}$,
    \begin{equation*}
        \begin{split}
        g_{\Psi_{k+1}}(x_{k+1}) - g_{\Psi_{k}}(x_k) &=  g_{\Psi_{k+1}}(x_{k+1}) - g_{\Psi_{k}}(x_{k+1}) + g_{\Psi_{k}}(x_{k+1})- g_{\Psi_{k}}(x_k)\\
        &\leq - \gamma_k \mathtt{A}_1 \norm{\nabla g_{\Psi_k}(x_k)}^2 + \mathtt{A}_2\gamma_k^{1+2 \alpha}  + g_{\Psi_{k+1}}(x_{k+1}) - g_{\Psi_{k}}(x_{k+1})\, .
        \end{split}
    \end{equation*}
    Moreover, applying Lemma~\ref{lem:payments_varying},
    \begin{equation*}
        \sum_{k'=k}^{+\infty}(g_{\Psi_{k'+1}}(x_{k'+1}) - g_{\Psi_{k'}}(x_{k'+1}))\leq C \gamma_{k}^{\beta} + C \sum_{k'=k}^{+\infty} \gamma_{k'}^{1+ \beta - \alpha} \leq C \sum_{k'=k}^{+\infty} \gamma_{k'}^{1+\beta} +  C \sum_{k'=k}^{+\infty} \gamma_{k'}^{1+ \beta - \alpha}\, ,
    \end{equation*}
    where for the second inequality we have used $\gamma_k =O(1/k)$. Since $\gamma_k^{1+ \beta}\leq \gamma_k^{1+\beta- \alpha}$, the result follows.
\end{proof}
We are now in position to present the result that controls the second term in Lemma~\ref{lem:payments_varying}.
\begin{lemma}
    Let $\bPsi$ be a good selection function. Denoting $h_k :=g_{\Psi_k}(x_k) + a_k$, there is $C>0$ such that for any $K>0$,
    \begin{equation*}
        \sum_{k=1}^K \gamma_k \norm{\nabla g_{\Psi_{k}}(x_k)} \leq C \left(\varphi(h_1) - \varphi(h_{K+1}) + \sum_{k=1}^{K}\frac{\gamma_k}{\varphi'(a_k)}\right)
    \end{equation*}
\end{lemma}
\begin{proof}
    Here, we closely follow the arguments of \cite[Proof of Lemma 4.1]{lai2025diameter}.

    First, note that by the previous lemma, the sequence $h_k$ is decreasing. Since $f(x_k),a_k \rightarrow 0$, we have that $g_{\Psi_k}(x_k) \rightarrow 0$ and $h_k \rightarrow 0$. Thus, if there is some $k_0$ such that $h_{k_0} = 0$, then $h_k = \norm{\nabla g_{\Psi_k}(x_k)} = 0$ for all $k \geq k_0$.  It is therefore enough to prove the statement for $K>0$ such that $h(x_{K})>0$.

    Since $\varphi$ restricted to positive reals is concave, by
    Lemma~\ref{lem:kl_ass}, for $k \leq K$,
\begin{equation*}
    \begin{split}
    \varphi(h_{k+1}) - \varphi(h_k) &\leq \varphi'(h_k)(h_{k+1} - h_k) \leq - \gamma_k \varphi'(h_k) \norm{\nabla g_{\Psi_k}(x_k)}^2 = - \gamma_k \frac{\norm{\nabla g_{\Psi_k}(x_k)}^2}{1/\varphi'(h_k)}\,,
    \end{split}
\end{equation*}
Moreover, the function $x\mapsto 1/\varphi'(x) = \theta^{-1}|x|^{1-\theta}$ is sub-additive.
Thus,
\begin{equation*}
    \gamma_k \frac{\norm{\nabla g_{\Psi_k}(x_k)}^2}{1/\varphi'(g_{\Psi_{k}}(x_k)) + 1/\varphi'(a_k)}\leq  \varphi(h_{k}) - \varphi(h_{k+1})
\end{equation*}
Using Corollary~\ref{cor:KL_g}, we obtain for some $C > 0$
\begin{equation*}
    \gamma_k \frac{\norm{\nabla g_{\Psi_k}(x_k)}^2}{C \norm{\nabla g_{\Psi_k}(x_k)} + 1/\varphi'(a_k)}\leq  \varphi(h_{k}) - \varphi(h_{k+1})\,.
\end{equation*}
Thus, we obtain after re-arranging and using $2ab \leq a^2/\varepsilon + \varepsilon b^2$ for any $\varepsilon > 0$
\begin{equation*}
    \begin{split}
      \norm{\nabla g_{\Psi_k}(x_k)} &\leq \sqrt{ \frac{1}{\gamma_k} (C \norm{\nabla g_{\Psi_k}(x_k)} + 1/\varphi'(a_k))(\varphi(h_{k}) - \varphi(h_{k+1}) )}\\
&\leq \frac{1}{2}\left(\varepsilon C\norm{\nabla g_{\Psi_k}(x_k)} + \frac{1}{\varphi'(a_k)}\right) + \frac{1}{2\varepsilon\gamma_k}(\varphi(h_{k}) - \varphi(h_{k+1}) )\,.
    \end{split}
\end{equation*}
Taking sufficiently small $\varepsilon$, above implies that for some $C, C' > 0$
\begin{equation*}
    \gamma_k \norm{\nabla g_{\Psi_k}(x_k)} \leq C \frac{\gamma_k}{\varphi'(a_k)} + C' (\varphi(h_k) - \varphi(h_{k+1}))\, .
\end{equation*}
Summing this inequality from $1$ to $K$ completes the proof.
\end{proof}

Finally, let $\bPsi$ be a good strata selection function  given by Theorem~\ref{thm:algo_produce_good}, iteratively applied to the sequences $(x_k)_{k(i) \leq k \leq k(i+1)-1}$. By Lemma~\ref{lem:payments_varying} it holds that $\sum_{k=1}^{\infty} \norm{\pi_{\Psi_{k+1}}(x_{k+1}) - \pi_{\Psi_k}(x_{k+1})} <  +\infty$. Noting that since  $\gamma_k \leq C/k$, there is $t>0$ such that $a_k \leq k^{-t}$, we obtain by the preceding lemma
     \begin{equation*}
          \sum_{k=1}^K \gamma_k \norm{\nabla g_{\Psi_{k}}(x_k)}  \leq C + \sum_{k=1}^{+\infty}k^{-(1+ t(1- \theta))}< +\infty\, .
    \end{equation*}
By Lemma~\ref{lm:KL_quanti} this completes the proof of Theorem~\ref{thm:KL_conv}.

\section{Remaining proofs}\label{sec:rem_proofs}
In this section we provide proofs for Lemma~\ref{lem:projection} and Theorem~\ref{thm:good_assumption}.
\subsection{Proof of Lemma~\ref{lem:projection}}

For the purpose of the proof, we first need to introduce some concepts related to the properties of the projection operator on a manifold.

In the following, we fix $\cX\subset \bbR^d$ a $C^p$ manifold. As previously, for $y \in \cX$, we will denote $P_{y}$ (respectively $P_{y}^{\perp}$) the projection onto $\cT_{y} \cX$ (respectively $\cN_{y} \cX$).
\paragraph{Projection domain and frontier function.}
Given a $C^p$ submanifold $\cX\subset \bbR^d$, denote $\cE(\cX)\subset \bbR^d$ the maximal open domain such that $\pi_{\cX}$ is $C^{p-1}$ on it. It is standard (see e.g. \cite{federer1959curvature,dudek1994nonlinear}) that such domain exists.

Following \cite{leobacher2021existence}, for $y \in \cX$ and $w \in \cN_{y} \cX$, with $\norm{w} =1$, define the radius of curvature of $\cX$ as
\begin{equation*}
    \varrho_y(w) = \inf \{r\in(0,+\infty) : \forall \tau \in (0,+\infty)\, , B(y, \tau) \cap \cX \cap B(y + r w, r) \neq \emptyset \}\,,
\end{equation*}
where $B(y, r) = \{x \,:\, \|x - y\| < r\}$, and
\begin{equation*}
    \varrho_{y} = \inf\{ \varrho_{y}(w) : w \in \cN_{y} \cX\, , \norm{w} = 1\}\, .
\end{equation*}

\begin{lemma}\label{lm:proj_open_cone}
    For any $x \in \bbR^d$ such that there is $y \in \cX$ satisfying $x-y \in \cN_{y} \cX$ and $\norm{x-y} < \varrho_y$, we have $x \in\cE(\cX)$. In particular,
       \begin{equation*}
\left\{ y+w : y \in \cX\, , w \in \cN_{y}\cX\,, \textrm{ with $\norm{w}< \varrho_y$}\right\} \subset \cE(\cX)\, .
    \end{equation*}
\end{lemma}
\begin{proof}
    Consider $y \in \cX$ and $x \in \bbR^d$ such that $x-y \in \cN_{y} \cX$ and $r = \norm{x-y} < \varrho_{y}$. This means that there is $\tau \in (0, +\infty)$ such that $B(y, \tau) \cap \cX \cap B(x, r) = \emptyset$. Thus, $\dist(x, \cX) = r = \norm{y-x}$ and by \cite[Theorem 3.13]{dudek1994nonlinear} this implies that the segment $[y,x) \in \cE(\cX)$. Since $r< \varrho_{y}$ was arbitrary this completes the proof.
\end{proof}

The following result is a local version of \cite[Lemma 4.17]{federer1959curvature}. For completeness, we provide a proof in Appendix~\ref{sec:technical_manifolds}.

\begin{lemma}[{\cite[Lemma 4.17]{federer1959curvature}}]\label{lm:local_reach_main}
    For any $y \in \cX$,
    \begin{equation}\label{eq:reach_property_main}
        \varrho_y=\inf_{\substack{y' \in \cX \\ y' \neq y}} \frac{\norm{y'-y}^2}{2 \dist(y' - y, \cT_{y}\cX)}\, .
    \end{equation}
\end{lemma}

\paragraph{Proof of Lemma~\ref{lem:projection}. First part: $\pi_{\cX}$ is $C^{p-1}$ on $\cC^*(\cX)$.}
Consider $x \in \cC^{*}(\cX)$, with $\cC^*(\cX)$ defined as in Equation~\eqref{eqdef:verybig_cone}.
There is $y \in \cX$ such that $\norm{y - x} = \dist(x, \cX)$. For such $y$, we can write
    \begin{equation}
        \label{eq:reach0}
    \truncdist(y, \sX_{j-1}) \geq \truncdist(x, \sX_{j-1}) - \norm{x-y} = \truncdist(x, \sX_{j-1}) - \dist(x, \cX) \geq \truncdist(x, \sX_{j-1}) - \mathtt{A}_3 \truncdist(x, \sX_{j-1})\, ,
\end{equation}
where for the last inequality one we used the fact that $x \in \cC^*(\cX)$. From now on assume that $\mathtt{A}_3 < 1$.

Since $y$ minimizes $y \mapsto \norm{x-y}$ on $\cX$, we also have $x-y \in \cN_{y} \cX$ and $x$ lies in the set
\begin{equation}
    \label{eq:reach1}
  \left\{ y+w : y \in \cX\, , w \in \cN_{y}\cX\,, \textrm{ with $\norm{w}< \mathtt{A}_3 (1-\mathtt{A}_3)^{-1} \truncdist(y, \sX_{j-1})$}\right\} \, .
\end{equation}
Indeed, \[\|x - y\| = \dist(x, \cX) \leq \mathtt{A}_3 \truncdist(x, \sX_{j-1}) \leq \mathtt{A}_3(1 - \mathtt{A}_3)^{-1}\truncdist(y, \sX_{j-1})\,,\] where we used Equation~\eqref{eq:reach0}.
By Lemma~\ref{lm:proj_open_cone} the set in Equation~\eqref{eq:reach1} is contained in $\cE(\cX)$ as long as $\mathtt{A}_3(1-\mathtt{A}_3)^{-1}\truncdist(y, \sX_{j-1})\leq \varrho_{y}$, for all $y \in \cX$.

Thus, in view of Lemma \ref{lm:local_reach_main} it suffices to show that there is $C>0$ such that for all $y \in \cX$, $y' \in \cX$, with $y' \neq y$
\begin{equation}\label{eq:reach_assumption}
   C \truncdist(y, \sX_{j-1}) <  \frac{\norm{y'-y}^2}{2 \dist(y' - y, \cT_{y}\cX)} \, .
\end{equation}

To prove Equation~\eqref{eq:reach_assumption} consider $y, y' \in \cX$. By Assumption~\ref{ass:main}-\ref{ass:metric} there is $\scc : [0,1] \rightarrow \cX$ a $C^p$ curve such that $\scc_0 = y$, $\scc_1 = y'$ and $\int_{0}^1 \norm{\dot{\scc}_t} \dif t \leq C_0 \norm{y'-y}$. It holds that
    \begin{equation*}
        y' - y = \int_0^1 \dot{\scc}_t \dif t = \int_{0}^1 \frac{\dot{\scc}_t}{ \norm{\dot{\scc}_t}} \norm{\dot{\scc}_t} \dif t\, .
    \end{equation*}
 Furthermore, since the distance to a vector space is convex and 1-homogeneous, we obtain by Jensen's inequality
        \begin{align}
            \label{eq:reach3}
            \dist(y' - y, \cT_y \cX) &\leq \int_{0}^1  \dist\left(\frac{\dot{\scc}_t}{ \norm{\dot{\scc}_t}}, \cT_y \cX\right)\norm{\dot{\scc}_t} \dif t\,.
        \end{align}
        We focus on bounding the integrand.
    Since $\frac{\dot{\scc}_t}{ \norm{\dot{\scc}_t}}  \in \cT_{\scc(t)} \cX$ is a unitary vector, denoting $\sv_t = P_y \big(\tfrac{\dot{\scc}_t}{ \norm{\dot{\scc}_t}}\big) \in \cT_y \cX$, by Assumption~\ref{ass:main}-\ref{ass:main2}:
    \begin{equation}
        \label{eq:reach4}
      \dist\left(\frac{\dot{\scc}_t}{ \norm{\dot{\scc}_t}}, \cT_y \cX\right)\leq \norm{\frac{\dot{\scc}_t}{ \norm{\dot{\scc}_t}} - \sv_t} \leq \|P_{\scc_t} - P_{y}\|\leq L_1 \frac{\norm{\scc_t - y}}{\truncdist(y, \sX_{j-1})} \leq L_1 C_0\frac{\norm{y' - y}}{\truncdist(y,  \sX_{j-1})}\, ,
    \end{equation}
   where the last inequality comes from:
\begin{align*}
        \|\scc_t - y\| \leq \int_{0}^t \norm{\dot{\scc}_\tau} \dif \tau \leq \int_{0}^1 \| \dot{\scc}_\tau\| \dif \tau \leq C_0 \|y' - y\|\, .
    \end{align*} Equations~\eqref{eq:reach3} and~\eqref{eq:reach4}, in combination with the defining property of $\scc_t$, yield
    \begin{equation*}
        \dist(y' - y, \cT_y \cX)
        \leq L_1C_0 \frac{\norm{y' - y}}{\truncdist(y, \sX_{j-1})} \int_{0}^1 \norm{\dot{\scc}_t} \dif t
        \leq L_1C_0^2 \frac{\norm{y' - y}^2}{\truncdist(y, \sX_{j-1})}\, ,
    \end{equation*}
    which is the desired claim in~\eqref{eq:reach_assumption}. This completes the proof of the first part of Lemma~\ref{lem:projection}.

Before moving to the second part of Lemma~\ref{lem:projection}, let us mention that the relations between the behaviour of projectors and the radius of curvature is absent in the analysis of~\cite{lai2025diameter}, who instead rely on \Loja \,inequalities.

\paragraph{Proof of Lemma~\ref{lem:projection}. Second part: eigenvalues of the Jacobian.}
To prove the claims on the Jacobian we first need to introduce the concept of the shape operator.

For $y \in \cX$, the \emph{second fundamental form} $\sff : \cT_y \cX\times \cT_y \cX \rightarrow \cN_y \cX$, is defined as
\begin{equation*}
    \sff(u,w) = P_{y}^{\perp}(\dif \su(y)[w]) \, , \footnote{With the Riemannian geometry vocabulary: $(u,w)\mapsto \nabla_{w} u := \dif \su(y)[w]$ is the standard Euclidean connection and $(u,w) \mapsto \bar{\nabla}_{w}u := P_x(\nabla_{w} u )$ is the induced connection on $\cX$. Thus, we have $\nabla_{w} u = \bar{\nabla}_{w}u + \sff(u,w)$.}
\end{equation*}
where $\su : \bbR^d \rightarrow \bbR^d$ is any smooth extension of $u$ around $y$, tangent to $\cX$, (that is $\su(x) = u$ and for all $y' \in \cX$ close to $y$, $\su(y') \in \cT_y \cX$). We note (see e.g. \cite[Proposition 4]{leobacher2021existence}) that it is a symmetric bilinear operator.

Given $v \in \cN_x \cX$, the \emph{shape operator} $S_{w,y} : \cT_y \cX \rightarrow \cT_y \cX $ is a self-adjoint linear operator, such that for all $u, v \in \cT_y \cX$,
\begin{equation*}
    \scalarp{\sff(u,w)}{p} = \scalarp{S_{w,y} u}{v}\, .
\end{equation*}

As the following propositions show, the projection, curvature, and the shape operator are intimately related.
\begin{proposition}[{ \cite[Theorem C]{leobacher2021existence} (see also \cite{ambrosio1996level,ambrosio1998curvature})}]\label{prop:jacobian_expression}
    For any $x \in \cE(\cX)$,
    \begin{equation*}
        \Jac \pi_{\cX}(x)  = (\Id_{|\cT_y \cX} - rS_{w,y})^{-1}P_{y}\, ,
    \end{equation*}
    where $y = \pi_{\cX}(x)$, $r = \norm{x-y}$ and $w = (x-y)/r$.
\end{proposition}

\begin{proposition}[{\cite[Proposition 4 and Lemma 5]{leobacher2021existence}}]
    \label{prop:reach_to_eigen_shape}
    Consider $y \in \cX$ and $w \in \cN_{y} \cX$, with $\norm{w} = 1$. Denote $\lambda_1, \ldots, \lambda_j$ eigenvalues of $S_{w,y}$, where $j  = \dim \cX$. Then,
    \begin{equation*}
      \varrho_y(w) = \frac{1}{\max(0,\max_{i}\lambda_i )}\, .
    \end{equation*}
\end{proposition}

To conclude the proof of the second part of Lemma~\ref{lem:projection}, note that in the proof of the first statement we could have chosen $\mathtt{A}_3$ such that if $x \in \cC^{*}(\cX)$ and $y= \pi_{\cX}(x)$, then $\norm{y-x} < \frac{1}{2} \varrho_y$ (in the first part we have only required $< \rho_y$).
For the rest, fix some $x \in \cC^*(\cX)$ and let $y = \pi_{\cX}(x)$, $r = \|x-y\|$, $w = \|x - y\| / r$.

Therefore, with our assumption that $r < \rho_y / 2$, we obtain for $\lambda$, the maximal eigenvalue of $S_{w,y}$,
    \begin{equation*}
        r \lambda \leq \frac{\lambda \varrho_{y}(w)}{2} \leq \frac{1}{2}\,,
    \end{equation*}
    where the last inequality follows from Proposition~\ref{prop:reach_to_eigen_shape}.
Noting that $S_{-w,y} = S_{w,y}$, and repeating the same argument gives a bound on the whole spectrum of $S_{w,y}$.

Finally, using the relation between the Jacobian of projector and shape operator, described in Proposition~\ref{prop:jacobian_expression}, there is indeed $\ulambda, \blambda>0$ such that Equation~\eqref{eqdef:min_max_jac} holds on $\cC^{*}(\cX)$, which completes the proof of the second part of Lemma~\ref{lem:projection}.
\subsection{Lipschitz stratifications and proof of Theorem~\ref{thm:good_assumption}}
Given a function $f:\bbR^d \rightarrow \bbR$ the structure of its subgradient can be deduced from the geometric properties of its graph. For instance, if $f$ is smooth, its gradient at $x$ is the unique vector such that $(\nabla f(x), -1)$ is normal to its graph. Similar characterizations, linking various subgradients (proximal, Frechet, Clarke) to the corresponding normal cones, exist (see e.g. \cite{roc-wets-livre98}).

The main idea behind the various projection formulae established in the literature is to notice that if the graph of $f$ is stratified, then suitable regularity properties of this stratification yield corresponding conditions on the Clarke subgradients of $f$ (see \cite{bolte2007clarke,bianchi2024stochastic,davis2025active,lai2025diameter}).

In the next section we present the concept of a Lipschitz stratification. In Section~\ref{subsec:end_proof_goodass}, using the fact that the graph of a semialgebraic function is Lipschitz stratifiable, we prove Theorem~\ref{thm:good_assumption}.

\subsubsection{Lipschitz stratifications}\label{subsec:lips_strat}
The concept of a Lipschitz stratification was initially defined by Mostowski in \cite{mostowski1985lipschitz} through the concept of chain. Here, we will stick to an equivalent definition due to Parusi\'nski (\cite{parusinski1988lipschitz,parusinski1994lipschitz}), which requires well-behaved extension of vector fields.
Let us start with a definition.

\begin{definition}[$\bbX$-compatible vector fields]
    Consider $\bbX = (\cX_i)$ a $C^p$ stratification of $A \subset \bbR^d$. For $W \subset A$, $\sv : W \rightarrow \bbR^d$ is said to be a \emph{$\bbX$-compatible $C^p$ vector field} if $\sv$ is $C^p$\footnote{We recall that  $\sv : W \rightarrow \bbR^d$ is $C^p$ if for any $x \in W$, there is $U \subset \bbR^d$ a neighborhood of $x$ and $\tilde{\sv}: U \rightarrow \bbR^d$ a $C^p$ function such that $\sv_{|U \cap W} = \tilde{\sv}_{|U \cap W}$. } for any $x \in W \cap \cX$, with $\cX \in \bbX$,
    \begin{equation*}
        \sv(x) \in \cT_{x} \cX\, .
    \end{equation*}
\end{definition}

The definition of a Lipschitz stratification is slightly technical. Recall the definition of $\sX_j$ in Equation~\eqref{eqdef:closed_skeleton}. In the following definition we let $j_{\min}$ denote the minimal index $j$ such that $\sX_{j} \neq \emptyset$.
\begin{definition}[{Lipschitz stratification \cite[Proposition 1.5]{parusinski1988lipschitz}}]
    Given $A \subset \bbR^d$, a stratification $\bbX = (\cX_i)$ of $A$ is said to be \emph{Lipschitz} if there is $C_1>0$ such that the following holds.\\
    For any $j \leq d-1$, any set $W$, satisfying
    \begin{equation*}
    \sX_{j} \subset W \subset \sX_{j+1}\, ,
    \end{equation*} and any $\bbX$-compatible $C^p$ vector field $\sv: W \rightarrow \bbR^d$ such that
    \begin{equation*}
    \forall x,x' \in W\quad \norm{\sv(x) - \sv(x')} \leq L \norm{x-x'}
    \end{equation*}
    and
    \begin{equation*}
    \sup_{x \in \sX_{j_{\min}} \cap W} \norm{\sv(x)} \le C_2\, ,
    \end{equation*} there is an extension of $\sv$ to a $\bbX$-compatible $C^p$ vector field on $A$ with a Lipschitz constant $C_1 (L + C_2)$.
\end{definition}

As we state in the following proposition, any compact set, definable in a polynomially bounded o-minimal structure admits a Lipschitz stratification. Such result was initially proven by Mostowski \cite{mostowski1985lipschitz} for the class of complex analytic sets. Subsequently, it was extended by Parusi\'nski to semianalytic and subanalytic sets in \cite{parusinski1988lipschitz,parusinski1994lipschitz}. Finally, Nguyen and Valette \cite{nguyen2016lipschitz} has extended the construction to arbitrary, polynomially bounded, o-minimal structures (see also \cite{halupczok2018lipschitz})\footnote{Unfortunately, every non polynomially bounded o-minimal structure contains a set that does not admit such stratification (a counter-example due to Parusi\'nski can be found in \cite[Example 2.9]{nguyen2016lipschitz}).}.

Finally, in the already mentioned work of Lai and Song \cite{lai2025diameter} the authors have noticed that in the construction of Nguyen and Valette the strata can be taken such that their inner metric is equivalent to the outer one.

\begin{proposition}[{\cite[Theorem 2.6]{nguyen2016lipschitz} and \cite[Theorem 2.5]{lai2025diameter}}]\label{prop:lipstrat_exist}
    Consider a polynomially bounded o-minimal structure $\cO$. Let $\cK \subset \bbR^d$ be a compact that is definable in $\cO$. There is $\bbX$ a definable in $\cO$, $C^p$ Lipschitz stratification of $\cK$ such that for every $\cX \in \bbX$ its inner and outer metric are equivalent. Moreover, this stratification can be taken compatible\footnote{See Definition~\ref{def:compatibility_strat}.} with any finite number of definable sets $A_1, \ldots, A_k \subset \cK$ and any $\cX \in \bbX$ has its inner metric equivalent to the outer one.
\end{proposition}

In the following, given $\bbX$ a stratification and $x \in \cX \in\bbX$, we denote $P_{x}$ (respectively $P_{x}^{\perp}$) the orthogonal projection onto $\cT_{x} \cX$ (respectively $\cN_{x}\cX$).  For $j\leq d$, denote $\rsX_{j}:=\sX_{j} \backslash \sX_{j-1}$. We will need two consequences of Lipschitz stratifications.

\begin{proposition}[{\cite[Corollary 1.6 and Equation (6)]{parusinski1988lipschitz}}]\label{prop:lip_verdier}
    Let $\bbX$ be a Lipschitz stratification. There is $C>0$ such that the following holds.
    \begin{enumerate}[label=(\roman*)]
    \item For any $x \in \rsX_j$ and $x' \in\rsX_{j'}$, with $j' > j$, it holds that
    \begin{equation*}
        \norm{P_{x'}^{\perp} P_{x}} \leq C \frac{\norm{x - x'}}{\truncdist(x, \sX_{j-1})}\, ,
    \end{equation*}
    \item For any $x,x' \in \rsX_j$,
    \begin{equation}\label{eq:lipsch_tangents}
    \norm{P_{x} - P_{x'}} \leq C \frac{|x - x'|}{\truncdist(x, \sX_{j -1})}\, .
  \end{equation}
    \end{enumerate}
\end{proposition}

\subsubsection{Proof of Theorem~\ref{thm:good_assumption}}\label{subsec:end_proof_goodass}

We have now all the tools to prove Theorem~\ref{thm:good_assumption}. Fix a polynomially bounded o-minimal structure $\cO$. In this section definable will always mean definable in $\cO$.

We will actually prove a slightly more general statement given by the following proposition. We refer to Appendix~\ref{app:o-minimal} for relevant definition of definable and compatible stratifications.
\begin{proposition}
    Let $f: \bbR^d \rightarrow \bbR$ be definable locally Lipschitz, $D: \bbR^d \rightrightarrows \bbR^d$ a definable conservative set-valued field of $f$, and $\cK$ a definable compact. Then, there is $\bbX$ a definable stratification of $\cK$ such that Assumption~\ref{ass:main} holds and where in~\eqref{eq:projformula_bolte_new}, $v\in \partial f(x)$ is replaced by $v\in D(x)$. Furthermore, given any finite collection of definable sets $\bbA = \{A_1, \ldots, A_n\}$, the stratification can be taken compatible with $\bbA$.
\end{proposition}
Taking $\cO$ the o-minimal structure of semialgebraic sets and $D = \partial f$, the proposition implies Theorem~\ref{thm:good_assumption}.
\begin{proof}

In the following proof we will denote $\Pi: \bbR^{d+1} \rightarrow \bbR^d$ the projection onto the first $d$ coordinates.

By \cite[Theorem 2.2]{lewis2021structure} (see also \cite[Lemma 8 and Corollary 9]{bolte2007clarke}), there is $\bbX'$, a stratification of $\cK$, such that for any $\cX' \in \bbX'$, $f_{|\cX'}$ is $C^p$ and for all $x \in \cX$ and $v \in D(x)$, $P_x v = \nabla_{\cX}f(x)$. Such stratification can be taken compatible with $\bbA$.

Let $\bbG = (\cG_i)$ be a Lipschitz stratification of $\Graph f$ compatible with $\{\Graph f_{|\cX'}: \cX' \in \bbX'\}$. By Corollary~\ref{cor:strat_dom_graph} it holds that
\begin{enumerate}
    \item[i)]  for any $\cG \in \bbG$, $\Pi(\cG) = \cX$ is a manifold with $\dim (\cX) = \dim (\cG)$;
    \item[ii)] $f_{|\cX}$ is $C^p$;
\end{enumerate}
We claim that Assumption~\ref{ass:main} holds for $\bbX = (\Pi(\cG))_{\cG \in \bbG}$.

Indeed, Assumption~\ref{ass:main}-\ref{ass:main3} holds by construction. Assumption~\ref{ass:main}-\ref{ass:metric} holds since the metrics are equivalent for $\cG = \Graph f_{|\cX}$ and $f$ is Lipschitz.

To prove Assumption~\ref{ass:main}-\ref{ass:main2} consider $0 \leq j \leq d$ and $x,x' \in \cX \in \bbX_{j}$. Denote $\cG = \Graph f_{|\cX}$. By \eqref{eq:lipsch_tangents} and the fact that $f$ is $G$-Lipschitz on $\cK$, it holds that
\begin{equation}\label{eq:curvature_graph}
        \begin{split}
    \norm{P_{(x,f(x))} - P_{(x', f(x'))}} &\leq \frac{C}{\truncdist((x, f(x)), \sG_{j-1})}\norm{(x,f(x)) - (x', f(x'))}\\
        &\leq \frac{C\sqrt{1 + G^2}}{\truncdist((x, f(x)), \sG_{j-1})}\norm{x - x'}\\
        &\leq \frac{C\sqrt{1 + G^2}}{\truncdist(x, \sX_{j-1})} \norm{x-x'}\, ,
        \end{split}
    \end{equation}
    where for the last inequality we used $\dist(x, \sX_{j-1}) \leq \dist((x, f(x)), \sG_{j-1})$.

Note that $\cT_x$ (respectively $\cT_{x'}$) is the image of $\Pi \circ P_{(x, f(x))}$ (respectively $\Pi \circ P_{(x',f(x'))}$) viewed as a linear operator. Moreover, by Lemma~\ref{lm:graph_tplane}, for any $(h_x, h_f) \in \cT_{(x,f(x))} \cG$, it holds that $h_f = \scalarp{\nabla_{\cX}f(x)}{h_x}$ and
\begin{equation*}
\norm{\Pi (h_x, h_f)} = \norm{h_x} \geq \frac{1}{\sqrt{1 + G^2}} \norm{(h_x, h_f)}\, ,
\end{equation*}
where $G>0$ is the Lipschitz constant of $f$ on $\cK$. Thus, the minimal non-zero singular value of $\Pi \circ P_{(x,f(x))}$ is greater than $1/\sqrt{1 + G^2}$. By Wedin's theorem this implies
\begin{equation*}
    \norm{P_{x} - P_{x'}} \leq \sqrt{1 + G^2} \norm{P_{(x,f(x))} - P_{(x', f(x'))}} \, ,
\end{equation*}
which combined with Equation~\eqref{eq:curvature_graph} establishes Assumption~\ref{ass:main}-\ref{ass:main2}.

Finally, to prove Assumption~\ref{ass:main}-\ref{ass:main4} consider $(y, f(y)) \in \mathring{\sG}_j$, with $y \in \rsX_{j}$ and any $x \in \rsX_{j'}$ with $j' > j$. By Proposition~\ref{prop:lip_verdier}, and repeating similar argument, it holds that
    \begin{equation}\label{eqpf:lip_proj_tanj}
        \norm{P_{(x,f(x))}^{\perp} P_{(y,f(y))}} \leq \frac{C}{\truncdist((y,f(y)),\sG_{j-1} )}\norm{(y,f(y)) - (x,f(x))} \leq  \frac{C \sqrt{1 + G^2} \norm{x-y}}{\truncdist(y,\sX_{j-1} )} \, .
    \end{equation}
Note that (see Lemma~\ref{lm:graph_tplane}),
\begin{equation*}
\cT_{(y,f(y))} = \{(h, \scalarp{\nabla_{\rsX_{j}} f(y)}{h}): h \in \cT_{y} \rsX_{j}\}\,,
\end{equation*}
and since $D(x) \subset \nabla_{\rsX_{j}} f(x) + \cN_{x} \rsX_{j}$,
\begin{equation*}
    \cT_{(x,f(x))} = \{(h, \scalarp{v}{h}): h \in \cT_{x}\rsX_{j'}\}\,,
\end{equation*}
where $v \in D(x)$ is arbitrary. Since $D$ is a conservative field, there is $C>0$ such that for any $v \in D(x)$ and $x \in\cK$, $\norm{x} \leq C$.

Using this fact,
 by simple computations (see e.g. \cite[Proof of Theorem 1]{bianchi2024stochastic}) from Equation~\eqref{eqpf:lip_proj_tanj} we deduce
\begin{equation*}
    \norm{P_y v - \nabla_{\cX}f(y)}\leq  \frac{C(1 + G^2)^{3/2} \norm{x-y}}{\truncdist(y,\sX_{j-1})}\, ,
\end{equation*}
which completes the proof of Assumption~\ref{ass:main}-\ref{ass:main4}.
\end{proof}

\newpage
\appendix

\section{Mathematical Preliminaries}\label{app:prelim}
In this section we gather some useful definitions and properties related to differential geometry, conservative set-valued fields, o-minimal structures and stratifications.

In the following, we fix $p \geq 2$.

\subsection{Submanifolds}\label{app:prel_subm}
We refer to \cite{lee2022manifolds,lafontaine2015introduction,boumal2023intromanifolds} for general references on differential geometry

Given an open set $\cU \subset \bbR^d$, we say that a $C^p$ function $G: \cU \rightarrow \bbR^{d-k}$ is a submersion if its Jacobian (or equivalently its differential) is surjective at every point.

A set $\cX \subset \bbR^d$ is said to be a $C^p$-submanifold of dimension $k$ if for each $x \in \cX$, there is a neighborhood $\cU$ of $x$ and a $C^p$ submersion $G: \cU \rightarrow \bbR^{d -k}$ such that $\cU \cap \cX = G^{-1}(0)$. For $x \in \cM$, the tangent plane of $\cM$ at $x$ is $\cT_x \cX := \ker \Jac G(y)$. The normal plane is $\cN_x \cX  = (\cT_x \cX)^{\perp}$. We note that $\dim \cT_x \cX = k$ and $\dim \cN_x \cX = d-k$.

We say that a function $f: \cX \rightarrow \bbR$ is $C^p$, if $\cX$ is a $C^p$ submanifold and if for every $x \in \cX$, there is a neighborhood $U\subset\bbR^d$ of $x$ and a $C^p$ function $\tilde{f}: U \rightarrow \bbR$ that agrees with $f$ on $\cX \cap U$. Such $\tilde{f}$ is called a smooth extension of $f$ around $x$.

If $f: \cX \rightarrow \bbR$ is $C^p$, we define for every $x\in \cM$,
$$
\nabla_{\cX} f(x) := P_{x} \nabla \tilde{f}(x)\,,
$$
where $P_{x}$ is the orthogonal projection onto $\cT_x\cX$,
and where $\tilde{f}$ is any smooth extension of $f$. We note that $  \nabla_{\cX} f(x) $ does not depend on the choice of this extension
(see e.g. \cite[Section 3.8]{boumal2023intromanifolds}). We refer to $ \nabla_{\cX} f(y)$
as the (Riemannian) gradient of $f$ at $\cX$.

\paragraph{Immersions, submersions and diffeomorphisms.}
Consider two $C^p$ submanifolds $\cX \subset \bbR^d$ and $\cY \subset \bbR^m$ and a smooth function $g: \cX \rightarrow \cY$.
The differential of $g$ at $x \in \cX$, $\dif g(x) : \cT_{x} \cX \rightarrow \cT_{g(x)} \cY$, is the unique linear operator such that for any smooth curve $\c : (-\varepsilon, \varepsilon)$, with $\c(0) = x$,
\begin{equation*}
    \frac{\dif}{\dif t} (g \circ \c)(0) = \dif g(x)[\dot{\c}_0]\, .
\end{equation*}
Such $g$ is said to be in an immersion (respectively submersion) if for every $x \in \cX$, $\dif g(x)$ is injective (respectively surjective). The function $g$ is said to be a diffeomorphism if it admits a smooth inverse $g^{-1} : \cY \rightarrow \cX$. If $g$ is a submersion, then for any $y \in \cY$, $\cM:= g^{-1}(\{y\}) \subset \cX \subset \bbR^d$ is a submanifold of dimension $\dim(\cX) - \dim(\cY)$. Moreover, for any $x \in \cM$, $\cT_{x} \cM = \ker \dif g(x)$.

\subsection{Clarke subgradients and conservative fields}\label{sec:cons_fields}
By Rademacher's theorem a locally Lipschitz function is differentiable almost everywhere. For such functions the Clarke subgradient of $f$ is defined in the following way.

\begin{definition}[Clarke subgradient {\cite{cla-led-ste-wol-livre98}}]
  Let $f: \bbR^d \rightarrow \bbR$ be a locally Lipschitz function. The Clarke subgradient of $f$ at $x$ is defined as
  \begin{equation*}
    \partial f(x) := \conv \{ v \in \bbR^d: \textrm{ there is $x_n \rightarrow x$, with $f$ differentiable at $x_n$ and $\nabla f(x_n) \rightarrow v$}\}\, .
  \end{equation*}
\end{definition}

Another concept describing first-order properties of a function is the one of a conservative set-valued field. It was introduced by Bolte and Pauwels in \cite{bolte2021conservative} as an elegant description of the automatic differentiation operatorautomatic differentiation provided by numerical libraries such as TensorFlow and PyTorch (\cite{tensorflow2015-whitepaper,paszke2017automatic}). They constitute an important tool for establishing the convergence of first-order methods in nonsmooth optimization (\cite{dav-dru-kak-lee-19, bolte2021conservative,bolte2023one,bolte2024differentiating,xiao2023adam, le2023nonsmooth}).

For the following definition, we say that a set-valued map $D: \bbR^d \rightrightarrows \bbR^d$ is graph-closed if $\Graph D$ is closed and is locally bounded if for any compact $K \subset \bbR^d$, there is $C>0$ such that for any $x \in K$ and $v \in D(x)$, $\norm{v} \leq C$

\begin{definition}[\cite{bolte2021conservative}]\label{def:cons_f}
  A graph-closed, locally bounded set-valued map
  $D: \bbR^d \rightrightarrows \bbR^d$ with nonempty values is a \emph{conservative field} for the \emph{potential} $f: \bbR^d \rightarrow \bbR$, if for any absolutely continuous curve $\sx: [0, 1] \rightarrow \bbR^d$, for almost every $t \in [0,1]$,
  \begin{equation*}
    \frac{\dif}{\dif t} f(\sx(t)) = \scalarp{v}{\dot{\sx}(t)} \quad \textrm{ for all $v \in D(\sx(t))$} \, .
  \end{equation*}
  Functions that are potentials of some conservative field are called \emph{path differentiable}.
\end{definition}

For semialgebraic functions the Clarke subgradient is a conservative set-valued field. Moreover, it is in some sense minimal. If $D : \bbR^d \rightrightarrows \bbR^d$ is another conservative set-valued field, then for all $x \in \bbR^d$,
\begin{equation*}
  \partial f(x) \subset \conv D(x)\, .
\end{equation*}

As we will see in the next section the concept of conservative fields is intimately related to the one of stratification.

 \subsection{o-minimal structures}
 \label{app:o-minimal}

 The definition of an o-minimal structure is inspired by properties that are satisfied by semialgebraic sets.
For more details on o-minimal structure we refer to the monographs \cite{cos02,van1998tame,van96}. A nice review of their importance in optimization is \cite{iof08}.

The definition of an o-minimal structure is inspired by properties that are satisfied by semialgebraic sets.
\begin{definition}
  We say that $\cO:=(\cO_n)$, where for each $n \in \bbN$, $\cO_n$ is a collection of sets in $\bbR^n$, is an o-minimal structure if the following holds.
  \begin{enumerate}[label=\roman*)]
    \item If $Q: \bbR^n \rightarrow \bbR$ is a polynomial, then $\{x \in \bbR^n : Q(x) = 0 \} \in \cO_n$.
    \item For each $n \in \bbN$, $\cO_n$ is a boolean algebra: if $A, B \in \cO_n$, then $A \cup B, A \cap B$ and $A^c$ are in $\cO_n$.
    \item If $A \in \cO_n$ and $B \in \cO_m$, then $A \times B \in \cO_{n+m}$.
    \item If $A \in \cO_{n+1}$, then the projection of $A$ onto its first $n$ coordinates is in $\cO_n$.
    \item Every element of $\cO_1$ is exactly a finite union of intervals and points of $\bbR$.
  \end{enumerate}
\end{definition}

Sets contained in $\cO$ are called \emph{definable}. We call a map
$f : \bbR^d \rightarrow \bbR^m$ definable if its graph is definable. Similarly, $D: \bbR^d \rightrightarrows \bbR^d$ is definable if $\Graph D = \{(w,v): v \in D(w) \}$ is definable.
Definable sets and maps have remarkable stability
properties. For instance, if $f$ and $A$ are definable, then $f(A)$
and $f^{-1}(A)$ are definable and definability is stable by most of the common operators such as  $\{+, -, \times, \circ, \circ^{-1}\}$. Let us look at some examples of o-minimal structures.

\textbf{Semialgebraic sets -- $\bbR_{\mathrm{alg}}$.} Semialgebraic sets form an o-minimal structure. This follows from the celebrated result of Tarski \cite{tarski1951decision}. In fact, every o-minimal contains the semialgebraic one.

\textbf{Globally subanalytic sets -- $\bbR_{\mathrm{an}}$.} There is an o-minimal structure that contains, for every $n \in \bbN$, sets of the form $\{ (x,t) : t = f(x)\}$, where $f : [-1, 1]^n \rightarrow \bbR$ is an analytic function. This comes from the fact that subanalytic sets are stable by projection, which was established by Gabrielov \cite{gabrielov1968projections, gabrielov1996complements}. The sets belonging to this structure are called globally subanalytic (see \cite{bier_semi_sub} for more details).

Since definable functions are stable by composition we see that most of neural network architectures, restricted to compact semialgebraic sets, are definable in $\bbR_{\mathrm{an}}$ (see the end of this section).
\begin{definition}
    An o-minimal structure is said to be polynomially bounded if for every definable function $f: \bbR \rightarrow \bbR$ there is $N \in \bbN$ such that $f(t) = O(t^N)$, as $t \rightarrow + \infty$.

    Equivalently, if $f(0) = 0$, then there is a rational number $q$ such that $f(t) = O(|t|^{q})$, as $t \rightarrow 0$.
\end{definition}

The previously introduced structures $\bbR_{\mathrm{alg}}$ and $\bbR_{\mathrm{an}}$ are both polynomially bounded.

One of the most striking properties of definable sets is that they can be stratified. Moreover, the stratification can be taken compatible with any finite family of definable sets.

\begin{definition}[Compatible stratifications]
  \label{def:compatibility_strat}
  Given a finite collection of sets of $\bbR^d$, $\bbA = \{ A_1, \ldots, A_k\}$, we say that a stratification $\bbX$ is compatible with $\bbA$ if for every $\cX \in \bbX$ and $A \in \bbA$ we have either $\cX \cap \cA = \emptyset$, or $\cX \subset A$.
\end{definition}

Given an o-minimal structure, we say that a stratification $\bbX$ is definable if every $\cX \in \bbX$ is definable.
\begin{proposition}[\cite{van1998tame}]
    Consider an o-minimal structure, a finite family  $A_1, \ldots, A_k$ of definable sets and $p\geq 2$. There is a definable $C^p$ stratification of $\bbR^d$, compatible with $A_1, \ldots, A_k$.
\end{proposition}

The domain of a definable function can always be stratified in such way that $f$ restricted to each stratum is $C^p$ (for any $p$). Furthermore, as shown in the seminal work of \cite{bolte2007clarke}, Clarke subgradients of a definable function admit a transparent geometric description: projected along the corresponding stratum they are simply the Riemannian gradient of $f$. The result actually holds for an arbitrary, definable, conservative set-valued field of $f$.

\begin{proposition}{{\cite[Theorem 2.2]{lewis2021structure}}}
  Consider $f: \bbR^d \rightarrow \bbR$, locally Lipschitz and definable. Let $D: \bbR^d \rightrightarrows\bbR^d$ be a definable conservative set-valued field of $f$. Then, there is $\bbX$ a stratification of $\bbR^d$ such that for any $\cX \in\bbX$, $f_{|\cX}$ is $C^p$ and, moreover, for any $x \in \cX$ and any $v \in \partial f(x)$,
    \begin{equation}\label{eq:projformula_consfield}
        P_{x} v = \nabla_{\cX} f(x)\, ,
    \end{equation}
    where $P_{x}$ denotes the orthogonal projection onto $\cT_{x} \cX$.
\end{proposition}

\begin{remark}\label{rmk:cons_fields_varstrat}
  Conversely, if $D: \bbR^d \rightrightarrows \bbR^d$ is a closed-graph, locally bounded set-valued mapping, for which there is a stratification such that~\eqref{eq:projformula_consfield} holds, then $D$ is a conservative set-valued field of $f$.
\end{remark}

\subsection{Deep learning architectures}\label{app:deep_L}

Given parameters $w \in \mathbb{R}^d$, the $L$-layer neural
network defines a function $f_w : \mathbb{R}^{n_x} \to \mathbb{R}$, which is
constructed recursively as
\begin{align*}
    h^0 &= x , \\
    h^\ell &= \mathcal{T}_\ell(h^{\ell-1}; w_\ell), \qquad \ell = 1,\dots,L ,
\end{align*}
with output $f_w(x) = h^L$. The operators \(\mathcal{T}_\ell\) act on finite-dimensional Euclidean spaces and are parametrized by $w_\ell$, with
$w = (w_1,\dots,w_L)$.

Given data $\{(x_i,y_i)\}_{i=1}^N$, define the quadratic loss

\begin{equation}\label{eqdef:NN_quadratic_loss}
        \mathcal{L}(w) = \frac{1}{N} \sum_{i=1}^N \bigl(f_w(x_i) - y_i\bigr)^2.
\end{equation}

As we recall in the following proposition, if the parameter vector $w \in \bbR^d$ is restricted to a compact set $K \subset \bbR^d$, then for usual operators $\cT_{\ell}$, $\cL$ is definable in a polynomially bounded o-minimal structure.

\begin{proposition}[Definability of neural network losses]
Let \(f_w : \mathbb{R}^{n_x} \to \mathbb{R}\) be a neural network defined as above,
where each layer operator \(\mathcal{T}_\ell\) is chosen from the following
classes:
\begin{enumerate}
    \item \emph{Affine or convolutional operators:}
    \[
        \mathcal{T}(h) = A h + b,
    \]
    where \(A\) and \(b\) are learnable parameters of compatible dimensions
    (including the case where \(A\) represents a discrete convolution).

    \item \emph{Elementwise activation operators:}
    \[
        \mathcal{T}(h) = \sigma(h),
    \]
    where \(\sigma\) is applied coordinatewise and is either semialgebraic
    (e.g., ReLU, leaky ReLU) or real-analytic (e.g., sigmoid, \(\tanh\), softplus,
    GELU).

    \item \emph{Pooling operators:}
    \[
        \mathcal{T}(h)_j = \max_{k \in \mathcal{I}_j} h_k
        \quad \text{or} \quad
        \mathcal{T}(h)_j = \frac{1}{|\mathcal{I}_j|} \sum_{k \in \mathcal{I}_j} h_k,
    \]
    where \(\{\mathcal{I}_j\}\) is a fixed collection of index sets.

    \item \emph{Normalization operators:}
    \[
        \mathcal{T}(h) = \gamma \odot \frac{h - \mu(h)}{\sqrt{\sigma^2(h) + \varepsilon}} + \beta,
    \]
    where \(\mu(h)\) and \(\sigma^2(h)\) denote empirical mean and variance computed
    over fixed coordinate groups, and \(\gamma,\beta\) are learnable parameters.

    \item \emph{Residual operators:}
    \[
        \mathcal{T}(h) = h + \mathcal{S}(h),
    \]
    where \(\mathcal{S}\) is itself a finite composition of operators of the
    preceding types.
\end{enumerate}
Then, given any semialgebraic compact set $K \subset \bbR^d$, the function $\cL_{|K}$ is definable in the polynomially bounded o-minimal structure $\bbR_{\mathrm{an}}$.
\end{proposition}

\section{Technical lemmas}\label{sec:technical_manifolds}

Let us record some lemmas on graphs of smooth functions.

\begin{lemma}\label{lm:graph_tplane}
    Let $\cX \subset \bbR^d$ be a $C^p$ submanifold and $f: \cX \rightarrow \bbR$ be $C^p$. Then
    \begin{equation*}
        \Graph f_{|\cX} := \{(x,f(x)) :x \in \cX\}  \subset \bbR^{d+1}
    \end{equation*}
    is a $C^p$ submanifold of dimension equal to $\dim \cX$. Moreover, for any $x \in \cX$,
    \begin{equation*}
        \cT_{(x, f(x))} \, \Graph f_{|\cX} = \{(h, \scalarp{\nabla_{\cX}f(x)}{h}): h \in \cT_{x}\cX\}\, .
    \end{equation*}
\end{lemma}
\begin{proof}
    The function $g:(x,y) \mapsto y-f(x)$ is a submersion from $\cX \times \bbR$ to $\bbR$. Since $\dim(\cX \times \bbR) = \dim(\cX) +1$, $\Graph f_{|\cX} = g^{-1}(0)$ is a manifold of dimension $\dim(\cX)$. Furthermore, for any $x \in \cX$,
    \begin{equation*}
        \cT_{(x,f(x))}\, \Graph f_{|\cX} = \ker \dif g((x,f(x))) =  \{(h_x, h_f) : h_f = \dif f(x)[h_x],\, h_x \in \cT_x \cX\}\, .
    \end{equation*}
    Since, by definition, $\dif f(x)[h_x]= \scalarp{\nabla_{\cX}f(x)}{h_x}$, this completes the proof.
\end{proof}

The following statement is the claim in \cite[page 6]{bolte2007clarke}.

\begin{lemma}
    \label{lem:transversal_proj}
    Let $f: \bbR^d \rightarrow \bbR$ be a function. Assume that $\cG \subset \Graph f$ is such that $\cG \subset \bbR^{d+1}$ is a $C^p$ submanifold and for any $ z \in \cG$, $(0_{\bbR^{d}}, 1)^{\top} \not \in \cT_{z}\cG$. Then, $\cX = \{x: (x, f(x)) \in \cG\}$ is a $C^p$ submanifold such that $f_{|\cX}$ is $C^p$.
\end{lemma}
\begin{proof}

    In this proof we denote $\Pi: \bbR^{d+1} \rightarrow \bbR^d$ the projection onto the first $d$ coordinates. We have $\cX = \Pi(\cG)$. Note that for any $(x,a) \in \bbR^{d}\times \bbR$ and $h \in \bbR^{d+1}$,  $\dif \Pi(x,a)h = \Pi(h)$. Since $(0_{\bbR^{d}},1)^{\top}$ is not a tangent vector of $\cG$, we have that $\Pi_{|\cG}$ is an immersion. Also observe that $\Pi_{|\cG} : \cG \mapsto \cX$ is bijective, since $\cG \subset \Graph f$ and continuous.

    We will show that the inverse of $\Pi_{|\cG}$ is continuous, or, equivalently, it is an open map. To this end, by local immersion theorem for every $z \in \cG$ there is open $V_z \subset \cG$ such that $\Pi_{|\cG}(V_z)$ is open in $\cX$, {is a $C^p$ submanifold}, and that $\Pi_{|V_z} : V_z \mapsto \Pi(V_z)$ is a $C^p$ diffeomorphism. Now let $W \subset \cG$ be open in $\cG$ and observe that $W = \cup_{z \in W} (V_z \cap W)$. As $\Pi_{|V_z}$ is a diffeomorphism, $\Pi(V_z \cap W)$ is open in $\cX$.
        Finally, since
        \begin{align*}
            \Pi_{|\cG}(W) = \bigcup_{z \in W} \Pi_{|\cG}(V_z \cap W)\,,
        \end{align*}
        then $\Pi_{\cG}(W)$ is open in $\cX$. Thus, $\Pi_{|\cG}$ is a homeomorphism. By \cite[Theorem 1.21]{lafontaine2015introduction}, $\cX = \Pi_{|\cG}(\cG)$ is a $C^p$ submanifold.

    Let $y : \bbR^{d+1} \mapsto \bbR$ be the projection onto the last coordinate. Thus, for $x \in \cX$,  $f(x) = (y \circ (\Pi_{|\cG})^{-1})(x)$. Thus, $f_{|\cX} = y_{|\cG} \circ (\Pi_{|\cG})^{-1}$, implying that $f_{|\cX}$ is a $C^p$ function as a composition of $C^p$ functions.
\end{proof}
\begin{corollary}\label{cor:strat_dom_graph}
   Let $f: \bbR^d \rightarrow \bbR$ be a function. Let $\cX' \subset \bbR^d$ be a $C^p$ manifolds such that $f : \cX' \rightarrow \bbR$ is $C^p$. Let $\cG \subset \Graph f_{|\cX'}$ be a $C^p$ submanifold. Then, $\cX = \{x: (x, f(x)) \in \cG\}$ is a $C^p$ submanifold such that $f_{|\cX}$ is $C^p$.
\end{corollary}
\begin{proof}
    Indeed, if $(x, f(x)) \in \cG$, then
    \begin{equation*}
        \cT_{(x,f(x))} \cG \subset \cT_{(x,f(x))} \Graph f_{|\cX'} = \{(h, \scalarp{\nabla_{\cX'} f(x)}{h}): h \in \cT_{x} \cX'\}\, .
    \end{equation*}
    Thus, $(0_{\bbR^d},1)^{\top} \not \in \cT_{(x,f(x))} \cG$ and the proof is concluded by Lemma~\ref{lem:transversal_proj}.
\end{proof}

    \paragraph{Proof of Lemma~\ref{lm:local_reach_main}}
Consider $r < \varrho_y $, $y' \in \cX$, $y' \neq y$, and
   \begin{equation}\label{eq:reach_interm_w}
   w = \frac{P_{y}^{\perp}(y' -y)}{\norm{P_{y}^{\perp}(y' -y)}}\,.
   \end{equation}
Note that since $r < \rho_y$, then by Lemma~\ref{lm:proj_open_cone}, for any $y' \in \cX$
   \begin{align*}
        \|y - (y+rw)\| \leq \|y' - (y+rw)\|\,.
   \end{align*}
  Squaring both sides, we obtain
  \begin{equation}\label{eq:intm1_main}
    \norm{y' - y}^2 + r^2 - 2 r\scalarp{y' -y}{w} \geq r^2 \, .
  \end{equation}
  Observing that
  \begin{equation*}
    \dist(y'-y, \cT_y \cX) = \norm{P_y^{\perp}(y' - y)} = \scalarp{y'-y}{w}\, ,
  \end{equation*}
  and re-arranging the penultimate inequality, yield
    \begin{equation}\label{eq:intm2_main}
    r \leq \frac{\norm{y' -y}^2}{2 \scalarp{y' -y}{w}} = \frac{\norm{y-y'}^2}{2 \dist(y'-y, \cT_y \cX)}\, .
  \end{equation}
  As $r < \varrho_y$ and $y' \in \cX$ were arbitrary, we get
   \begin{equation*}
    \varrho_{y}\leq \inf_{y' \in \cX, y' \neq y}\frac{\norm{y-y'}^2}{2 \dist(y'-y, \cT_y \cX)}\, .
   \end{equation*}

   Conversely, for any $\varepsilon >0$ and $r \in (\varrho_y, \varrho_y + \varepsilon)$, there is $y' \in \cX$
   and a unitary vector $w\in \cN_{y} \cX$ such that $\norm{y' - (y+rw)} < r$. Therefore, as in the previous case
   \begin{equation*}
   \norm{y'-y}^2 < 2 r \scalarp{y'-y}{w} < 2 r \dist(y'-y, \cT_{y} \cX)\, .
   \end{equation*}
   Which implies
   \begin{equation*}
    \frac{\norm{y-y'}^2}{2 \dist(y'-y, \cT_y \cX)} < r < \varrho_{y} + \varepsilon\,
   \end{equation*}
   and completes the proof.

   \section{Proof of Lemma~\ref{lem:geom_all}}
   \label{sec:proof_for_cones}
   This section is entirely devoted to the proof of each item in Lemma~\ref{lem:geom_all}.
    Here, we rely heavily on Definition~\ref{def:rank_stratum}, which
    introduces the rank of strata and of the stratification. Let $R$ denote the rank of $\bbX$ and $(j_r)_{r=0}^{R}$ the corresponding increasing sequence of dimensions present in $\bbX$. We emphasize a basic property of the sequence $(j_r)_{r=0}^R$. For any
    $\cX \in \bbX$ and any $x \in \bbR^d$, one has
    \begin{align*}
        \boxed{
            \dist(x, \sX_{\dim(\cX)-1})
        = \dist(x, \sX_{j_{\rank(\cX)-1}})
        }\, .
    \end{align*}
    It is also convenient to introduce for any $a, b > 0$ the mapping $T_{b}(a) = \min\{a, b\}$ and observe that for any $b > 0$, $a \mapsto T_{b}(a)$ is sub-additive on non-negative real line.

    \subsection*{Proof of Items~\ref{:1} and~\ref{:2}.} It follows immediately from Definitions in Equations~\eqref{eqdef:large_cone}--\eqref{eqdef:thin_cone}.

    \subsection*{Proof of Item~\ref{:3}.}
    The proof goes by induction. For $r = 1$, since $x \notin \smallC_{<j_1}$, we indeed have
    \begin{align}
        \dist(x, \sX_{j_0}) = \min_{\substack{\cX \in \bbX \\ \dim(\cX) = j_0}} \dist(x, \cX) \geq \gamma^\beta \trunc_{1}(\dist(x, \sX_{-1})) = \gamma^{\beta}\,.
    \end{align}

    Assuming that the statement holds up to $r-1$. If $x \notin \smallC_{<j_{r}}$, then
    \begin{align*}
        \dist(x, \sX_{j_{r-1}}) = \min_{\substack{\cX \in \bbX \\ \dim(\cX) = j_{r-1}}} \dist(x, \cX) \geq \gamma^{\beta} \trunc_{\gamma^{(r-1)\beta}}(\dist(x, \sX_{j_{r-2}})) \geq \gamma^{\beta} \trunc_{\gamma^{(r-1)\beta}}(\gamma^{(r-1)\beta}) = \gamma^{r\beta}\,,
    \end{align*}
    which completes the induction.

\subsection*{Proof of Items~\ref{:4} and~\ref{:5}}
    Both claims are proven separately, but similarly. In what follows, $\rank(\cX) = r$.

    \paragraph{Item~\ref{:4}.}
    Since $x \in \smallC(\cX)$, by $1$-lipschitzness of $x\mapsto \dist(x, \sX_{j-1})$ and $a\mapsto \trunc_{\gamma^{r\beta}}(a, \gamma^{r \beta})$, we have
    \begin{align*}
        \dist(x', \cX) \leq \|x - x'\| + \gamma^{\beta}\trunc_{\gamma^{r\beta}}(\dist(x, \sX_{j_{r-1}})) \leq (1 + \gamma^{\beta})\|x - x'\| + \gamma^{\beta}\trunc_{\gamma^{r\beta}}(\dist(x', \sX_{j_{r-1}}))\,.
    \end{align*}
    Since $x \notin \smallC_{<j}$, {denoting $M = \gamma_{0}^{(R+1) \beta - 1}$,} Corollary~\ref{cor:distance_to_boundary_is_large} implies
    \begin{align}\label{eq:interm_reldistance}
        \|x - x'\| \leq G\gamma \leq \frac{G}{M}\gamma^{\beta}\trunc_{\gamma^{r\beta}}(\dist(x, \sX_{j_{r-1}})) \leq \frac{G}{M}\gamma^{\beta}\trunc_{\gamma^{r\beta}}(\dist(x', \sX_{j_{r-1}})) + \frac{G}{M}\gamma^{\beta}\|x - x'\|\,.
    \end{align}
    Hence, as long as ${\tfrac{G}{M}\gamma^{\beta} \leq 1/2}$ and ${G \leq M}$, we have
    \begin{align*}
         \|x - x'\| \leq \gamma^{\beta}\trunc_{\gamma^{r\beta}}(\dist(x', \sX_{j_{r-1}}))\,.
    \end{align*}
    The last three displays in combination with ${\gamma \leq 1}$, yield the first claim.

    \paragraph{Item~\ref{:5}.}
    Since $x \in \bigC(\cX)$, we have
    \begin{align*}
        \dist(x', \cX) \leq \|x - x'\| + \gamma^{\alpha}\trunc_{\gamma^{r\beta}}(\dist(x, \sX_{j_{r-1}})) \leq (1 + \gamma^{\alpha})\|x - x'\| + \gamma^{\alpha}\trunc_{\gamma^{r\beta}}(\dist(x', \sX_{j_{r-1}}))\,.
    \end{align*}
  Since $x \notin \smallC_{<j}$, we have Equation~\eqref{eq:interm_reldistance}, and as previously, as long as ${\tfrac{G}{M}\gamma^{\beta} \leq 1/2}$ and ${G \leq M}$, we have
    \begin{align*}
         \|x - x'\| \leq \gamma^{\beta}\trunc_{\gamma^{r\beta}}(\dist(x', \sX_{j_{r-1}})) \leq \gamma^{\alpha}\trunc_{\gamma^{r\beta}}(\dist(x', \sX_{j_{r-1}}))\,.
    \end{align*}
    The last two displays in combination with $\gamma \leq 1$, yield the second claim.

\subsection*{Proof of Item~\ref{:6}}
    Let $\rank(\cX) = r$.
    By Item~\ref{:3}, if $x \notin \smallC_{ < j_r}$, then
    \begin{align*}
        \dist(x, \sX_{j_{r-1}}) \geq \gamma^{r\beta}\,.
    \end{align*}
    Also, according to Item~\ref{:1}, the condition $x \in \smallC(\cX)$ implies
    \begin{align*}
        \dist(x, \cX) \leq \gamma^{(r + 1) \beta}\,.
    \end{align*}
    We consider two case.
    \paragraph{Case 1.} $\dist(x', \sX_{j_{r-1}}) \geq \gamma^{r\beta} / 2$, then, since $x' \notin \bigC(\cX)$, we must have
    \begin{align*}
        \dist(x', \cX) > \gamma^{\alpha}\trunc_{\gamma^{r\beta}}(\dist(x', \sX_{j_{r-1}})) \geq \gamma^{\alpha}\gamma^{r\beta} / 2\,.
    \end{align*}
    And, as long as ${\gamma^{\alpha} / 4 \geq \gamma^{\beta}}$, we can write
    \begin{align*}
        \|x - x'\| \geq \dist(x', \cX) - \dist(x, \cX) \geq \gamma^{\alpha + r\beta} / 4\,.
    \end{align*}

    \paragraph{Case 2.} If $\dist(x', \sX_{j_{r-1}}) < \gamma^{r\beta} / 2$, then
    \begin{align*}
       \|x - x'\| \geq \dist(x, \sX_{j_{r-1}}) - \dist(x', \sX_{j_{r-1}}) \geq \gamma^{r\beta} / 2 \geq \gamma^{\alpha + r\beta} / 4\,.
    \end{align*}

\section{Avoiding zero-dimensional spurious points}
\label{sec:zaplatki}

The standard convergence result~\cite{dav-dru-kak-lee-19} for subgradient descent applied to semialgebraic functions states that, with decreasing step-sizes, if the iterates $(x_k)_{k\geq1}$ remain bounded, then $\dist(x_k,\cZ) \to 0$, where $\cZ := \{x \in \bbR^d : 0 \in \partial f(x)\}$ is the Clarke critical set.

In contrast, the guarantees both of Theorem~\ref{thm:main_announced} and Lemma~\ref{lem:payments_varying} are expressed in terms of the (pre)-conservative set-valued field $D$ associated with the stratification $\bbX$. The corresponding notion of criticality may differ from Clarke criticality, and in particular, every point belonging to a zero-dimensional stratum $\cX \in \bbX_0$ is a critical point of $D$. Thus, from our analysis it is not clear that the method cannot converge to such spurious points, even when they are not Clarke-critical.

In this section, we show that this phenomenon can be ruled out by a slight modification of our construction. Note that the only two results needed on $g_{\cX} = f \circ \pi_{\cX}$ are Lemmas~\ref{lem:descent_deterministic} and~\ref{lem:distance_different_projection_same_point}. We now show how to modify $g_{\cX}$ in the case where $\cX =\{ x\}$ and $0 \not \in \partial f(x)$.

Let $\cX = \{x\} \in \bbX_0$ be such that $0 \notin \partial f(x)$. By outer semicontinuity of the Clarke subdifferential, there exists $\delta > 0$ such that
\[
0 \notin \partial f_\delta(x) := \conv\{ v \in \bbR^d : v \in \partial f(x'), \ \norm{x'-x} \leq \delta \}\,.
\]
Let $v_{\cX}$ denote the element of minimal norm in the convex set $\partial f_\delta(x)$, so that $\norm{v_{\cX}} > 0$.

Recall that for $\cX = \{x\}$, the sets $\bigC(\cX)$ and $\smallC(\cX)$ are given by
\[
\bigC(\cX) = \{x' \in \bbR^d : \dist(x', \cX) \leq \gamma^\alpha\},
\quad\text{and}\quad
\smallC(\cX) = \{x' \in \bbR^d : \dist(x', \cX) \leq \gamma^\beta\}\,.
\]
Assume that $\gamma_0$ is chosen such that $\gamma_0^\alpha \leq \delta$. We define a modified Lyapunov function $g_{\cX} : \bigC(\cX) \to \bbR$ by
\[
g_{\cX}(x') = f(x) + \langle v_{\cX}, x' - x \rangle.
\]
This function replaces the constant function used in the original construction for zero-dimensional strata and enforces descent away from $x$.

\medskip

We now verify that $g_{\cX}$ satisfies the requirements of the analysis.

\smallskip

\begin{enumerate}
    \item \textbf{Descent property.} For any $x' \in \bigC(\cX)$, we have $\nabla g_{\cX}(x') = v_{\cX}$ and thus $\norm{\nabla g_{\cX}(x')} = \norm{v_{\cX}}$. If for some $k$ we have $[x_k, x_{k+1}] \subset \bigC(\cX)$, then
\[
g_{\cX}(x_{k+1}) - g_{\cX}(x_k) = -\gamma \langle v_{\cX}, v_k \rangle.\,
\]
Since $v_k \in \partial f(x_k)$ and $\norm{x_k - x} \leq \gamma^\alpha \leq \delta$, it follows that $v_k \in \partial f_\delta(x)$. By minimality of $v_{\cX}$ in the convex set $\partial f_\delta(x)$, we obtain
\[
\langle v_{\cX}, v_k \rangle \geq \norm{v_{\cX}}^2\,,
\]
and therefore
\[
g_{\cX}(x_{k+1}) - g_{\cX}(x_k) \leq -\gamma \norm{v_{\cX}}^2 = -\gamma \norm{\nabla g_{\cX}(x_k)}^2\,.
\]
Thus, the descent condition of Lemma~\ref{lem:descent_deterministic} holds.

\item \textbf{Compatibility across strata.}  It remains to verify Lemma~\ref{lem:distance_different_projection_same_point}. Let $x' \in \smallC(\cX) \cap \bigC(\cX')$ for some $\cX' \in \bbX$. Then
\[
\begin{aligned}
|g_{\cX'}(x') - g_{\cX}(x')|
&= \left| f(\pi_{\cX'}(x')) - f(x) - \langle v_{\cX}, x' - x \rangle \right| \\
&\leq C \norm{x' - x} + |f(\pi_{\cX'}(x')) - f(x')| + |f(x') - f(x)| \\
&\leq C \dist(x', \cX) + C \dist(x', \cX')\,,
\end{aligned}
\]
for some constant $C > 0$. Hence, the required compatibility condition holds.
\end{enumerate}

\section{Additional details on Algorithm}\label{app:add_details}

Figure~\ref{fig:four_plots2} depicts a trajectory illustrating the left corner case. The argument follows the same structure as the example presented in the main body, with the exception of Figure~\ref{fig:ddd9}. For this trajectory and the corresponding sub-interval $\sqbracket{\ell}{r}$, there is no guarantee that the trajectory exits the outer neighbourhood of $\cX_1$ before entering the inner one. To handle this, the stratum assignment begins immediately on $\cX_1$ and continues until an exit from the inner neighbourhood is detected. At this point, a right corner case is checked; it is void here, as the sub-trajectory exits the outer neighbourhood.

   \begin{figure}[t]
  \centering

  \begin{subfigure}{0.48\textwidth}
    \centering
    \includegraphics[width=\linewidth]{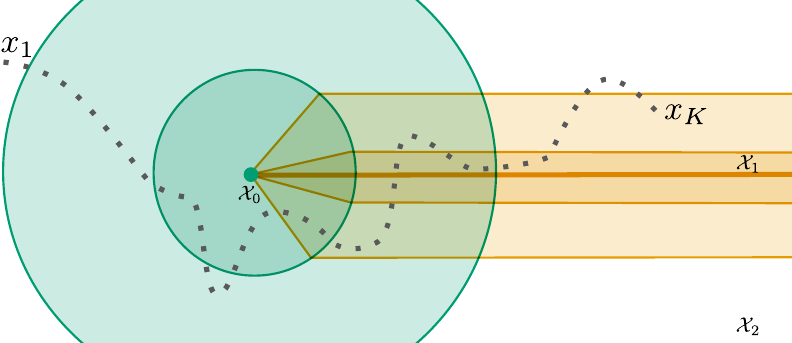}
    \caption{Full gradient descent trajectory}
    \label{fig:ddd7}
  \end{subfigure}\hfill
  \begin{subfigure}{0.48\textwidth}
    \centering
    \includegraphics[width=\linewidth]{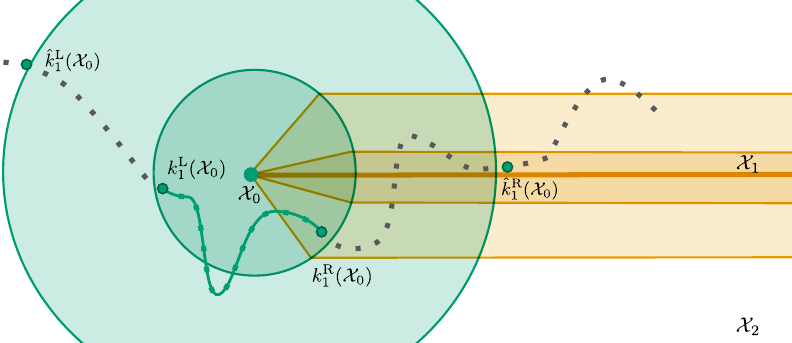}
    \caption{First pass of Algorithm~\ref{algo:main}, green assigned to $\cX_0$.}
    \label{fig:ddd8}
  \end{subfigure}

  \vspace{0.5em}

  \begin{subfigure}{0.48\textwidth}
    \centering
    \includegraphics[width=\linewidth]{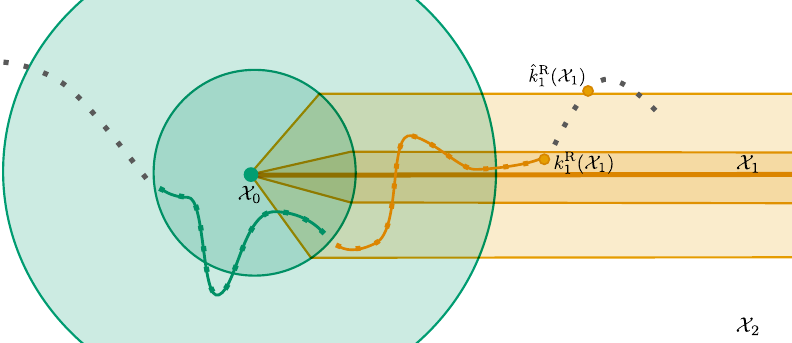}
    \caption{Second pass of Algorithm~\ref{algo:main}, orange assigned to $\cX_1$.}
    \label{fig:ddd9}
  \end{subfigure}\hfill
  \begin{subfigure}{0.48\textwidth}
    \centering
    \includegraphics[width=\linewidth]{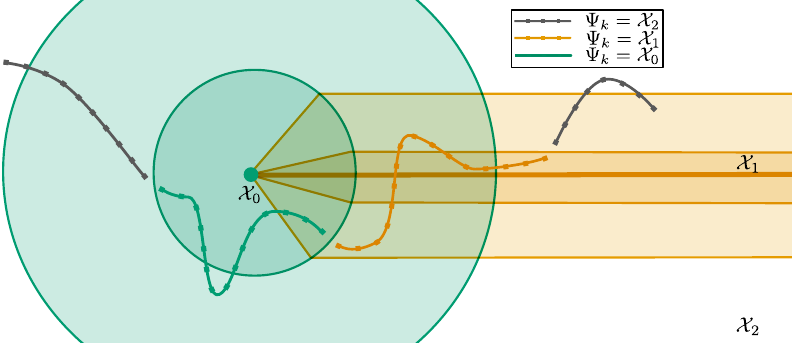}
    \caption{Final pass of Algorithm~\ref{algo:main}, black assigned to $\cX_2$.}
    \label{fig:ddd10}
  \end{subfigure}

  \caption{Step-by-step construction of the strata selection function, which requires \texttt{CheckLeft} sub-routing. Figure~\ref{fig:ddd7} displays the full trajectory, traversed from left to right. Figure~\ref{fig:ddd8} shows the outcome of the first step of the main loop in Algorithm~\ref{algo:main}, together with the corresponding switching times and outer-neighbourhood exit times. Figure~\ref{fig:ddd9} displays the second step of the same loop. Finally, Figure~\ref{fig:ddd10} presents the resulting strata selection function; the gaps between consecutive sub-trajectories correspond to strata switchings.}
  \label{fig:four_plots2}
\end{figure}

\section{Function from Figure~\ref{fig:switching_traj}}
\label{app:switching_traj}
The function is:
\begin{align*}\begin{split}
    f(x, y) &= |y| + \frac{\lambda}{2} (1+\sin(\pi \cdot y)) |x| + \frac{\lambda}{2}(1 -\sin(\pi \cdot y)) |x-c| + \mu x^{2}\\
    &\phantom{=} + B_1 \exp\left(-\frac{(x-c)^2}{\sigma_x^2}- \frac{(y-4.2)^2}{\sigma_y^2} \right) + B_2 \exp\left(-\frac{x^2}{\sigma_x^2}- \frac{(y-3.2)^2}{\sigma_y^2} \right)\, .
\end{split}
\end{align*}
Where  $B_1 = B_2 = 1$, $\sigma_x = 0.02$, $\sigma_y = 0.35$, $\mu = 0.1$, $c=0.5$, $\lambda = 1$.
The subgradient descent is run starting at $(x_1, y_1) = (0.4, 5.5)$ and with a constant step-size $\gamma = 0.01$.
We remark that the above function, restricted to any compact set, is definable in $\bbR_{\mathrm{an}}$, which is a polynomially bounded o-minimal structure. Thus, our analysis is applicable.

\bibliographystyle{apalike}
\bibliography{math}

\end{document}